\documentclass[11pt]{amsart}

\usepackage{amsfonts}
\usepackage{amssymb}
\usepackage{graphicx}
\usepackage{amsmath,mathtools}
\usepackage{latexsym}
\usepackage{amscd}
\usepackage{xypic}
\usepackage{mathrsfs}
\usepackage{enumitem} 
\usepackage{braket}
\usepackage{amsthm}
\usepackage{accents}
\usepackage[protrusion=true,expansion,stretch=5]{microtype}
\usepackage{tikz}
\usetikzlibrary{snakes}

\usepackage[pdftex,
            pdfauthor={Otis Chodosh},
            pdftitle={Large isoperimetric regions in asymptotically hyperbolic manifolds}]{hyperref}

\setlist{itemsep=3pt}

\allowdisplaybreaks

\numberwithin{equation}{section}

\newtheorem{prop}{Proposition}
\newtheorem{theo}[prop]{Theorem}
\newtheorem{lemm}[prop]{Lemma}
\newtheorem{coro}[prop]{Corollary}
\newtheorem{conj}[prop]{Conjecture}

\newtheorem*{claim}{Claim}

\theoremstyle{definition}
\newtheorem{defi}[prop]{Definition}

\newtheorem{rema}[prop]{Remark}

\numberwithin{prop}{section}

\newcommand{\RR}{\mathbb{R}}
\renewcommand{\SS}{\mathbb{S}}

\newcommand{\cA}{\mathcal A}
\newcommand{\cB}{\mathcal B}

\renewcommand{\cH}{\mathcal H}

\newcommand{\cJ}{\mathcal J}

\renewcommand{\cL}{\mathcal L}

\newcommand{\sA}{\mathscr{A}}

\newcommand{\sF}{\mathscr{F}}
\newcommand{\sL}{\mathscr{L}}

\newcommand{\bangle}[1]{\left\langle #1 \right\rangle}

\DeclareMathOperator{\Ric}{Ric}

\let\oldmarginpar\marginpar
\renewcommand\marginpar[1]{\-\oldmarginpar[\raggedleft\footnotesize #1]%
{\raggedright\footnotesize #1}}

\DeclareMathOperator{\Div}{div}

\DeclareMathOperator{\genus}{genus}

\newcommand{\sff}{h}
\newcommand{\tfsff}{\accentset{\circ}{\sff}}

\newcommand{\mm}{\mathbf{m}}

\usepackage{graphicx}

\allowdisplaybreaks

\title[Large isoperimetric regions in AH manifolds]{Large isoperimetric regions in asymptotically hyperbolic manifolds}
\author{Otis Chodosh}
\address{Department of Mathematics, Stanford University, 450 Serra Mall, Bldg
380, Stanford, CA 94305}
\email{ochodosh@math.stanford.edu}

\thanks{I am very grateful to my advisor, Simon Brendle, for suggesting this problem, as well as for his invaluable assistance, support, and many suggestions on a preliminary version of this paper. Additionally, I would like to thank Michael Eichmair and Brian White for patiently answering numerous questions, as well as for their continued encouragement. Finally, I would like to acknowledge Richard Bamler, Justin Corvino, Gerhard Huisken, Alex Volkmann, and Yi Wang for illuminating conversations about topics related to this work. This work was supported in part by an NSF Graduate Research Fellowship DGE-1147470.}

\begin{document}
\begin{abstract}
We show the existence of isoperimetric regions of sufficiently large volumes in general asymptotically hyperbolic three manifolds. Furthermore, we show that large coordinate spheres in compact perturbations of Schwarzschild-anti-deSitter are uniquely isoperimetric. This is relevant in the context of the asymptotically hyperbolic Penrose inequality. 

Our results require that the scalar curvature of the metric satisfies $R_{g}\geq -6$, and we construct an example of a compact perturbation of Schwarzschild-anti-deSitter without $R_{g}\geq -6$ so that large centered coordinate spheres are not isoperimetric. The necessity of scalar curvature bounds is in contrast with the analogous uniqueness result proven by Bray for compact perturbations of Schwarzschild, where no such scalar curvature assumption is required. 

This demonstrates that from the point of view of the isoperimetric problem, mass behaves quite differently in the asymptotically hyperbolic setting compared to the asymptotically flat setting. In particular, in the asymptotically hyperbolic setting, there is an additional quantity, the ``renormalized volume,'' which has a strong effect on the large-scale geometry of volume. 
\end{abstract}
\maketitle

\section{Introduction}
Asymptotically hyperbolic metrics arise naturally in the context of general relativity as spacelike hypersurfaces asymptotic to null infinity in Lorentzian spacetimes that solve Einstein's equations. For asymptotically flat manifolds, the deep relationship between the mass and behavior of large isoperimetric regions is now quite well understood (as we discuss below in \S\ref{subsect:intro-related-works}). However, for asymptotically hyperbolic manifolds, the relationship between the large scale geometry of the metric and the behavior of isoperimetric regions is not well understood. For example, even the existence of large isoperimetric regions is not known (note that there are non-compact manifolds where isoperimetric regions do not exist for any volume, cf.\ \cite{Ritore:surfaces}). Moreover, the only asymptotically hyperbolic metrics in which large isoperimetric regions have been classified are the exact Schwarzschild-anti-deSitter metrics. 

In this work, we show that large isoperimetric regions exist in a very general class of asymptotically hyperbolic manifolds. Furthermore, we show that large coordinate spheres are uniquely isoperimetric for metrics that are compact perturbations of Schwarzschild-anti-deSitter. Compact perturbations of Schwarzschild-anti-deSitter is a natural class of asymptotically hyperbolic metrics; nontrivial examples may be constructed using the gluing result of P.\ Chru\'sciel and E.\ Delay \cite{ChruscielDelay}. In particular, it is a reasonable class of metrics in which to study the conjectured Penrose inequality for asymptotically hyperbolic metrics. As we discuss in the sequel, an important consequence of our work is that the Penrose inequality holds for such metrics, under the assumption that the manifold has connected isoperimetric regions for all volumes.

Our first main result is the existence of large isoperimetric regions in a very general class of metrics (see Definition \ref{defi:AH-metric} below---we emphasize that this definition includes the assumption that the boundary of $M$, if non-empty, has constant mean curvature $H_{g}\equiv 2$ and is the only such compact surface in the manifold).

\begin{theo}\label{theo:exist-iso-AH}
Suppose that $(M,g)$ is an asymptotically hyperbolic manifold with $R_{g}\geq -6$. Then, there is $V_{0}> 0$ sufficiently large so that isoperimetric regions containing volume $V$ exist for $V\geq V_{0}$. 
\end{theo}
Interestingly, the corresponding statement for asymptotically flat manifolds is only partially resolved: M.\ Eichmair and J.\ Metzger have proven \cite[Theorem 1.2]{EichmairMetzger:high-d-iso} that in an arbitrary asymptotically flat manifold with non-negative scalar curvature, there exists a sequence of isoperimetric regions whose volumes which tend to infinity.

It is well known that a geodesic ball in hyperbolic space $(\overline M,\overline g)$ is isoperimetric. Moreover, J.\ Corvino, A.\ Gerek, M.\ Greenberg, and B.\ Krummel have proven in \cite{CovinoGerekGreenbergKrummel} that for $\mm>0$, the centered coordinate balls are the unique isoperimetric regions in Schwarzschild-anti-deSitter of mass $\mm>0$. Their proof is a modification of the technique developed by H.\ Bray in his thesis \cite{Bray:thesis} to prove that centered coordinate balls are uniquely isoperimetric in the Schwarzschild metric. We note that Bray's proof also works for compact perturbations of Schwarzschild showing that sufficiently large coordinate balls are uniquely isoperimetric in such manifolds.

Our second main result concerns uniqueness of large isoperimetric regions in a compact perturbation of Schwarzschild-anti-deSitter (see Definition \ref{def:cpt-pert-SADS} below).

 \begin{theo}\label{theo:main-theo}
 Let $(M,g)$ be a compact perturbation of Schwarzschild-anti-deSitter of mass $\mm>0$ with scalar curvature $R_{g}\geq -6$. Then, sufficiently large centered coordinate spheres are uniquely isoperimetric.
 \end{theo}
 
 Hence, the large isoperimetric regions are completely determined in such $(M,g)$. We refer the reader to \cite[Appendix H]{EichmairMetzger:high-d-iso} for a survey of manifolds in which some or all of the isoperimetric regions are known. See also the survey articles \cite{Osserman:isoSurvey,Ros:iso,Ritore:isoSurvey} for more information concerning the isoperimetric problem.
 
In Theorem \ref{theo:main-thm-is-sharp}, we show that the assumption on the scalar curvature cannot be dropped in Theorem \ref{theo:main-theo}. More precisely, we construct a metric $(M,g)$ that is a compact perturbation of Schwarzschild-anti-deSitter of mass $\mm>0$ (but without $R_{g}\geq-6$ in some parts of the compact region) so that sufficiently large centered coordinate spheres are \emph{not} isoperimetric. This is in sharp contrast to the situation for compact perturbations of Schwarzschild: Bray's proof \cite{Bray:thesis} that large centered coordinate spheres are isoperimetric in a compact perturbation of Schwarzschild does not require non-negativity of scalar curvature.

\begin{rema}
In Definitions \ref{defi:AH-metric} and \ref{def:cpt-pert-SADS}, we have assumed that the horizon (boundary) of $(M,g)$ is connected. However, this is not strictly necessary in our proof of Theorems \ref{theo:exist-iso-AH} and \ref{theo:main-theo}. We have included it because it simplifies considerably the notation and arguments involved in the portions using inverse mean curvature flow with jumps; cf.\ the comment after Proposition \ref{prop:vol-comp-H2}.
\end{rema}

\subsection{The renormalized volume} 
One interesting consequence of Theorem \ref{theo:main-theo} is that for $(M,g)$, a compact perturbation of Schwarzschild-anti-deSitter of mass $\mm>0$ with scalar curvature $R_{g}\geq -6$, the isoperimetric profile $A_{g}(V)$ may be computed for sufficiently large $V$. In particular, as a corollary to Theorem \ref{theo:main-theo}, we see (by inverting the series obtained in Lemma \ref{lemm:vol-large-coord-balls-g}) that
\begin{equation*}
A_{g}(V) = A_{\overline g}(V)- 2 V(M,g) + 8\sqrt{2}\pi^{\frac 32} \mm V^{-\frac 12} + o(V^{-\frac 12}),
\end{equation*}
as $V\to\infty$. Here, as one might expect, the first term is the isoperimetric profile of hyperbolic space, $A_{\overline g}(V)$ and the remaining two terms depend on the geometry of $(M,g)$. As in the asymptotically flat setting (cf.\ \cite[(3)]{EichmairMetzger:3d-iso}), the mass $\mm$ causes the isoperimetric profile to deviate from that of hyperbolic space. However, there are two features of this expansion that differ from the asymptotically flat setting: (1) there is a quantity $V(M,g)$ that makes a stronger contribution (i.e., it is of lower order in the expansion) and (2) the mass term is decaying for large $V$. 

We term the quantity $V(M,g)$ mentioned in (1) the \emph{renormalized volume} (see Definition \ref{defi:renorm-vol}). Because $V(M,g)$ appears before the mass term $\mm$ in the expansion of $A_{g}(V)$, it is natural to conclude the renormalized volume is the most natural notion of ``isoperimetric mass,'' in the sense of G.\ Huisken's work \cite{Huisken:iso-mass}, in the asymptotically hyperbolic setting. With S.\ Brendle, we have recently proven \cite{BrendleChodosh:VolCompAH} a result that can be seen as a Penrose inequality\footnote{We note that \cite{BrendleChodosh:VolCompAH} considers metrics with \emph{minimal} (i.e., $H_{g}\equiv 0$) boundaries, rather than $H_{g}\equiv 2$, but similar arguments work in our setting; see \S\ref{sec:renorm-vol} for an explanation of this point.} for asymptotically hyperbolic manifolds, with the renormalized volume replacing the mass. The fact mentioned in (2), that the mass term is decaying as $V\to \infty$, provides some insight as to why Bray's comparison argument cannot be modified in a simple way to prove Theorem \ref{theo:main-theo} (of course, the fact that $R_{g}\geq -6$ cannot be dropped as shown in Theorem \ref{theo:main-thm-is-sharp}, also implies that such an argument should not work). 
 
 \subsection{The asymptotically hyperbolic Penrose inequality}
 
Theorem \ref{theo:main-theo} is relevant in the context of the conjectured Penrose inequality for asymptotically hyperbolic manifolds. The following conjecture originally appeared (in a more general form) in \cite{Wang:AHmass}; see also the survey articles \cite{BrayChrusciel:Penrose} and \cite{Mars:Penrose}.
\begin{conj}[Asymptotically hyperbolic Penrose inequality]\label{conj:AH-penrose}
Suppose that $(M,g)$ is a compact perturbation of Schwarzschild-anti-deSitter of mass $\mm\geq 0$ and has scalar curvature $R_{g}\geq -6$. If $m_{\partial M} \geq 0$ satisfies $\cH^{2}_{g}(\partial M) \geq \cH^{2}_{\overline g_{m_{\partial M}}}(\partial\overline M_{m_{\partial M}})$, then $\mm \geq  m_{\partial M}$, with equality if and only if $(M,g)$ is isometric to Schwarzschild-AdS of mass $m$.
\end{conj}
 Several partial results have been obtained towards the asymptotically hyperbolic Penrose inequality, as discussed in \S \ref{subsub:penrose}. In particular,  \cite[Proposition 6.3]{CovinoGerekGreenbergKrummel}, J.\ Corvino, A.\ Gerek, M.\ Greenberg and B.\ Krummel have modified isoperimetric profile techniques developed by H.\ Bray in his thesis \cite{Bray:thesis} to prove that compact perturbations of Schwarzschild--anti-deSitter with $R_{g}\geq -6$ satisfy the Penrose inequality provided (a) there exist connected isoperimetric regions for every volume $V>0$ and (b) large coordinate spheres are isoperimetric. Our result above shows that (b) is always satisfied, i.e.,
\begin{coro}
Let $(M,g)$ be a compact perturbation of Schwarzschild-anti-deSitter of mass $\mm\geq 0$, with scalar curvature $R_{g}\geq -6$, and so that there exists a connected isoperimetric region enclosing any volume $V\geq 0$. Then $(M,g)$ satisfies the Penrose inequality as described in Conjecture \ref{conj:AH-penrose}. 
\end{coro}
 
\subsection{Related works}  \label{subsect:intro-related-works} In this section, we briefly discuss previous works on topics related to this paper. 

\subsubsection{CMC hypersurfaces in initial data sets} The study of the relationship between critical points of the isoperimetric problem and initial data sets in general relativity was initiated by  G.\ Huisken and S.-T.\ Yau in \cite{HuiskenYau} when they showed that for certain asymptotically flat metrics with positive mass, there is a foliation of the asymptotic region by CMC spheres which are stable with respect to variations preserving the enclosed volume (see also \cite{Ye:AF-CMC}). Moreover, they proved that any such volume-preserving stable CMC sphere that is sufficiently centered must be a leaf in the foliation. The class of surfaces to which the uniqueness result applies was subsequently extended to all volume-preserving stable CMC surfaces lying outside of a sufficiently large set by J.\ Qing and G.\ Tian \cite{QingTian}. See also \cite{Huang:centerofmass,Huang:CMC,Ma:CMC-RT,Ma:CMC-RT-2} for results along these lines for metrics with more general asymptotics.

M.\ Eichmair and J.\ Metzger have shown in \cite{EichmairMetzger:CMC} that large volume-preserving stable CMC surfaces cannot pass through a compact set of positive mean curvature, and S.\ Brendle and M.\ Eichmair have established \cite{BrendleEichmair:outlying} an intricate relationship between non-negative scalar curvature and the non-existence of outlying volume-preserving stable CMC spheres.   

For metrics which are asymptotic to Schwarzschild-anti-deSitter, R.\ Rigger \cite{Rigger:AHcmc} has shown that such metrics have are foliated near infinity by volume-preserving stable CMC spheres. A.\ Neves and G.\ Tian have shown that the spheres constructed by Rigger are unique, as long as their inner and outer radii are comparable in a certain sense \cite{NevesTian:AHcmc1} (see also \cite{NevesTian:AHcmc2}). R.\ Mazzeo and F.\ Pacard \cite{MazzeoPacard} have proven the existence of CMC foliations for a more general class of metrics.

S.\ Brendle has recently proven a beautiful Alexandrov-type theorem in a wide class of warped product spaces. In particular, a consequence of his result is the following characterization of CMC surfaces in Schwarzschild-anti-deSitter, which we will make use of in the proof of Theorem \ref{theo:main-theo}. 
\begin{theo}[S.\ Brendle {\cite{Brendle:warpedCMC}}]\label{theo:brendle-IHES}
For $\mm>0$, if $\Sigma \hookrightarrow (\overline M_{\mm},\overline g_{\mm})$ is a closed CMC hypersurface in Schwarzschild-anti-deSitter of mass $\mm$, then it is a centered coordinate sphere. 
\end{theo}

We note that S.\ Brendle, P.-K.\ Hung, and M.-T.\ Wang have proven a Minkowski-type inequality in Schwarzschild-anti-deSitter \cite{BrendleHungMTWang:minkowski-type-ineq-adsSchw}, using inverse mean curvature flow in combination with a Heintze--Karcher type inquality from \cite{Brendle:warpedCMC}. 
 
\subsubsection{Isoperimetric regions in initial data sets}

As mentioned above, H.\ Bray showed in his thesis \cite{Bray:thesis} that the coordinate spheres in Schwarzschild and compact perturbations of Schwarzschild are isoperimetric. Using an effective version of Bray's method, M.\ Eichmair and J.\ Metzger have shown that for an asymptotically flat metric that is asymptotically Schwarzschild, large Huisken--Yau spheres are uniquely isoperimetric \cite{EichmairMetzger:3d-iso}. They have also extended their results to all dimensions, showing that an asymptotically flat metric that is asymptotic to Schwarzschild must have a unique foliation near infinity by isoperimetric surfaces \cite{EichmairMetzger:high-d-iso}. An interesting feature of the results just mentioned concerning isoperimetric regions in asymptotically flat manifolds is that they do not require the manifold to have non-negative scalar curvature (this should be compared to Theorem \ref{theo:main-thm-is-sharp}). 

Morover, G.\ Huisken has established \cite{Huisken:iso-mass,Huisken:MM-iso-mass-video} a deep relationship between the mass of an asymptotically flat manifold and its isoperimetric profile. Our argument proving Theorem \ref{theo:main-theo} is inspired in part by Huisken's techniques. We also mention that M.\ Eichmair and S.\ Brendle have characterized the isoperimetric surfaces in the ``doubled Schwarzschild'' metric \cite{BrendleEichmair:doubledSchw}.

\subsubsection{The asymptotically hyperbolic Penrose inequality}\label{subsub:penrose} A.\ Neves has shown in \cite{Neves:IMCFonAH} that the inverse mean curvature flow proof used by G.\ Huisken and T.\ Ilmanen in \cite{HuiskenIlmanen:Penrose} to prove the asymptotically flat Penrose inequality (see also the proof of the asymptotically flat Penrose inequality by H.\ Bray, using a completely different method \cite{Bray:Penrose3d}) breaks down when evaluating the large time limit of the Hawking mass of the flow. 

On the other hand, with S.\ Brendle we have established \cite{BrendleChodosh:VolCompAH} a sharp renormalized volume comparison result for asymptotically hyperbolic manifolds, which is similar to the Penrose inequality, except where ``mass'' is replaced by the quantity ``renormalized volume.'' We note that X.\ Hu, D.\ Ji, and Y.\ Shi \cite{HuJiShi:renorm-vol} have recently proven that an appropriate scalar curvature lower bound implies positivity of the renormalized volume introduced in \cite{BrendleChodosh:VolCompAH} in higher dimensions, within a class of metrics having no boundary and which are a (globally) small perturbation of a model metric.

The asymptotically hyperbolic Penrose inequality has been proven in the special case of manifolds that may be embedded as graphs in hyperbolic space in works by M.\ Dahl, R.\ Gicquaud, and A.\ Sakovich \cite{DahlGicquaudSakovich:AHpenroseGraphs} as well as L.\ de Lima and F.\ Gir\~ao \cite{LimaGirao:AH-penrose}. Moreover, based on an observation of Bray \cite{Bray:thesis} that the Hawking mass is monotone along a foliation of volume-preserving stable CMC spheres, L.\ Ambrozio has recently shown \cite{Ambrozio:penrose} that a Penrose inequality for metrics which are sufficiently small perturbations of Schwarzschild-anti-deSitter. 

Finally, we note that A.\ Neves and D.\ Lee have shown \cite{LeeNeves} that a related class of metrics known as ``asymptotically locally hyperbolic metrics'' satisfy a Penrose inequality as long as the mass is non-positive.

\subsection{Outline of the proof of Theorems \ref{theo:exist-iso-AH} and \ref{theo:main-theo}}
The general strategy for the proof of Theorem \ref{theo:main-theo} is to show that large isoperimetric regions cannot pass through the perturbed region of $(M,g)$, which, in conjunction with Brendle's Alexandrov theorem (cited here as Theorem \ref{theo:brendle-IHES}), will allow us to conclude that if large isoperimetric regions exist, then they must be centered coordinate spheres. From this, it would not be hard to complete the proof: if large isoperimetric regions do not exist, then a minimizing sequence for the isoperimetric problem must split into a region diverging to infinity (so the background metric is approaching hyperbolic space) and a region converging to an isoperimetric region in $(M,g)$. Comparison of volume would allow us to rule this possibility out. The actual proof works somewhat differently, as we will only be able to show that large, connected, genus zero, isoperimetric regions cannot pass through the compact region. Hence, a large portion of the argument is devoted to obtaining sufficient control of large isoperimetric regions that do not have these properties, so as to rule them out.

We now give a somewhat more detailed outline of the argument. As we have just discussed, one crucial step is to show that large (connected, genus zero) isoperimetric regions cannot pass through the perturbed region. For asymptotically flat metrics, as is proven in \cite[Proposition 6.1]{EichmairMetzger:3d-iso}: first, taking the limit of isoperimetric sets passing through the compact region, one can find an area-minimizing boundary.

Next, using a modification of the mechanism discovered by R.\ Schoen and S.-T.\ Yau \cite{SchoenYau:PMT} in their proof of the positive mass theorem, the existence of such a boundary can be ruled out under appropriate assumptions on the scalar curvature, as long as one can understand the behavior of the limit at infinity, cf. \cite[Proposition 6.1(b)]{EichmairMetzger:3d-iso} (see also the related works concerning stable minimal surfaces in asymptotically flat manifolds: \cite{EichmairMetzger:CMC} and subsequently \cite{Carlotto:StabMinSurf}). 

In the asymptotically hyperbolic setting this argument proves difficult. We have been unable to obtain sufficient control of the behavior at infinity of such a limit in the asymptotically hyperbolic setting; a particularly difficult issue is the lack of ability to blow-down the metric in a way analogous to blowing-down an asymptotically flat metric, as well as the fact that such a surface is likely to exhibit exponential extrinsic area growth.

Because of the difficulty with carrying out the aforementioned argument in the asymptotically hyperbolic setting, we deal with the isoperimetric surfaces directly, before taking the limit. A crucial observation is that for a sequence of genus zero, connected isoperimetric regions, whose Hawking mass is uniformly bounded and whose surface area is becoming large, the well known result of D.\ Christodoulou and S.-T.\ Yau \cite{ChristodoulouYau} shows that $R+6+|\tfsff|^{2}$ becomes small in an integral sense. Hence, the limit of such a sequence will be totally geodesic (because $R_{g}\geq -6$), and is thus easily analyzed. 

To obtain uniform Hawking mass bounds on such surfaces, an obvious strategy is to make use of the monotonicity of the Hawking mass along the inverse mean curvature flow, as in G.\ Huisken and T.\ Ilmanen's proof of the Penrose inequality \cite{HuiskenIlmanen:Penrose} (a related strategy has been used by G. Huisken for his isoperimetric mass of asymptotically flat manifolds \cite{Huisken:iso-mass}). This strategy would work in an asymptotically flat manifold, but in our setting, it is not clear that we may bound the limit of the Hawking mass along the flow, by the examples constructed in \cite{Neves:IMCFonAH}. Our argument relies instead on a mechanism discovered by the author and S.\ Brendle \cite{BrendleChodosh:VolCompAH} in which a quantity we term the ``renormalized volume'' is shown to be bounded from below by combining the Geroch monotonicity for the inverse mean curvature flow with the isoperimetric inequality in the exact Schwarzschild-anti-deSitter metric. 

Because this allows us to bound the volume outside of an (outer-minimizing) region, the argument may also be combined with the fact that the renormalized volume is finite to bound the volume contained inside of an isoperimetric region from above; see Proposition \ref{prop:imcf-bds-good-case}. It turns out that this bound is nearly sharp. By comparing the volume contained inside of a large isoperimetric region with that contained inside of a coordinate sphere and expanding both expressions in a series depending on the surface area, the first term which differs contains the Hawking mass of the isoperimetric region and the mass of the background metric. This allows us to establish the desired Hawking mass bounds.

At this point, we are able to prove Theorem \ref{theo:exist-iso-AH}, which asserts the existence of large isoperimetric regions in general asymptotically hyperbolic manifolds in \S\ref{sec:pf-exist-in-AH}. To do so, we use similary bounds on the volume contained inside of a general isoperimetric region, which hold even for disconnected and/or higher genus regions, which follow from a argument similar to what we have just discussed, using inverse mean curvature flow with jumps, cf.\ Proposition \ref{prop:coarse-vol-bds} and Corollary \ref{coro:coarse-bds-general-large-iso}. A crucial step in the existence proof is the relationship between the renormalized volume and the area of the horizon obtained in Proposition \ref{prop:vol-comp-H2}.

In particular, it follows from the above results that in compact perturbations of Schwarzschild-anti-deSitter, large isoperimetric regions exist and if they are connected and genus zero, then they must be centered coordinate spheres. To rule out the other possibilities, e.g., higher genus and/or disconnected large isoperimetric regions, we combine the volume bounds from inverse mean curvature flow with jumps with an argument inspired by H.\ Bray's approach to the asymptotically flat Penrose inequality via the isoperimetric profile \cite{Bray:thesis}. We show (in Section \ref{sec:proof-main-theo}) that if the genus zero case does not occur, then the region must consist primarily of a higher genus component, which would give a bound on the isoperimetric inequality that is too strong to be satisfied for large volumes. This completes the proof of Theorem \ref{theo:main-theo}.

\subsection{Structure of the paper} We define the terms and notation we will use throughout this work in \S\ref{sec:defi} and collect several fundamental properties of isoperimetric regions that we will need later in the paper in \S\ref{sec:fund-iso-prop}. In \S\ref{sec:fund-prop-IMCF}, we recall properties of weak solutions to the inverse mean curvature flow and discuss a mean curvature flow with jumps over obstacles. Then in \S\ref{sec:renorm-vol}, we discuss the renormalized volume as it applies in our setting. In \S\ref{sec:vol-bd-large-iso}, we prove nearly sharp upper bounds on the volume contained inside of isoperimetric regions. This then allows us to prove Theorem \ref{theo:exist-iso-AH} in \S\ref{sec:pf-exist-in-AH}.

In \S\ref{sec:behav-large-iso} we show that large, connected, genus zero isoperimetric regions must leave any compact set in a compact perturbation of Schwarzschild-AdS and additionally study large isoperimetric regions of arbitrary genus and connectivity. Then, in \S\ref{sec:proof-main-theo} we put these properties together with an analysis of the isoperimetric profile to prove Theorem \ref{theo:main-theo}. Finally, \S\ref{sec:main-theo-is-sharp} contains an example showing that $R_{g}\geq -6$ cannot be removed in Theorem \ref{theo:main-theo}. 

In Appendix \ref{sec:vol-coord-balls}, we compute asymptotic expansions for the volume of large coordinate balls and in Appendix \ref{app:comp-bdry-bds}, we prove Proposition \ref{prop:comp-bdry-bds}, which gives an upper bound on the number of connected components in an isoperimetric region.

\section{Definitions and notation}\label{sec:defi} 
In this section we fix several definitions and notation which we will use throughout the work. 
\subsection{Asymptotically hyperbolic metrics} We define the hyperbolic metric
\begin{equation*}
\overline g = \frac{1}{1+s^{2}} ds\otimes ds + s^{2} g_{\SS^{2}}
\end{equation*}
on $\overline M : = \RR^{3}$. Let $\overline D$ denote the connection induced by the hyperbolic metric $\overline g$. In this work, we will be concerned with the following general class of metrics.
\begin{defi}\label{defi:AH-metric}
Suppose that $(M,g)$ is a metric on $\RR^{3}\setminus \overline K$ for some precompact open set $K$. We say that $(M,g)$ is \emph{asymptotically hyperbolic} if, for $k=0,1,2$, we have that 
\begin{equation*}
|\overline D^{k}(g-\overline g)|_{\overline g} = O(s^{-3})
\end{equation*}
as $s \to \infty$. We will also require that $\partial M$ is a connected, outermost,\footnote{Outermost means that there are no compact, constant mean curvature surfaces with $H_{g}\equiv 2$ in the interior of $(M,g)$.} constant mean curvature (CMC) surface with $H_{g}\equiv 2$.
\end{defi}
We will call $\partial K = \partial M$ the \emph{horizon}, and $K$ the \emph{horizon region}. The fundamental example of an asymptotically hyperbolic metric is:
\begin{defi}
We define \emph{Schwarzschild-anti-deSitter} (\emph{Schwarzschild-AdS}) of mass $\mm\geq0$ to be the metric
\begin{equation}\label{eq:schwAdS}
\overline g_{\mm} = \frac{1}{1+s^{2} - 2 \mm s^{-1}} ds \otimes ds + s^{2} g_{\SS^{2}}
\end{equation}
on $\overline M_{\mm} : = \{s \geq 2 \mm\}\times \SS^{2}\subset \RR^{3}$. 
\end{defi}
Schwarzschild-AdS may be realized as a totally umbilic, spacelike hypersurface in the exterior region of the Schwarzschild spacetime\footnote{Recall that the Schwarzschild spacetime is the unique, rotationally symmetric, Ricci flat Lorentzian metric in $3+1$ dimensions.}, which limits to null infinity (cf.\ \cite[\S 3.6]{BrendleWang:gibbons-penrose}). The metric has constant scalar curvature $R_{\overline g_{\mm}} \equiv -6$, and the horizon $\partial \overline M_{\mm}$ has constant mean curvature $H_{g} \equiv 2$ (indeed, it is the unique such rotationally symmetric manifold). Notice that setting $\mm=0$ in Schwarzschild-AdS yields hyperbolic space. 

We will also consider the following sub-class of asymptotically hyperbolic metrics:
\begin{defi}\label{def:cpt-pert-SADS}
Suppose that $K \subset \RR^{3}$ is a precompact open set with smooth boundary 
and let $M = \RR^{3}\setminus \overline K$. If $g$ is a Riemannian metric on $M$, we say that $(M,g)$ is a \emph{compact perturbation of Schwarzschild-AdS} of mass $\mm$ if there is some compact set $\tilde K$ containing $K$, so that $g = \overline g_{\mm}$ in $M \setminus \tilde K$. We will additionally require that $\partial M$ is a connected, outermost, CMC surface with mean curvature $H_{g} \equiv 2$. 
\end{defi}

We will use the following notation for centered coordinate balls: for $A$ large enough, we write $\cB_{\overline g_{m}}(A)$ for the centered coordinate ball in $(\overline M_{m},\overline g_{m})$ of surface area $A$. Regarded as a set in $\RR^{3}$ (using the coordinate system as in \eqref{eq:schwAdS}) we will always regard $\cB_{\overline g_{m}}(A)$ as containing the horizon, i.e., a set of the form $\{s\leq s(m,A)\}$ for some $s(m,A)$. If $(M,g)$ is a compact perturbation of Schwarzschild-AdS, then for $A$ sufficiently large we will still write $\cB_{g}(A)$ for centered $\overline g_{\mm}$-coordinate balls whose boundary lies completely in the unperturbed region and has surface area $A$ with respect to $g$ (or $\overline g_{\mm}$).

Finally, for a hypersurface $\Sigma$ in $\RR^{3}$, we define the \emph{inner radius} of $\Sigma$ by
\begin{equation*}
\underline s(\Sigma) := \inf \left\{ s(x) : x \in \Sigma \right\},
\end{equation*}
where the coordinate $s$ is the one used in the definition of Schwarzschild-AdS in \eqref{eq:schwAdS}.

\subsection{Isoperimetric regions} For $(M,g)$, an asymptotically hyperbolic manifold, we will always extend $g$ inside of the horizon region $K$ to some smooth metric $\hat g$ on all of $\RR^{3}$. We say that a Borel set $\Omega \subset \RR^{3}$ \emph{contains the horizon} if $K \subset \Omega$. For such a set $\Omega$, the \emph{reduced boundary} (cf.\ \cite[\S14]{Simon:GMT}) is denoted by $\partial^{*}\Omega$. It is clear that $\partial^{*}\Omega$ is supported in $M$ and $\cH^{2}_{g}(\partial^{*}\Omega) = \cH^{2}_{\hat g}(\partial^{*}\Omega)$. We will write $\sL^{3}_{g}(\Omega): = \sL^{3}_{\hat g} (\Omega\cap M)$. We define the \emph{isoperimetric profile} of $(M,g)$ by
\begin{equation*}
A_{g}(V) : = \inf\left\{ \ \cH^{2}_{g}(\partial^{*} \Omega ) \ \ :
\begin{array}{c}
 \text{$\Omega$ is a finite perimeter Borel set in $\RR^{3}$}\\
  \text{ containing the horizon with } \sL^{3}_{g}(\Omega) = V 
 \end{array}
 \right\}.
\end{equation*}
We say that $\Omega$, a Borel set of finite perimeter that contains the horizon is \emph{isoperimetric} if $\cH^{2}_{g}(\partial^{*}\Omega) = A_{g}(\sL^{3}_{g}(\Omega))$ and that it is \emph{uniquely isoperimetric} if any other isoperimetric region of the same volume differs only on a set of measure zero. We will occasionally abuse notation and say that $\partial^{*}\Omega$ is (uniquely) isoperimetric if $\Omega$ is. 

\subsection{Hawking mass and volume-preserving stability} Because the boundaries of isoperimetric regions are always embedded and two-sided, we will always require this of closed hypersurfaces under consideration. An important notion for a hypersurface of constant mean curvature (CMC) is:
\begin{defi}\label{def:vp-stab-CMC}
For $\Sigma\hookrightarrow (M,g)$ a CMC hypersurface, we say that $\Sigma$ is \emph{volume-preserving stable} if for all $u\in C^{1}(\Sigma)$ with $\int_{\Sigma} u \, d\cH^{2}_{g} = 0$, it holds that
\begin{equation*}
\int_{\Sigma}\left( \Ric(\nu,\nu) + |\sff|^{2}\right)d\cH^{2}_{g} \leq \int_{\Sigma} |\nabla u|^{2} d\cH^{2}_{g}.
\end{equation*}
\end{defi}
Note that volume-preserving stable CMC surfaces are stable critical points of area under a volume constraint. In particular, isoperimetric regions have volume-preserving stable boundaries. A closely related notion is:
\begin{defi}\label{defi:hawking-mass}
The \emph{Hawking mass} of a surface $\Sigma$ in $(M,g)$ is defined to be
\begin{equation*}
m_{H}(\Sigma,g) = \frac{\cH_{g}^{2}(\Sigma)^{\frac 12}}{(16\pi)^{\frac 32}} \left( 16\pi - \int_{\Sigma} \left(H^{2}_{g} - 4\right) \right).
\end{equation*}
\end{defi}
It is important to note that we have chosen the exact form of $\overline g_{\mm}$ and $m_{H}$ so that the Hawking mass of a centered coordinate sphere is $\mm$, i.e., $m_{H}(\partial \cB_{\overline g}(0;r),\overline g_{\mm}) = \mm$. We will drop the reference to the ambient metric when it is clear from context.

\section{Fundamental properties of isoperimetric regions} \label{sec:fund-iso-prop}
The results in this section will hold for general asymptotically hyperbolic manifolds, without any assumptions on the scalar curvature, unless otherwise noted. 

\begin{prop}\label{prop:reg-iso-surf}
An isoperimetric region $\Omega$ containing the horizon in an asymptotically hyperbolic manifold $(M,g)$ has smooth, compact boundary. If $\partial^{*}\Omega$ intersects the horizon, then they must coincide, i.e., $\partial^{*}\Omega = \partial M$. 
\end{prop}
\begin{proof} We give a proof which is similar to the proof in \cite[Proposition 4.1]{EichmairMetzger:3d-iso} of a similar result in the asymptotically flat setting. The only major change needed is the use of the ``brane functional'' instead of the area functional. Additionally, in the final step of the proof, we use the Hopf boundary point lemma, rather than the weak Harnack inequality; the interested reader may verify that argument used in the end of \cite[Proposition 4.1]{EichmairMetzger:3d-iso} is also applicable. 

Suppose that $\Omega$ is an isoperimetric region containing the horizon in $(M,g)$. The regularity and behavior of $\partial^{*}\Omega$ away from the horizon is well known (see the proof of \cite[Proposition 4.1]{EichmairMetzger:3d-iso} and references therein): in particular, $\partial^{*}\Omega\setminus \partial M$ is smooth, bounded, and has constant mean curvature. Hence, if $\partial^{*}\Omega \cap\partial M=\emptyset$ (or if $\partial^{*}\Omega=\partial M$), then the claim follows. As such it remains to rule out the possibility that $\partial^{*}\Omega\cap \partial M\not=\emptyset$, but they do not coincide. 

First, recall that $\partial^{*}\Omega$ will be a $C^{1,\alpha}$ surface everywhere, including near $\partial M$, cf. \cite[Regularity Theorem 1.3]{HuiskenIlmanen:Penrose}. We claim that the  constant mean curvature $H_{g}$ of $\partial^{*}\Omega\setminus \partial M$ satisfies $H_{g}\geq 2$. If the mean curvature of $\partial^{*}\Omega\setminus \partial M$ satisfies $H_{g}<2$, we may find a (bounded) Borel set of finite perimeter, $\widetilde \Omega$ strictly containing $\Omega$, which minimizes the ``brane functional''
\begin{equation*}
\sF_{\Omega}(\hat\Omega) : = \cH^{2}_{g}(\partial^{*}\hat\Omega) -2 \sL^{3}_{g}(\hat\Omega\setminus\Omega)
\end{equation*}
among finite perimeter Borel sets $\hat \Omega$ containing $\Omega$. Unlike the area functional used in \cite[Proposition 4.1]{EichmairMetzger:3d-iso}, there could potentially be some issue with existence of a minimizer, due to the volume term (which, a priori, could allow for a sequence $\hat \Omega_{j}$ with $\sF_{\Omega}(\hat \Omega_{j})\to-\infty$). 

Hence, to justify this step, we must use a barrier argument (cf.\ \cite[\S 2.3]{ACG:AHpmt}): we define a vector field $X$ in the exterior region as the (outward pointing) unit normal vector field (with respect to $g$) to the foliation $ \{s\} \times \SS^{2}$. Let $\overline X$ denote the unit-normal with respect to $\overline g$. Note that $|X-\overline X|_{\overline g} = |D (X) -\overline D ( \overline X)|_{\overline g} = O(s^{-3})$. Furthermore, because 
\begin{equation*}
\Div_{\overline g}(\overline X) =  2 + 4s^{-2} + O(s^{-3}),
\end{equation*}
we see that if we fix $B$ a large enough (centered) coordinate sphere, then $\Div_{g}(X) \geq 2$ outside of $B$. 

Now, pick $\hat\Omega_{j}$ a minimizing sequence for $\sF_{\Omega}$ and define the truncated regions $\hat\Omega_{j}^{B}:= \hat\Omega_{j}\cap B$. We compute
\begin{align*}
\sF_{\Omega}(\hat\Omega_{j}^{B}) - \sF_{\Omega}(\hat\Omega_{j}) & =\cH^{2}_{g}(\partial B\cap \hat\Omega_{j}) -  \cH_{g}^{2}(\partial^{*}\hat\Omega_{j}\setminus B) + 2\sL^{3}_{g}(\hat\Omega_{j}\setminus B)\\
& \leq \cH^{2}_{g}(\partial B\cap \hat\Omega_{j}) -  \cH_{g}^{2}(\partial^{*}\hat\Omega_{j}\setminus B) + \int_{\hat\Omega_{j}\setminus B} \Div(X) d\sL^{3}_{g}\\
& = \cH^{2}_{g}(\partial B\cap \hat\Omega_{j}) -  \cH_{g}^{2}(\partial^{*}\hat\Omega_{j}\setminus B) \\
& \qquad + \int_{\partial^{*}\hat \Omega_{j}\setminus B}\bangle{X,\nu} d\cH^{2}_{g} - \int_{\partial B\cap \hat \Omega_{j}}\bangle{X,\nu} d\cH^{2}_{g}\\
& \leq \cH^{2}_{g}(\partial B\cap \hat\Omega_{j}) -  \cH_{g}^{2}(\partial^{*}\hat\Omega_{j}\setminus B)\\
& \qquad +  \cH^{2}_{g}(\partial^{*}\hat\Omega_{j}\setminus B) - \cH^{2}_{g}(\partial B\cap\hat\Omega_{j})\\
& = 0.
\end{align*}
Thus, $\sF_{\Omega}(\hat\Omega_{j}^{B})$ is also a minimizing sequence. Given the boundedness of $\hat\Omega_{j}^{B}$, we may take a subsequential limit and obtain a minimizer $\widetilde\Omega$. 

We know that $\partial^{*} \widetilde\Omega$ will be smooth, and of constant mean curvature $H_{g}\equiv 2$ away from $\partial^{*}\Omega$ and $\partial M$. Moreover, it will be a compact $C^{1,\alpha}$ surface everywhere. Hence, if $\partial^{*}\Omega$ is disjoint from both surfaces, then it will be a smooth, compact mean curvature $H_{g}\equiv 2$ surface. This would contradict the outermost property of $\partial M$. 

On the other hand, suppose that $\partial^{*}\widetilde\Omega$ touches $\partial M$. We may find $p \in \partial^{*}\widetilde\Omega\cap \partial M$ and a sufficiently small open ball $B \subset T_{p}\partial M$ so that $0 \in\partial B$, $\partial^{*}\widetilde\Omega$ and $\partial M$ are $C^{2}$-graphical over $B$, $C^{1}$-graphical over a larger ball $\hat B$, strictly containing $B$, and $\partial^{*}\widetilde\Omega$ lies strictly above $\partial M$ on $B$. It is well known that because both surfaces $\partial^{*}\widetilde\Omega$ and $\partial M$ are smooth over $B$ and both graphs have mean curvature $H_{g}\equiv 2$, the difference of the two graphs satisfies a (linear) elliptic second order PDE on $B$, so the Hopf boundary point lemma implies that the normal derivative of the difference is nonzero at $0\in T_{p}\partial M$. However, because both surfaces are $C^{1,\alpha}$ everywhere, and touch at $p$, their tangent planes must agree there. This contradicts the fact that the derivative of the graphs describing the two surfaces must be different at $0$. 

A nearly identical argument shows that $\partial^{*}\widetilde \Omega$ cannot touch $\partial^{*}\Omega$ away from $\partial M$. The only change is that $\partial^{*}\Omega$ has mean curvature $H_{g} <2$, by assumption. Recall that it is impossible for a smooth surface with mean curvature $H_{g} <2$ to touch a smooth surface of mean curvature $H_{g}\equiv 2$ from the inside. For essentially the same reason, the Hopf lemma proof just described works in this setting as well: the zero-order term in the linear PDE for the difference of the graphs will necessarily have the correct sign to apply the Hopf lemma. 

Hence, $\partial^{*}\Omega$ must have mean curvature $H_{g}\geq 2$. Now, we may repeat the Hopf lemma argument yet again to see that $\partial^{*}\Omega$ cannot touch $\partial M$ (unless, of course, they coincide). This shows that the two surfaces must be disjoint unless they coincide, completing the proof. 
\end{proof}

We further recall the standard ``concentration compactness'' picture for isoperimetric regions in non-compact manifolds, as applied to asymptotically hyperbolic manifolds. We will denote by $\cB_{\overline g}(S)$ a ball in hyperbolic space with area $\cH^{2}_{\overline g}(\partial \cB_{\overline g}(S)) = S$. The following proposition says that a minimizing sequence for the isoperimetric problem will either converge to an isoperimetric region, diverge to infinity (where it is more optimal to replace it with a hyperbolic ball) or some combination of the two possibilities. 
\begin{prop}\label{prop:conc-comp}
Fix an asymptotically hyperbolic manifold $(M,g)$. Then, for $V >0$, there exists an isoperimetric region $\Omega$ containing the horizon in $(M,g)$ and some number $S \geq 0$, so that $\sL^{3}_{g}(\Omega) + \sL^{3}_{\overline g}(\cB_{\overline g}(S)) = V$ and $\cH^{2}_{g}(\partial^{*}\Omega) + S = A_{g}(V)$. If $S> 0$ and $\Omega$ is not empty, then $\partial^{*}\Omega$ and $\partial\cB_{\overline g}(S)$ have the same mean curvature. 
\end{prop}
As in the asymptotically flat case (cf.\ \cite[Proposition 4.2]{EichmairMetzger:3d-iso}), this follows readily from the arguments in \cite[Theorem 2.1]{RitoreRosales}. See also \cite{Bray:thesis,DuzaarSteffen} for earlier related results and \cite{Nardulli:existence} for a more recent result along the lines of Proposition \ref{prop:conc-comp}, except for manifolds without boundary (it is clear from the proof that this is not an issue, as any difficulty occurs in the asymptotic regime). We will refer to the union of the regions $\Omega$ and $\cB_{\overline g}(S)$ as a \emph{generalized solution} to the isoperimetric problem. 

\begin{lemm}\label{lemm:iso-prof-strict-increase}
The isoperimetric profile $A_{g}(V)$ of an asymptotically hyperbolic manifold $(M,g)$ is strictly increasing. 
\end{lemm}
\begin{proof}
First, note that $A_{g}(V)$ is absolutely continuous. This is standard (for the isoperimetric profile of a compact manifold, this and more was first proven by \cite{BavardPansu}) as long as the isoperimetric profile is achieved for each volume $V$, i.e., there exist isoperimetric regions of each volume $V \geq 0$. While we do not know that isoperimetric regions of each volume exist in $(M,g)$, the concentration compactness result stated above allows us to find generalized isoperimetric regions of each volume in the disjoint union of $(M,g)$ with hyperbolic space. From this, absolute continuity follows in the exact same way as in \cite{BavardPansu}. See also \cite[Corollary 1]{Nardulli:existence} and \cite[Remark 2.9]{MondinoNardulli}.

Now, suppose that $\Omega$ and $\cB_{\overline g}(S)$ are the generalized solution to the isoperimetric problem for some fixed volume $V\geq 0$. Denote by $H_{V}$ the mean curvature of their boundary. As in the previous paragraph, we may easily generalize from compact case (again, first proven by \cite{BavardPansu}, see also \cite[Theorem 18]{Ros:iso}) to show that $A_{g}(V)$ has left and right derivatives at $V$ (we will write them as $A'_{g}(V)_{-},A'_{g}(V)_{+}$) which satisfy
\begin{equation*}
A'_{g}(V)_{-} \leq H_{V} \leq A'_{g}(V)_{+}.
\end{equation*}
This is a consequence of the first variation formula, see \cite{Bray:thesis,Nardulli:existence} where this is proven in various noncompact settings. Notice that $\partial M$ is an outermost minimal surface of mean curvature $H_{g}=2$, so the boundary of any isoperimetric region in $(M,g)$ must have mean curvature greater than $2$; additionally, any ball in hyperbolic space has mean curvature greater than $2$. Thus, we see that $H_{V}\geq 2$, so $A'_{g}(V)_{+}\geq 2$. Combined with the absolute continuity of $A_{g}(V)$, this implies the claim. 
\end{proof}

\begin{lemm}\label{lemm:iso-are-outer-min}
If $\Omega$ is an isoperimetric region in an asymptotically hyperbolic manifold $(M,g)$, then each component of $\Omega$ is strictly outer-minimizing. 
\end{lemm}
\begin{proof}
Write $\Omega = \Omega_{1}\cup\dots\cup\Omega_{k}$, where each $\Omega_{i}$ is connected and $\Omega_{1}$ is not strictly outer-minimizing. Then, the outer-minimizing hull of $\Omega_{1}$, which we denote by $\Omega_{1}'$, strictly contains $\Omega_{1}$ and has $\cH^{2}_{g}(\partial^{*}\Omega_{1}') \leq \cH^{2}_{g}(\partial^{*}\Omega_{1})$. Notice that
\begin{equation*}\begin{split}
 \cH^{2}_{g} &  (\partial^{*}(\Omega_{1}'\cup \Omega_{2}\cup\dots \cup \Omega_{k}))  \\
& = \cH^{2}_{g}(\partial^{*}\Omega_{1}' \setminus (\Omega_{2}\cup\dots\cup\Omega_{k})) + \cH^{2}_{g}(\partial^{*}(\Omega_{2}\cup\dots\cup\Omega_{k}) \setminus \Omega_{1}')\\
& \leq \cH^{2}_{g}(\partial^{*}\Omega_{1}') + \cH^{2}_{g}(\partial^{*}(\Omega_{2}\cup\dots\cup\Omega_{k}))\\
& \leq \cH^{2}_{g}(\partial^{*}\Omega_{1}) + \cH^{2}_{g}(\partial^{*}(\Omega_{2}\cup\dots\cup\Omega_{k}))\\
& = \cH^{2}(\partial^{*}\Omega).
\end{split}\end{equation*}
However,
\begin{equation*}
\sL^{3}_{g}(\Omega_{1}'\cup \Omega_{2}\cup\dots \cup \Omega_{k}) > \sL^{3}_{g}(\Omega).
\end{equation*}
This contradicts Lemma \ref{lemm:iso-prof-strict-increase}. 
\end{proof}
\begin{lemm}\label{lemm:iso-connected-bdry}
Each connected component of an isoperimetric region in an asymptotically hyperbolic manifold $(M,g)$ has a connected boundary.
\end{lemm}
\begin{proof}
Suppose that $\Omega$ is an isoperimetric region. If some component of $\Omega$ had a disconnected boundary, then at least one of the boundary components must bound a compact region in $M\setminus \Omega$. Adding this region to $\Omega$ increases volume and decreases area, contradicting Lemma \ref{lemm:iso-prof-strict-increase}.
\end{proof}
We will make use of the following celebrated result of Christodoulou--Yau concerning the Hawking mass (see Definition \ref{defi:hawking-mass}) of volume-preserving stable CMC surfaces (see Definition \ref{def:vp-stab-CMC}).
\begin{prop}[{\cite{ChristodoulouYau}}]\label{prop:christodoulou-yau}
Fix $\Sigma$, a connected, volume-preserving stable CMC surface in a manifold $(M,g)$. If $\genus(\Sigma) = 0$, then\footnote{We note that this inequality actually holds for all $\Sigma$ with even genus, by using Christodoulou--Yau's proof in combination with improved bounds on the degree of meromorphic functions on algebraic curves, cf.\ \cite[p.\ 261]{GriffithsHarris} or \cite{Yang:BrillNoetherSurvey}. We will not make use of this fact, as we would still lack desired control of odd genus regions and the argument we use to control odd genus regions (see \S \ref{sec:proof-main-theo}) applies equally well to rule out large regions with non-zero genus.}
\begin{equation*}
\int_{\Sigma} \left( R_{g} + 6 + |\tfsff|^{2} \right) d\cH^{2}_{g} \leq \frac 3 2 \cH^{2}_{g}(\Sigma)^{-\frac 12}(16\pi)^{\frac 32} m_{H}(\Sigma).
\end{equation*}
Without the genus zero assumption, we have the bound
\begin{equation*}
\int_{\Sigma} \left( R_{g} + 6 + |\tfsff|^{2} \right) d\cH^{2}_{g} \leq \frac 3 2 \cH^{2}_{g}(\Sigma)^{-\frac 12} (16\pi)^{\frac 32}m_{H}(\Sigma) + 8\pi.
\end{equation*}
Equivalently, we have the inequality
\begin{equation*}
\frac 2 3 \int_{\Sigma} \left( R_{g} + 6 + |\tfsff|^{2}\right) d\cH^{2}_{g} + \int_{\Sigma}(H^{2}-4) d\cH^{2}_{g} \leq \frac{64\pi}{3},
\end{equation*}
valid for any connected, volume-preserving stable CMC surface in $(M,g)$. 
\end{prop}
The work \cite{ChristodoulouYau} is concerned with the setting when $R_{g}\geq 0$, but it is well known that to compensate for the fact that $R_{g}\geq -6$ one must modify the Hawking mass by changing the $\int H^{2}$ term to $\int (H^{2}-4)$ as we have done above. Granted this change, the proof of these inequalities proceeds in an identical manner to \cite{ChristodoulouYau}. 
\begin{coro}\label{coro:CY-mH-bd}
Assume that the manifold $(M,g)$ satisfies the scalar curvature bound $R_{g}\geq -6$. Suppose that $\Sigma$ is a connected, volume-preserving stable CMC surface in $(M,g)$. If $\Sigma$ has genus zero, then $m_{H}(\Sigma) \geq 0$. In general, $m_{H}(\Sigma) \geq -\frac 1 3 (16\pi)^{-\frac 12} \cH^{2}_{g}(\Sigma)^{\frac 12}$.  
\end{coro}

Later, it will be important to know that there are no isoperimetric regions with arbitrarily many connected components. 

\begin{prop}\label{prop:comp-bdry-bds}
For an asymptotically hyperbolic manifold $(M,g)$ with $R_{g}\geq -6$, the number of boundary components of an isoperimetric region is bounded by some constant $n_{0}$ depending only on $(M,g)$. 
\end{prop}
This follows from an adaptation of the argument in \cite[\S 5]{EichmairMetzger:CMC}. Because several of the arguments must be modified, we give the proof in Appendix \ref{app:comp-bdry-bds}. 

Now, applying Propositions \ref{prop:christodoulou-yau} and \ref{prop:comp-bdry-bds} to each component individually we obtain the following corollary, which we will later use to show that large isoperimetric regions (which are not necessarily connected) have mean curvature very close to $2$. 
\begin{coro}\label{coro:coarse-CY-mult-bdry}
If $(M,g)$ is asymptotically hyperbolic and $R_{g}\geq-6$, then for an isoperimetric region, $\Omega$, defining $\Sigma: = \partial^{*}\Omega$ we have that $(H^{2}_{g}-4) \cH^{2}_{g}(\Sigma) \leq \frac{64\pi}{3} n_{0}$.
\end{coro}

Finally, we have a convenient compactness property of isoperimetric regions in asymptotically hyperbolic manifolds. To state the result, we will say that a set $\Omega$ is \emph{locally isoperimetric} if for any Borel set of locally finite perimeter, $\widetilde\Omega$, such that $(\widetilde\Omega \setminus \Omega)\cup(\Omega\setminus\widetilde\Omega)$ is contained in a compact set $R$ and which has has zero relative volume with $\Omega$, i.e.
\begin{equation*}
\cL^{3}_{g}(\widetilde\Omega \setminus \Omega)=\cL^{3}_{g}(\Omega\setminus\widetilde\Omega),
\end{equation*}
we have that
\begin{equation*}
\cH^{2}_{g}(\partial^{*}\widetilde\Omega\cap R) \geq \cH^{2}_{g}(\partial^{*}\Omega\cap R).
\end{equation*}
\begin{prop}\label{prop:cptness-iso-regions}
Suppose that $(M,g)$ is asymptotically hyperbolic and $\Omega^{(l)}$ is a sequence of isoperimetric regions in $(M,g)$ where $\partial^{*}\Omega^{(l)}$ has constant mean curvature satisfying $H_{g}\to 2$ as $l\to\infty$. After extracting a subsequence, we may write $\Omega^{(l)}$ as the disjoint union of open sets $\Omega^{(l)}=\Omega^{(l)}_{h}\cup\Omega^{(l)}_{c}\cup\Omega^{(l)}_{d}$ and find a locally isoperimetric region $\Omega$ whose boundary is a properly embedded hypersurface with constant mean curvature $H_{g}\equiv 2$ so that 
\begin{itemize}
\item $\Omega_{h}^{(l)}$ converges to the horizon region, which is contained in $\Omega$
\item $\Omega_{c}^{(l)}$ converges to the other components of $\Omega$, and
\item $\Omega_{d}^{(l)}$ diverges, i.e., it is eventually disjoint from any compact set.
\end{itemize}
Here, the convergence statements are all in the sense of local convergence of sets of finite perimeter (i.e., in the BV sense) as well local smooth convergence of the boundary surfaces. In particular, the only compact component of $\Omega$ is the horizon region. Furthermore, $\sL^{3}_{g}(\Omega^{(l)}_{h}) = o(1)$ and $\cH^{2}_{g}(\partial^{*}\Omega^{(l)}_{h})= \cH^{2}_{g}(\partial M) + o(1)$ as $l\to\infty$. 
\end{prop}
\begin{proof}
Standard BV compactness results (cf. \cite[Theorem 6.3]{Simon:GMT}) guarantee that we may extract a subsequence of $\Omega^{(l)}$ which converges locally as sets of finite perimeter to $\Omega$, a locally isoperimetric region in $(M,g)$. By Proposition \ref{prop:comp-bdry-bds}, we may choose a further subsequence so that $\Omega^{(l)}$ has a fixed number of components. Each component will either converge locally as a set of finite perimeter to some component of $\Omega$ (and thus should be labeled as a member of $\Omega^{(l)}_{h}$ or $\Omega^{(l)}_{c}$ depending on whether or not the component is converging to the horizon region or not). 

For any other regions, we claim that they must be diverging, rather than shrinking away. If a component is shrinking away, then the monotonicity formula shows that it will have a definite amount of area while containing a tiny amount of volume. This cannot happen: flowing one of the other components outwards by constant speed allows us to find a region with less area and the same volume. Hence, any other region must diverge. Note that this argument works as long as we are not in the following case: $(M,g)$ has no horizon and $\sL^{3}_{g}(\Omega^{(l)})\to 0$ (in this case, there might be no other component to flow outwards with unit speed). By assumption, this case cannot occur: the monotonicity formula would imply that $\partial^{*}\Omega^{(l)}$ has constant mean curvature $H_{g}\to\infty$. 

By the blowup argument in \cite[Proposition 5]{Ros:iso}, we may extract a further subsequence so that the convergence occurs in the sense of local smooth convergence. Thus, $\Omega$ has constant mean curvature $H_{g}\equiv 2$, so by the outermost assumption for $\partial M$, the only compact component of $\Omega$ will be the horizon region. That $\partial^{*}\Omega$ is properly embedded follows from the ``cut and paste'' argument used in the proof of \cite[Theorem 18]{Ros:iso}. 

Finally, the convergence of the volume and area of $\Omega^{(l)}_{h}$ follows from the smooth convergence to the horizon.
\end{proof}

\section{Fundamental properties of the inverse mean curvature flow} \label{sec:fund-prop-IMCF}
Our fundamental tool for studying isoperimetric regions in an asymptotically hyperbolic manifold $(M,g)$ will be the inverse mean curvature flow. In particular, we will use the weak formulation of the inverse mean curvature flow developed in the foundational work by G. Huisken and T.\ Ilmanen in \cite{HuiskenIlmanen:Penrose}. 

While Huisken--Ilmanen developed the weak inverse mean curvature flow to prove the Penrose inequality for asymptotically flat manifolds, they conveniently established existence and other properties of the flow in much greater generality. Below, we have stated only the properties of the weak flow that we will make use of in the sequel and included references to the relevant sections in \cite{HuiskenIlmanen:Penrose} for the reader's convenience. We recall\footnote{The reason that we have stated the theorem in this way is that Huisken--Ilmanen have defined the flow only on manifolds without boundary. However, as in \cite{HuiskenIlmanen:Penrose}, we will never allow the flow to pass through $\RR^{3}\setminus M$, and will instead ``jump'' over this region in an appropriate manner if necessary.} that we have extended the metric $g$ inwards to $\hat g$, defined on all of $\RR^{3}$.

\begin{theo}\label{theo:HI-wIMCF}
Suppose that $\Omega$ is a connected, compact region in $\RR^{3}$ with smooth, connected boundary $\Sigma = \partial\Omega$ which is is contained entirely in $M$. We will always assume that $\Sigma$ is outer-minimizing with respect to $g$. Then, by \cite[Theorem 3.1]{HuiskenIlmanen:Penrose} there exists a proper, locally Lipschitz function $u \geq 0$ on $\RR^{3} \setminus \Omega$ with the following properties:
\begin{enumerate}
\item Initial conditions, \cite[Property 1.4(iv)]{HuiskenIlmanen:Penrose}\textup{:}  $\{u=0\} = \Sigma$.
\item  Gradient bounds, \cite[Theorem 3.1]{HuiskenIlmanen:Penrose}\textup{:} 
We have the gradient bound 
\begin{equation*}
|\nabla u(x)| \leq  \max\left\{0, \max_{p\in\Sigma} H_{g}(p)\right\} + C
\end{equation*}
for a.e.\ $x \in M \setminus \Omega$. Here, $C=C(M,g)$ is a constant which only depends on $(M,g)$ but not on $\Omega$. 
\item Regularity, \cite[Theorem 1.3]{HuiskenIlmanen:Penrose}\textup{:} The regions $\Sigma_{t} := \partial \{u > t\}$ form a increasing family of $C^{1,\alpha}$ surfaces. 
\item Minimizing hull property, \cite[Property 1.4]{HuiskenIlmanen:Penrose}\textup{:} For $t\geq0$, $\Sigma_{t}$ strictly minimizes area among homologous surfaces in $\{u \geq t\}$. 
\item Weak mean curvature, \cite[(1.12)]{HuiskenIlmanen:Penrose}\textup{:}  For a.e.\ $t>0$ and a.e.\ $x \in \Sigma_{t}$, the weak mean curvature of $\Sigma_{t}$ is defined, equal to $|\nabla u|$, and strictly positive.
\item Exponential area growth, \cite[Lemma 1.6]{HuiskenIlmanen:Penrose}\textup{:} We have $\cH^{2}_{g}(\Sigma_{t}) = e^{t} \cH^{2}_{g}(\Sigma)$ for $t \geq 0$. 
\item Connectedness, \cite[Lemma 4.2]{HuiskenIlmanen:Penrose}\textup{:} The surfaces $\Sigma_{t}$ remain connected for $t\geq 0$.
\item Geroch monotonicity,  and \cite[\S 5]{HuiskenIlmanen:Penrose}\textup{:} The Hawking mass $m_{H}(\Sigma_{t})$ is monotone non-decreasing for $t \geq 0$ as long as $\Sigma_{t}$ does not cross through the horizon \textup{(}recall that we have assumed that $R_{g}\geq -6$\textup{)}. 
\item Equality in Geroch monotonicity, \cite[\S 5]{HuiskenIlmanen:Penrose}\textup{:} Assuming the flow avoids the horizon in the time interval $(t,s)$, then we have that $m_{H}(\Sigma_{t}) = m_{H}(\Sigma_{s})$ if and only if the interior of $\{ t < u \leq s \}$ is isometric to an annulus in Schwarzschild-AdS of mass $m = m_{H}(\Sigma_{t})$. 
\item Avoidance principle, \cite[Theorem 2.2(ii)]{HuiskenIlmanen:Penrose}\textup{:} If $\widetilde \Omega\subseteq \Omega$ also satisfies the hypothesis above, then the weak inverse mean curvature flow starting at $\partial\widetilde\Omega$, $\widetilde \Sigma_{t}$, remains inside of $\Sigma_{t}$ for all $t$, as long as $\widetilde\Sigma_{t}$ continues to bound a compact region.
\end{enumerate}
We will say that $\Sigma_{t}$ is the solution to the weak inverse mean curvature flow starting at $\Sigma$. 
\end{theo}

We note that in order to apply \cite[Theorem 3.1]{HuiskenIlmanen:Penrose} to obtain existence, one must find a subsolution in the asymptotic region. This is achieved by considering large coordinate balls flowing slightly slower than inverse mean curvature flow would dictate; see Proposition \ref{prop:imcf-plus-isoper-ineq}. 

We will also need to define a weak inverse mean curvature flow with jumps. In \cite[\S 6]{HuiskenIlmanen:Penrose}, Huisken--Ilmanen devised a method for jumping over regions whose boundaries are minimal (and outer-minimizing) and showed that the Hawking mass was still monotone along the flow with jumps. Here, we slightly modify Huisken--Ilmanen's definition of weak mean curvature flow with jumps (namely, we jump at the earliest possible time) and observe that the Hawking mass fails to be monotone over the jumps in a controllable way (cf.\ \cite[p.\ 412]{HuiskenIlmanen:Penrose} for a discussion concerning the freedom to choose the jump time). We remark that in order to jump over multiple components, one could apply the following proposition multiple times, restarting the flow between jumps.

\begin{prop}\label{prop:how-to-jump}
We assume that $(M,g)$ is a compact perturbation of Schwarzschild-AdS with $R_{g}\geq -6$. Recall that we have extended $g$ to a metric on all of $\RR^{3}$. Fix some $\delta > 0$ and suppose that $\Omega$, $\cJ$, $\Gamma$ are compact regions in $\RR^{3}$ so that 
\begin{enumerate}
\item Both surfaces $\partial\Omega$ and $\partial\cJ$ are smooth and contained entirely in $\RR^{3}\setminus K = M\cup\partial M$.
\item The region $\Omega\cup\cJ$ contains the horizon.
\item The surfaces $\partial\Omega$, $\partial\cJ$ and $\partial\Omega\cup\partial\cJ$ are all outer-minimizing in $(M,g)$.
\item The surfaces $\partial \Omega$ and $\partial\cJ$ are connected. 
\item We have that $\cH^{2}_{g}(\partial\Omega) \geq 1$.
\item At each point in $\partial\cJ$, we have that $\partial\cJ$ has mean curvature $H_{g} \geq 2$ . 
\item The surfaces $\partial\Omega$ and $\partial\cJ$ both have nonempty intersection with $\Gamma$. 
\item The regions $\Omega$ and $\cJ$ are disjoint.
\end{enumerate}
Then, under these assumptions we can construct a weak inverse mean curvature flow starting at $\partial \Omega$ which ``jumps over $\cJ$.'' 

More precisely, we can find a time $T > 0$ and an increasing family of connected, closed, $C^{1,\alpha}$ surfaces $\Sigma_{t}$ for $t \in [0,\infty)\setminus\{T\}$, so that:
\begin{enumerate}
\item On the intervals $[0,T)$ and $(T,\infty)$, $\Sigma_{t}$ is a solution to the weak inverse mean curvature flow, which is always disjoint from the interior of $\cJ$.
\item At time $t=0$ we have that $\Sigma_{0} = \partial \Omega$.
\item Denote the surface obtained by flowing $\partial\Omega$ for time $T$ by $\Sigma_{T,-}$, and $\Sigma_{T,+}$ by the minimizing hull of $\Sigma_{T_,-}\cup \partial\cJ$. Then, for $t > T$, the surfaces $\Sigma_{t}$ are a weak inverse mean curvature flow with initial condition at $t=T$ given by $\Sigma_{T,+}$.
\item $\Sigma_{T,+}$ is connected.
\item There exists a constant $\beta\geq0$ so that for $t > T$, we have that  
\begin{equation*}
\cH^{2}_{g}(\Sigma_{t}) = e^{t +\beta} \cH^{2}_{g}(\partial\Omega).
\end{equation*}
The constant $\beta$ satisfies the bound
\begin{equation*}
\beta \leq \log\left( 1 + \frac{\cH^{2}_{g}(\partial \cJ)}{\cH^{2}_{g}(\partial \Omega)} e^{-T} \right).
\end{equation*}
\end{enumerate}
Furthermore, there exists $C_{1} ,C_{2} > 0$ so that the Hawking mass of $\Sigma_{t}$ behaves as follows: If $m_{H}(\partial \Omega) \geq 0$, then for all $ t \not = T$
we have the bound
\begin{equation*}
m_{H}(\Sigma_{t}) \geq m_{H}(\partial\Omega) - \delta - C_{2} \cH^{2}_{g}(\partial\Omega)^{\frac 12} \int_{\partial\cJ} \left( H^{2}_{g} -4 \right)d\cH^{2}_{g}.
\end{equation*}
If $m_{H}(\partial \Omega) < 0$, then for all $t \not = T$
we have the bound
\begin{equation*}
m_{H}(\Sigma_{t}) \geq C_{1} m_{H}(\partial\Omega) - \delta - C_{2} \cH^{2}_{g}(\partial\Omega)^{\frac 12} \int_{\partial\cJ} \left( H^{2}_{g} -4 \right)d\cH^{2}_{g}.
\end{equation*}
The constant $C_{1}$ only depends on an upper bound for $\cH^{2}_{g}(\partial \cJ)$, while the constant $C_{2}$ depend on the metric $g$, the compact set $\Gamma$, as well as upper bounds for the quantities $\cH^{2}_{g}(\partial \cJ)$ and $\max_{p\in\partial\Omega} H_{g}(p)$.
 \end{prop}
\begin{proof}
In \cite[\S 6]{HuiskenIlmanen:Penrose}, Huisken--Ilmanen argue that if we consider the inverse mean curvature flow starting from $\partial\Omega$, then there exists some time $T$ so that $\Sigma_{t}$ is disjoint from the interior of $\cJ$ for all $t \leq T$ and that the minimizing hull of $\Sigma_{T}\cup\partial\cJ$ is connected. They do so by taking the largest $T$ so that $\Sigma_{t}$ is disjoint from the interior of $\cJ$ for all $t < T$ (such $T$ exists, by the gradient bounds for weak solutions to inverse mean curvature flow). Shrinking $T$ slightly if necessary, they then arrange that $\Sigma_{T}\cap\cJ =\emptyset$. On the other hand, because $\Sigma_{T}$ is about to touch $\cJ$, near some point, the two $\Sigma_{T}$ and $\partial\cJ$ look like close, nearly parallel planes. One may easily see that if the planes are close enough, by forming a neck between them one may strictly reduce the area. Thus, the minimizing hull of $\partial\Sigma_{T}\cup\partial\cJ$ is connected. Then, they redefine $\Sigma_{T}$ to be given by this minimizing hull, and restart the flow with initial conditions given by $\Sigma_{T}$.

One may see that in the asymptotically flat setting, the Hawking mass\footnote{Due to the different assumption on scalar curvature, i.e., $R_{g}\geq 0$, the appropriate quantity to consider in the asymptotically flat setting would be $m_{H}(\Sigma) := (16\pi)^{- \frac 32}  \cH^{2}_{g}(\Sigma)^{\frac 12} \left( 16\pi - \int_{\Sigma} H_{g}^{2} d\cH^{2}_{g}\right)$.} is actually monotone nondecreasing along this process, when jumping over outer-minimizing minimal surfaces. This is because the ``new part,'' of $\Sigma_{T}$ is minimal so the quantity $\int H^{2}$ decreases under the jump. Furthermore, it is clear from the minimizing hull property of the flow, that the area must strictly increase under the jump. In our setting, the minimizing hull property still holds, so the area does increase. However, the relevant mean curvature term is $\int \left( H^{2} - 4\right)$, which does not behave as nicely as in the asymptotically flat case, in particular because the integrand could be negative. In addition, we would like to jump over regions which are not minimal, which is an additional complication.

As such, we must modify the jump procedure, so as to jump at (nearly) the earliest possible time. In particular, this choice allows us to arrange that the area drops only a small amount over the jump. We have illustrated a jump in Figure \ref{fig:how-to-jump}.

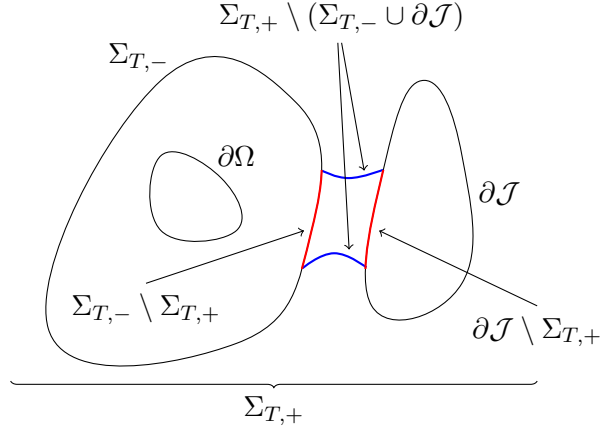
\begin{figure}
\begin{tikzpicture}

	\draw [tension = .8, blue, thick] plot [smooth] coordinates {(2.14,.8) (2.5, .7) (2.95,.81)};
	\draw [tension = .8, blue, thick] plot [smooth] coordinates {(1.88,-.5) (2.3,-.3) (2.73,-.49)};
	
	\draw [tension = .8] plot [smooth cycle] coordinates {(0,0) (1,0) (.8,.7) (0,1)};

	\begin{scope}
		\draw [tension = .8] plot [smooth cycle] coordinates {(-1.5,-1.3) (1,-1.5) (2,0) (1.8,1.7) (0,2)};
		\draw [tension = .8] plot [smooth cycle] coordinates {(2.8,-1) ( 4,-.7) (4,1) (3.5,2) (3,1)};
	\end{scope}
	
	\begin{scope}
		\clip (1.87,-.5) rectangle (3,.8);
		\draw [red, thick, tension = .8] plot [smooth cycle] coordinates {(-1.5,-1.3) (1,-1.5) (2,0) (1.8,1.7) (0,2)};
		\draw [red, thick, tension = .8] plot [smooth cycle] coordinates {(2.8,-1) ( 4,-.7) (4,1) (3.5,2) (3,1)};
	\end{scope}

	\node at (-.3,2.3) {$\Sigma_{T,-}$};
	\node at (1,1) {$\partial \Omega$};
	\node at (4.5,.5) {$\partial\cJ$};
	
	\draw [->] (2.4,2.5) node [above] {$\Sigma_{T,+}\setminus(\Sigma_{T,-}\cup\partial\cJ)$} -- (2.7,.8);
	\draw [->] (2.35,2.5) -- (2.5,-.25);
	
	\draw [->] (-.2,-.7) node [below] {$\Sigma_{T,-}\setminus \Sigma_{T,+}$} -- (1.9,0);
	
	\draw [->] (5,-1) node [below] {$\partial\cJ\setminus\Sigma_{T,+}$} -- (2.9,0);
	
	\draw [snake = brace] (5,-2) -- (-2,-2) node [pos = .5, below = 2] {$\Sigma_{T,+}$};
\end{tikzpicture}
\caption{A diagram of the inverse mean curvature flow with jumps as defined in Proposition \ref{prop:how-to-jump}. We replace the red region, $(\Sigma_{T,-}\cup \partial\cJ)\setminus\Sigma_{T,+}$ with the blue region, $\Sigma_{T,+}\setminus(\Sigma_{T,-}\cup\partial\cJ)$, by taking the minimizing hull of $\Sigma_{T,-} \cup\partial\cJ$. By choosing $T$ nearly as small as possible, we can ensure that the blue and red regions have almost the same area. }
\label{fig:how-to-jump}
\end{figure}

\begin{claim}
For any $\epsilon>0$, there exists $T$ so that $\Sigma_{t}$ is disjoint from $\cJ$ for all $t \leq T$, so that the outer-minimizing enclosure of $\Sigma_{T} \cup \partial\cJ$, which we will write as $(\Sigma_{T} \cup \partial\cJ)'$, is connected, and so that
\begin{equation*}
\cH^{2}_{g}((\Sigma_{T} \cup \partial\cJ)') \geq\cH^{2}_{g} (\Sigma_{T} \cup \partial\cJ) - \epsilon.
\end{equation*}
\end{claim}
\begin{proof}[Proof of the Claim]\renewcommand{\qedsymbol}{}
We define $\hat T : = \inf\{t : (\Sigma_{t} \cup \partial\cJ)' \text{ is connected} \}$ and choose sequences $s_{k}\nearrow \hat T$ and $t_{k}\searrow \hat T$. By definition, $(\Sigma_{t_{k}} \cup \partial\cJ)'$ is connected for each $k$ and $(\Sigma_{s_{k}}\cup\partial\cJ)'$ is disconnected for each $k$. Also, we may arrange that $\Sigma_{t_{k}}$ is disjoint from $\cJ$, for $k$ sufficiently large (this follows from the fact that $\hat T$ must be strictly before the first time of contact; see \cite[\S 6]{HuiskenIlmanen:Penrose}). Suppose that 
\begin{equation*}
\cH^{2}_{g}((\Sigma_{t_{k}} \cup \partial\cJ)') < \cH^{2}_{g} (\Sigma_{t_{k}} \cup \partial\cJ) - \epsilon.
\end{equation*}
for each $k$. Note that $\cH^{2}_{g}(\Sigma_{t}\cup\partial\cJ)$ is continuous in $t$ as long as $\Sigma_{t}$ remains disjoint from $\cJ$; this follows easily from the exponential area growth of $\Sigma_{t}$ (see (6) in Theorem \ref{theo:HI-wIMCF}). Hence, we may choose $k$ sufficiently large so that 
\begin{equation*}
\cH^{2}_{g} (\Sigma_{t_{k}} \cup \partial\cJ) - \epsilon < \cH^{2}_{g}(\Sigma_{s_{k}}\cup\partial \cJ).
\end{equation*}
Combining these two inequalities, we obtain
\begin{equation*}
\cH^{2}_{g}((\Sigma_{t_{k}} \cup \partial\cJ)') < \cH^{2}_{g}(\Sigma_{s_{k}}\cup\partial \cJ).
\end{equation*}
This is a contradiction: on one hand, $(\Sigma_{t_{k}} \cup \partial\cJ)'$ contains $\Sigma_{s_{k}}\cup\partial \cJ$, so this shows that $\Sigma_{s_{k}}\cup\partial \cJ$ is not outer-minimizing. On the other hand, it is not hard to see that the only way that this could happen is if $(\Sigma_{s_{k}}\cup\partial \cJ)'$ is connected, as $\Sigma_{s_{k}}$ and $\partial \cJ$ are both individually outer-minimizing. This contradicts our choice of $s_{k}$.
\end{proof}

We choose $T$ as in the claim and write $\Sigma_{t}$ for $t < T$ and $\Sigma_{T,-}$ for the flow continued until time $T$. We further define $\Sigma_{T,+} = (\Sigma_{T,-} \cup \partial\cJ)'$. Thus, if $\Sigma_{T,-}$ is smooth, we may compute as follows (if it is not smooth, we may approximate it in $C^{1}$ from the outside in by smooth surfaces as in \cite[\S 6]{HuiskenIlmanen:Penrose} and apply this argument to the approximating surfaces---that the inequality also holds for the limit then follows from lower semicontinuity of $\int H^{2}$ under $C^{1}$ convergence, cf.\ \cite[(1.14)]{HuiskenIlmanen:Penrose})
\begin{align*}
m_{H}(\Sigma_{T,+}) & =\frac{ \cH^{2}_{g}(\Sigma_{T,+})^{\frac 12}}{(16\pi)^{\frac 32}}\left(16\pi - \int_{\Sigma_{T,+}}\left(H^{2}_{g}-4\right) d\cH^{2}_{g}\right)\\
& = \frac{\cH^{2}_{g}(\Sigma_{T,+})^{\frac 12}}{(16\pi)^{\frac 32}}\left(16\pi - \int_{\Sigma_{T,-}}\left(H^{2}_{g}-4\right) d\cH^{2}_{g}\right) \\
& \qquad + \frac{ \cH^{2}_{g}(\Sigma_{T,+})^{\frac 12}}{(16\pi)^{\frac 32}} \int_{\Sigma_{T,-} \setminus \Sigma_{T,+}}\left(H^{2}_{g}-4\right) d\cH^{2}_{g}\\
&  \qquad+ 4\frac{ \cH^{2}_{g}(\Sigma_{T,+})^{\frac 12} }{(16\pi)^{\frac 32}} \cH^{2}_{g}(\Sigma_{T,+} \setminus (\Sigma_{T,-}\cup\partial\cJ)) \\
& \qquad-  \frac{\cH^{2}_{g}(\Sigma_{T,+})^{\frac 12}}{(16\pi)^{\frac 32}} \int_{\partial\cJ \cap \Sigma_{T,+}} \left(H^{2}_{g}-4\right)d\cH^{2}_{g}\\
& =  \frac{\cH^{2}_{g}(\Sigma_{T,+})^{\frac 12}}{\cH^{2}_{g}(\Sigma_{T,-})^{\frac 12}}m_{H}(\Sigma_{T,-})+  \frac{\cH^{2}_{g}(\Sigma_{T,+})^{\frac 12} }{(16\pi)^{\frac 32}} \int_{\Sigma_{T,-} \setminus \Sigma_{T,+}}\left(H^{2}_{g}-4\right) d\cH^{2}_{g}\\
&  \qquad+ 4\frac{ \cH^{2}_{g}(\Sigma_{T,+})^{\frac 12}}{(16\pi)^{\frac 32}} \cH^{2}_{g}(\Sigma_{T,+} \setminus (\Sigma_{T,-}\cup\partial\cJ)) \\
& \qquad -  \frac{\cH^{2}_{g}(\Sigma_{T,+})^{\frac 12} }{(16\pi)^{\frac 32}} \int_{\partial\cJ \cap \Sigma_{T,+}} \left(H^{2}_{g}-4\right)d\cH^{2}_{g}\\
& \geq \frac{\cH^{2}_{g}(\Sigma_{T,+})^{\frac 12}}{\cH^{2}_{g}(\Sigma_{T,-})^{\frac 12}}m_{H}(\Sigma_{0}) - \frac{ \cH^{2}_{g}(\Sigma_{T,+})^{\frac 12}}{(16\pi)^{\frac 32}} \int_{\partial\cJ} \left(H^{2}_{g}-4\right)d\cH^{2}_{g} \\
&  \qquad +  4 \frac{ \cH^{2}_{g}(\Sigma_{T,+})^{\frac 12}}{(16\pi)^{\frac 32}} \left(\cH^{2}_{g}(\Sigma_{T,+} \setminus (\Sigma_{T,-}\cup\partial\cJ)) -  \cH^{2}_{g}(\Sigma_{T,-} \setminus \Sigma_{T,+})\right).
\end{align*}
Now, the above claim implies that
\begin{align*}
 \cH^{2}_{g}(\Sigma_{T,+} \setminus &  (\Sigma_{T,-}\cup\partial\cJ))  -  \cH^{2}_{g}(\Sigma_{T,-} \setminus \Sigma_{T,+}) - \cH^{2}_{g}(\partial\cJ \setminus \Sigma_{T,+}) \\
& = \cH^{2}_{g}(\Sigma_{T,+}) - \cH^{2}_{g}(\Sigma_{T,-}) - \cH^{2}_{g}(\partial\cJ)  \geq -\epsilon.
\end{align*}
Thus, we have the inequality
\begin{align*}
m_{H}(\Sigma_{T,+}) &  \geq \frac{\cH^{2}_{g}(\Sigma_{T,+})^{\frac 12}}{\cH^{2}_{g}(\Sigma_{T,-})^{\frac 12}}m_{H}(\Sigma_{0})\\
&   - 4 \epsilon  \frac{\cH^{2}_{g}(\Sigma_{T,+})^{\frac 12}}{(16\pi)^{\frac 32}}   - \frac{ \cH^{2}_{g}(\Sigma_{T,+})^{\frac 12} }{(16\pi)^{\frac 32}} \int_{\partial\cJ} \left(H^{2}_{g}-4\right)d\cH^{2}_{g} .
\end{align*} 
Furthermore, by the exponential area growth of $\Sigma_{t}$ (cf. (6) in Theorem \ref{theo:HI-wIMCF}), we have that 
\begin{equation*}
\cH^{2}_{g}(\Sigma_{T,+}) \leq e^{T} \cH_{g}^{2}(\partial\Omega) + \cH^{2}_{g}(\partial\cJ) \leq \cH_{g}^{2}(\partial\Omega) \left( e^{T} + \cH^{2}_{g}(\partial\cJ) \right). 
\end{equation*}
We may bound $T$, the time to jump, by using the gradient bounds for weak inverse mean curvature flow, described in (2) in Theorem \ref{theo:HI-wIMCF}. From this, one may clearly bound the time $t$ after which $\cJ$ would be totally contained inside of $\Sigma_{t}$ (of course, the jump time $T$ must be before this time) in terms of $\max_{p\in\partial\Omega} H_{g}(p)$ and the compact set $\Gamma$. In particular, we have that 
\begin{equation*}
\frac{\cH^{2}_{g}(\Sigma_{T,+})^{\frac 12}}{(16\pi)^{\frac 3 2}} \leq C_{2} \cH^{2}_{g}(\partial\Omega)^{\frac 12}, 
\end{equation*}
where $C_{2}$ depends on the metric $g$, the compact set $\Gamma$, as well as upper bounds for the quantities $\cH^{2}_{g}(\partial \cJ)$ and $\max_{p\in\partial\Omega} H_{g}(p)$. Hence, if $m_{H}(\partial\Omega) \geq 0$, then because $\cH^{2}_{g}(\Sigma_{T,+})\geq \cH^{2}_{g}(\Sigma_{T,-})$ by the outer-minimizing property, then by choosing $\epsilon < \frac {1}{4C_{2}\cH^{2}_{g}(\partial\Omega)^{\frac 12}} \delta$, we obtain
\begin{equation*}
m_{H}(\Sigma_{T,+})  \geq m_{H}(\Sigma_{0}) - \delta - C_{2} \cH^{2}_{g}(\partial\Omega)^{\frac 12} \int_{\partial\cJ} \left(H^{2}_{g}-4\right)d\cH^{2}_{g} .
\end{equation*} 

On the other hand, when $m_{H}(\partial\Omega) < 0$, its coefficient could make the inequality worse. Hence, we must use the bound
\begin{equation*}
\frac{\cH^{2}_{g}(\Sigma_{T,+})}{\cH^{2}_{g}(\Sigma_{T,-})} \leq \frac{\cH^{2}_{g}(\Sigma_{T,-}) + \cH^{2}_{g}(\partial\cJ)}{\cH^{2}_{g}(\Sigma_{T,-})} \leq 1 + \cH^{2}_{g}(\partial\cJ),
\end{equation*}
which follows from the outermost property of $\Sigma_{T,-}$ and assumption (5) in the statement of the Proposition. Now, the asserted inequality for $m_{H}(\Sigma_{T,+})$ follows in this case as well by the same argument we have just used. 

Now, it follows that we may restart the flow at $\Sigma_{T,+}$ by the same argument as Huisken--Ilmanen, in particular using \cite[Lemma 6.2]{HuiskenIlmanen:Penrose} to approximate $\Sigma_{T,+}$ in $C^{1}$ by smooth surfaces. We will write the surface obtained by flowing $\Sigma_{T,+}$ for time $t - T$ by $\Sigma_{t}$. By the exponential area growth of the flow, we may define $\beta$ so that $\cH^{2}(\Sigma_{t}) = e^{t+\beta} \cH^{2}_{g}(\partial\Omega)$. On the other hand, by the outer-minimizing property of $\Sigma_{T,-}$, we see that
\begin{align*}
e^{T+\beta}\cH^{2}_{g}(\partial\Omega) & = \cH^{2}_{g}(\Sigma_{T,+}) \leq \cH^{2}_{g}(\Sigma_{T,-}) + \cH^{2}_{g}(\partial \cJ) = e^{T}\cH^{2}_{g}(\partial\Omega) + \cH^{2}_{g}(\partial \cJ).
\end{align*}
Thus, 
\begin{equation*}
\beta \leq \log\left( 1 + \frac{\cH^{2}_{g}(\partial\cJ)}{\cH^{2}_{g}(\partial\Omega)} e^{-T} \right).
\end{equation*}
This completes the proof.
\end{proof}

In the remainder of this section, we recall several results contained in \cite{BrendleChodosh:VolCompAH}. Because the mechanism in \cite{BrendleChodosh:VolCompAH} comprises an essential element of our argument (and there are minor modifications to our setting), we recall the proofs for the reader's convenience. We fix $(M,g)$ an asymptotically hyperbolic manifold with scalar curvature $R_{g}\geq -6$ and $\Omega$ a connected, outer-minimizing, smooth open set of finite perimeter (possibly not containing the horizon). We let $\Sigma_{t} = \partial\{u<t\}$ denote the (weak) inverse mean curvature flow starting at $\partial^{*}\Omega$, which exists by Theorem \ref{theo:HI-wIMCF}. We additionally let $\Omega_{t} : =\{u<t\}\setminus \Omega$ denote the region swept out by the flow (note that $\Omega_{t}$ does not contain $\Omega$). 

Note that the Hawking mass in \cite{BrendleChodosh:VolCompAH} differs from the present work by a constant multiple of $(16\pi)^{\frac 32}$. We emphasize that the convention used here is chosen so that the Hawking mass of a centered coordinate sphere in Schwarzschild-AdS of mass $m$ is equal to $m$.

\begin{prop}[cf. {\cite[Proposition 3]{BrendleChodosh:VolCompAH}}]\label{prop:vol-swept-IMCF-bd}
Suppose that for $\tau \in [0,T_{0})$, $\Sigma_{\tau}$ remains disjoint from the horizon and $m_{H}(\partial \Omega) \geq m$. Then, for $\tau \in [0,T_{0})$
\begin{equation*}
\sL^{3}_{g}(\Omega_{\tau}) \geq \int_{0}^{\tau} e^{\frac{3t}{2}} A^{\frac 32} \left( 4 e^{t} A + 16\pi - e^{-\frac t 2} A^{-\frac 12} (16\pi)^{\frac 32}  m \right)^{-\frac 12} dt,
\end{equation*}
where $A : = \cH^{2}_{g}(\partial\Omega)$. Equality holds for $\tau > 0$ if and only if $\Omega$ is a centered coordinate ball in exact Schwarzschild-AdS of mass $m=m_{H}(\partial\Omega)$.
\end{prop}
\begin{proof}
By Geroch monotonicity (cf.\ (8) in Theorem \ref{theo:HI-wIMCF}), we have that $m_{H}(\Sigma_{t}) \geq m$ for $t \in [0,T_{0})$ (we have not assumed that the extension $\hat g$ of the metric inside of the horizon has $R_{g}\geq -6$, so this monotonicity could fail if the flow passes through the origin---we will never allow this to happen). Furthermore, we have (cf.\ (6) in Theorem \ref{theo:HI-wIMCF}) that $\cH^{2}_{g}(\Sigma_{t}) = e^{t}\cH^{2}_{g}(\partial\Omega)$. Hence, for a.e., $t > 0$, we have that
\begin{align*}
& \int_{\Sigma_{t}} \frac{1}{|du|_{g}} d\cH^{2}_{g} \\
& = \int_{\Sigma_{t}}\frac{1}{H_{g}} d\cH^{2}_{g}\\
& \geq \cH^{2}_{g}(\Sigma_{t})^{\frac 32} \left( \int_{\Sigma_{t}} H_{g}^{2}d\cH^{2}_{g}\right)^{-\frac 12}\\
& = \cH^{2}_{g}(\Sigma_{t})^{\frac 32} \left( 4\cH^{2}_{g}(\Sigma_{t}) + 16\pi - e^{-\frac t 2} \cH^{2}_{g}(\Sigma_{t})^{-\frac 12} (16\pi)^{\frac 32}  m_{H}(\Sigma_{t}) \right)^{-\frac 12} \\
& =e^{\frac{3t}{2}} A^{\frac 32}  \left( 4 e^{t} A + 16\pi - e^{-\frac t 2} A^{-\frac 12} (16\pi)^{\frac 32}  m_{H}(\Sigma_{t}) \right)^{-\frac 12} \\
& \geq e^{\frac{3t}{2}} A^{\frac 32}  \left( 4 e^{t} A + 16\pi - e^{-\frac t 2} A^{-\frac 12} (16\pi)^{\frac 32}  m \right)^{-\frac 12}.
\end{align*}
Integrating this with respect to $t$, from $0$ to $\tau$ yields (using the co-area formula)
\begin{align*}
\sL^{3}_{g}(\Omega_{\tau}) & \geq \int_{0}^{\tau}\left(\int_{\Sigma_{t}} \frac{1}{|du|_{g}} d\cH^{2}_{g}\right)d\tau\\
& \geq \int_{0}^{\tau} e^{\frac{3t}{2}} A^{\frac 32} \left( 4 e^{t} A + 16\pi - e^{-\frac t 2} A^{-\frac 12} (16\pi)^{\frac 32}  m \right)^{-\frac 12} dt,
\end{align*}
as claimed. The equality case follows easily from the case of equality in Geroch monotonicity (cf.\ (9) in Theorem \ref{theo:HI-wIMCF}). 
\end{proof}
In the remainder of this section, we assume that $\Omega$ contains the horizon. We let $A = \cH^{2}_{g}(\partial^{*}\Omega)$. 
\begin{prop}[cf. {\cite[p.\ 5]{BrendleChodosh:VolCompAH}}]\label{prop:imcf-plus-isoper-ineq}
For a fixed $m \geq 0$ and $t\geq 0$, define $\cB_{\overline g_{m}}(e^{t}A)$ to be the centered coordinate sphere in $(\overline M_{m},\overline g_{m})$ satisfying $\cH^{2}_{\overline g_{m}}(\partial \cB_{\overline g_{m}}(e^{t}A)) =  e^{t}A$. Then,
\begin{equation*}
\sL^{3}_{\overline g_{m}}(\cB_{\overline g_{m}}(e^{t}A)) \geq \sL^{3}_{\overline g_{m}}(\Omega_{t} \cup\Omega) + o(1)
\end{equation*}
as $t\to\infty$.
\end{prop}
\begin{proof}
It is clear that $\{s\leq e^{\frac{4t}{9}}\}$ is eventually a weak subsolution to the inverse mean curvature flow, in the sense that for $t \geq t_{0}$ sufficiently large, the surface $\{s= e^{\frac{4t}{9}}\}$ flows with speed less than $\frac {1}{H_{g}}$. By the gradient bound (2) in Theorem \ref{theo:HI-wIMCF}, $\Omega_{t_{1}}\cup \Omega$ will contain $\{s\leq  e^{\frac{4t_{0}}{9}}\}$, if we choose $t_{1}\geq t_{0}$ sufficiently large. Now, by the avoidance property, i.e., (10) in Theorem \ref{theo:HI-wIMCF}, for $t\geq t_{1}$, we have that $\{s\leq  e^{\frac{4(t+ t_{0}-t_{1})}{9}}\} \subset \Omega_{t}\cup\Omega$. This implies that on $\Sigma_{t} = \partial(\Omega_{t}\cup\Omega)$, we have that
\begin{equation*}
|g-\overline g_{m}|_{\overline g_{m}}\leq O(s^{-3}) \leq O(e^{-\frac{4t}{3}}).
\end{equation*}
As such,
\begin{equation*}
\cH^{2}_{\overline g_{m}}(\Sigma_{t}) = \cH^{2}_{ g}(\Sigma_{t}) (1+O(e^{-\frac{4t}{3}})) = e^{t} A +O(e^{-\frac t 3}). 
\end{equation*}
By \cite{CovinoGerekGreenbergKrummel}, centered coordinate balls in $(\overline M_{m},\overline g_{m})$ are isoperimetric (we are not using the fact that the coordinate balls are \emph{uniquely} isoperimetric for $m > 0$, so this holds for $m=0$ as well). Hence
\begin{equation*}
\sL^{3}_{\overline g_{m}}(\cB_{\overline g_{m}}(e^{t}A + O(e^{-\frac t 3})) \geq \sL_{\overline g_{m}}^{3}(\Omega_{t}\cup\Omega). 
\end{equation*}
Combined with Lemma \ref{lemma:vol-large-balls}, this finishes the proof. 
\end{proof}
The barrier argument we have just used also establishes 
\begin{prop}[cf. {\cite[Proposition 2]{BrendleChodosh:VolCompAH}}]\label{prop:imcf-forms-exhaustion}
For any sequence $t_{i}\to\infty$, the sets $\Omega_{t_{i}}\cup\Omega$ form an exhaustion of $\RR^{3}$. 
\end{prop}




\section{The renormalized volume} \label{sec:renorm-vol} 
In this section, we discuss the renormalized volume of asymptotically hyperbolic manifolds. In \cite{BrendleChodosh:VolCompAH} we have defined the renormalized volume of an asymptotically hyperbolic manifold as follows. 
\begin{defi}\label{defi:renorm-vol}
For $\Omega_{i}$ an exhaustion of $\RR^{3}$ by open sets, we define the \emph{renormalized volume} of $(M,g)$ by $V(M,g) : =  \lim_{i\to\infty} \left( \sL^{3}_{g}(\Omega_{i}) - \sL^{3}_{\overline g}(\Omega_{i})\right)$.
\end{defi}
We remind the reader that $\sL^{3}_{g}(\Omega_{i})$ is the $g$-volume of the region $\Omega_{i}\cap M = \Omega_{i}\setminus \overline K$. It is not hard to check that if $(M,g)$ is asymptotically hyperbolic, then $V(M,g)$ is finite and independent of the exhaustion $\Omega_{i}$. The main result in \cite{BrendleChodosh:VolCompAH} was a Penrose type inequality for asymptotically hyperbolic manifolds where the renormalized volume replaced the mass. In \cite{BrendleChodosh:VolCompAH}, we required\footnote{Here, we have studied the case where $\partial M$ has mean curvature $H_{g}\equiv 2$ because of the fact that large isoperimetric regions in $(M,g)$ have mean curvature $H_{g} \approx 2$. For example, we make use of this when proving the existence of large isoperimetric regions in \S \ref{sec:pf-exist-in-AH}.}  that $\partial M$ was an outermost connected \emph{minimal} surface, rather than having $H_{g}\equiv 2$. This distinction somewhat changes the behavior of the renormalized volume. In fact, it is easy to check that Schwarzschild-AdS (with boundary the $H_{g}\equiv 2$ coordinate sphere) of mass $\mm > 0$ has \emph{negative} renormalized volume. However, can modify the techniques used in \cite{BrendleChodosh:VolCompAH} in a straightforward manner to prove the following proposition. 

\begin{prop}
Suppose that $(M,g)$ is asymptotically hyperbolic in the sense of Definition \ref{defi:AH-metric}, and has $R_{g}\geq -6$. If $m\geq0$ is chosen so that $A_{\partial M} : = \cH^{2}_{g}(\partial M) = \cH^{2}_{\overline g_{m}}(\partial \overline M_{m}) : = A_{\partial \overline M_{m}}$ then $V(M,g) \geq V(\overline M_{m},\overline g_{m})$, with equality if and only if $(M,g)$ is isometric to $(\overline M_{m},\overline g_{m})$. 
\end{prop}
\begin{proof}
Let $\Sigma_{t} = \partial\Omega_{t}$ denote the weak solution to the inverse mean curvature flow starting at $\partial M$, as in Theorem \ref{theo:HI-wIMCF}. Recall that $\Omega_{t}$ as a set in $\RR^{3}$ contains the horizon. Note that $\cB_{\overline g_{m}}(e^{t}A_{\partial \overline M_{m}})$ is a solution to the inverse mean curvature flow in $(\overline M_{m},\overline g_{m})$. By Proposition \ref{prop:vol-swept-IMCF-bd} applied to both flows (using that equality holds in the model case), we have that
\begin{equation*}
\sL^{3}_{g}(\Omega_{t}) \geq \sL^{3}_{\overline g_{m}}(\cB_{\overline g_{m}}(e^{t}A_{\partial \overline M_{m}}))
\end{equation*}
for $t \geq 0$, where $\cB_{\overline g_{m}}(e^{t}A_{\partial \overline M_{m}})$ is the centered coordinate sphere in $(\overline M_{m},\overline g_{m})$ with $\cH^{2}_{\overline g_{m}}(\cB_{\overline g_{m}}(e^{t}A_{\partial \overline M_{m}})) = e^{t}A_{\partial \overline M_{m}}$. Then, by Proposition \ref{prop:imcf-plus-isoper-ineq}, we have that 
\begin{equation*}
\sL^{3}_{g}(\Omega_{t}) \geq \sL^{3}_{\overline g_{m}}(\Omega_{t}) +o(1),
\end{equation*}
or equivalently, 
\begin{equation*}
\sL^{3}_{g}(\Omega_{t}) - \sL^{3}_{\overline g}(\Omega_{t}) \geq \sL^{3}_{\overline g_{m}}(\Omega_{t}) - \sL^{3}_{g}(\Omega_{t}) +o(1),
\end{equation*}
as $t\to\infty$. Sending $t \to\infty$, we conclude that $V(M,g) \geq V(\overline M_{m},\overline g_{m})$. If equality holds, it is not hard to see that equality must hold in Geroch monotonicity, i.e., $m_{H}(\Sigma_{t}) = m$ for all $t \geq 0$, which implies by (9) in Theorem \ref{theo:HI-wIMCF} that $(M,g)$ is isometric to $(\overline M_{m},\overline g_{m})$.
\end{proof}

We may compute
\begin{equation*}
V(\overline M_{m},\overline g_{m}) = 4\pi \lim_{R\to\infty} \left[ \int_{2m}^{R} \frac{s^{2}}{\sqrt{1+s^{2}-2m s^{-1}}}  ds - \int_{0}^{R} \frac{s^{2}}{\sqrt{1+s^{2}}} ds \right].
\end{equation*}
Because the integrands are non-singular\footnote{In \cite{BrendleChodosh:VolCompAH} the corresponding integrand was singular at the lower limit of integration so considerable care needed to be applied in the next step, cf.\ \cite[Appendix A]{BrendleChodosh:VolCompAH}.} at the lower limit of integration, we see that
\begin{equation*}
\frac{d}{dm} V(\overline M_{m},\overline g_{m}) = - 16 \pi m + 4\pi \int_{2m}^{\infty} \frac{s}{(1+s^{2}-2ms^{-1})^{\frac 32}}ds
\end{equation*}
From $\cH^{2}_{\overline g_{m}}(\partial\overline M_{\overline g_{m}}) = A_{\partial \overline M_{m}}= 16 \pi m^{2}$, this implies
\begin{equation*}
\frac{d}{dm} \left( V(\overline M_{m},\overline g_{m}) + \frac 1 2 A_{\partial \overline M_{m}} \right) > 0
\end{equation*}
where $ A_{\partial \overline M_{m}} : = \cH^{2}_{\overline g_{m}}(\partial \overline M_{m})$. Combined with the previous proposition, we have thus proven: 
\begin{prop}\label{prop:vol-comp-H2}
Suppose that $(M,g)$ is an asymptotically hyperbolic manifold with $R_{g} \geq -6$. We emphasize that this includes the statement that $\partial M$, if non-empty, is a connected, outermost, CMC surface with $H_{g}\equiv 2$. Let $A_{\partial M} : = \cH^{2}_{g}(\partial M)$. Then, the renormalized volume of $(M,g)$ satisfies
\begin{equation*}
V(M,g) + \frac 12 A_{\partial M} \geq 0,
\end{equation*}
with equality if and only if $(M,g)$ is isometric to hyperbolic space. 
\end{prop}
We remark that it is possible to drop the assumption that $\partial M$ is connected in this case. Because we do not make use of this later, we only briefly describe the proof: we may use the inverse mean curvature flow with jumps starting at one component of the horizon and jumping over the other horizon regions (using Proposition \ref{prop:how-to-jump} repeatedly if necessary). Then, a computation similar to that done in the end of Proposition \ref{prop:coarse-vol-bds}, shows that we may bound the volume gained during a jump by the area of the component of the horizon being jumped over, obtaining the desired inequality.

\section{Volume bounds for large isoperimetric regions}\label{sec:vol-bd-large-iso}

In this section, we will assume that $(M,g)$ is asymptotically hyperbolic and has scalar curvature $R_{g}\geq -6$. 

\begin{prop}\label{prop:imcf-bds-good-case}
Suppose that $\Omega$ is a Borel set of finite perimeter in $(M,g)$ strictly containing the horizon with smooth, connected, outer-minimizing, CMC boundary $\Sigma:=\partial\Omega$. Suppose further that $0 \leq m\leq m_{H}(\Sigma)$. Then
\begin{equation*}
\sL^{3}_{g}(\Omega) \leq \sL^{3}_{\overline g_{m}}(\cB_{\overline g_{m}}(A)) + V(M,g) - V(\overline M_{m},\overline g_{m}),
\end{equation*}
where $A = \cH^{2}_{g}(\Sigma)$ is the $g$-area of the boundary of $\Omega$. 
\end{prop}
\begin{proof}
Let $\Sigma_{\tau}$ denote the weak solution to inverse mean curvature flow starting at $\Sigma$ and write $\Omega_{\tau}$ for the region bounded between $\Sigma$ and $\Sigma_{\tau}$. By Proposition \ref{prop:vol-swept-IMCF-bd}, we have that
\begin{equation}\label{eq:vol-cont-ICMF-region-lb}
\sL^{3}_{g}(\Omega_{\tau}) \geq \int_{0}^{\tau} e^{\frac{3t}{2}} A^{\frac 32} \left( 4 e^{t} A + 16\pi - e^{-\frac t 2} A^{-\frac 12} (16\pi)^{\frac 32}  m \right)^{-\frac 12} dt. 
\end{equation}
Notice that there exists a coordinate sphere (outside the horizon) of area $A$ in $(\overline M_{m},\overline g_{m})$. To see this, note\footnote{That $\Sigma$ satisfies $H_{g}>2$ is a consequence of the fact that $\partial M$ is an outermost $H_{g}\equiv 2$ surface---by definition of outermost, $\Sigma$ does not have $H_{g}\equiv 2$, and if it had $H_{g}<2$, we could minimize the brane functional to the outside of $\Sigma$ as in Proposition \ref{prop:reg-iso-surf} to obtain a compact $H_{g}\equiv 2$ surface outside of $\Sigma$, contradicting the outermost assumption on $\partial M$.} that the mean curvature of $\Sigma$ satisfies $H_{g} > 2$, so
\begin{align*}
(16\pi)^{-\frac 12} \cH^{2}_{\overline g_{m}}(\partial\overline M_{m})^{\frac 12} & =  m_{H}(\partial \overline M_{m},\overline g_{m}) \\
& = m\\
& \leq m_{H}(\Sigma,g) \\
& = (16\pi)^{- \frac 32} A^{\frac 12}\left( 16\pi - A(H^{2}_{g}-4) \right)\\
& < (16\pi)^{-\frac 12} A^{\frac 12}.
\end{align*}
From this, it is clear that there is a coordinate sphere $\cB_{\overline g_{m}}(A)$, in $(\overline M_{m},\overline g_{m})$ having area $A$. 

If we flow $\partial\cB_{\overline g_{m}}(A)$ by inverse mean curvature flow in $(\overline M_{m},\overline g_{m})$, it is easy to see that after time $\tau$ we obtain $\partial\cB_{\overline g_{m}}(e^{\tau}A)$. In this case, by Proposition \ref{prop:vol-swept-IMCF-bd} we must have equality in \eqref{eq:vol-cont-ICMF-region-lb} i.e.,
\begin{align*}
& \sL^{3}_{\overline g_{m}} (\cB_{\overline g_{m}}(e^{\tau}A) \setminus \cB_{\overline g_{m}}(A)) \\
& =  \int_{0}^{\tau} e^{\frac{3t}{2}} A^{\frac 32} \left( 4 e^{t} A + 16\pi - e^{-\frac t 2} A^{-\frac 12} (16\pi)^{\frac 32} m \right)^{-\frac 12} dt. 
\end{align*}
By Proposition \ref{prop:imcf-plus-isoper-ineq}, we have that
\begin{equation*}
\sL^{3}_{\overline g_{m}}(\cB_{\overline g_{m}}(e^{\tau}A)) \geq \sL^{3}_{\overline g_{m}}(\Omega_{\tau} \cup\Omega) + o(1)
\end{equation*}
as $\tau \to \infty$. As such, 
\begin{align*}
\sL^{3}_{g}(\Omega_{\tau} \cup\Omega) & \geq \sL^{3}_{g}(\Omega) + \int_{0}^{\tau} e^{\frac{3t}{2}} A^{\frac 32} \left( 4 e^{t} A + 16\pi - e^{-\frac t 2} A^{-\frac 12} (16\pi)^{\frac 32} m \right)^{-\frac 12} dt\\
& = \sL^{3}_{g}(\Omega) + \sL^{3}_{\overline g_{m}} (\cB_{\overline g_{m}}(e^{\tau}A) \setminus \cB_{\overline g_{m}}(A))\\
& = \sL^{3}_{g}(\Omega) - \sL^{3}_{\overline g_{m}}(\cB_{\overline g_{m}}(A))+ \sL^{3}_{\overline g_{m}} (\cB_{\overline g_{m}}(e^{\tau}A))\\
& \geq \sL^{3}_{g}(\Omega) - \sL^{3}_{\overline g_{m}}(\cB_{\overline g_{m}}(A))+ \sL^{3}_{\overline g_{m}} (\Omega_{\tau}\cup \Omega) +o(1).
\end{align*}
The conclusion follows upon letting $\tau \to \infty$, using the fact that $\Omega_{\tau}\cup\Omega$ forms an exhaustion of $(M,g)$, as proven in Proposition \ref{prop:imcf-forms-exhaustion}.
\end{proof}

We will also need bounds similar to the previous proposition when the boundary of $\Omega$ has negative Hawking mass and/or does not contain the horizon. 
\begin{prop}\label{prop:coarse-vol-bds}
Suppose that $\Omega$ is a Borel set of finite perimeter in $(M,g)$ with smooth, connected, outer-minimizing, CMC boundary $\Sigma:=\partial\Omega$. We will write $A : = \cH^{2}_{g}(\Sigma)$. 
If $A \geq 1$ and $m \leq m_{H}(\Sigma)$ satisfies $  -\frac{1}{3} A^{\frac 12}\leq (16\pi)^{\frac 12} m \leq  A^{\frac 12} $, then 
\begin{equation*}
\sL^{3}_{g}(\Omega) \leq \sL^{3}_{\overline g}(\cB_{\overline g}(A)) + C,
\end{equation*}
where $C$ only depends on on $(M,g)$. 
\end{prop}
\begin{proof}
Consider $\Omega$, a Borel set with finite perimeter with smooth, connected, outer-minimizing boundary $\Sigma$ with Hawking mass $  -\frac{1}{3} A^{\frac 12}\leq (16\pi)^{\frac 12} m \leq A^{\frac 12} $. First, we will assume that $\Omega$ contains the horizon. Because $m$ may be negative, we cannot necessarily use the isoperimetric inequality in $(\overline M_{m},\overline g_{m})$, so we will instead compare to a ball in hyperbolic space.

As in the previous proof, if we flow $\Sigma$ by weak inverse mean curvature flow, writing the resulting surface after time $\tau$ as $\Sigma_{\tau}$ and the region between $\Sigma$ and $\Sigma_{\tau}$ as $\Omega_{\tau}$, then Proposition \ref{prop:vol-swept-IMCF-bd} gives
\begin{equation*}
\sL^{3}_{g}(\Omega_{\tau}) \geq \int_{0}^{\tau} e^{\frac{3t}{2}} A^{\frac 32} \left( 4 e^{t} A + 16\pi - e^{-\frac t 2} A^{-\frac 12} (16\pi)^{\frac 32} m \right)^{-\frac 12} dt. 
\end{equation*}
Now, we consider the hyperbolic coordinate ball $\cB_{\overline g}(A)$ in $(\overline M,\overline g)$ and flow its boundary by inverse mean curvature flow, obtaining 
\begin{equation*}
\sL^{3}_{\overline g} (\cB_{\overline g}(e^{\tau}A) \setminus \cB_{\overline g}(A)) =  \int_{0}^{\tau} e^{\frac{3t}{2}} A^{\frac 32} \left( 4 e^{t} A + 16\pi \right)^{-\frac 12} dt. 
\end{equation*}
Proposition \ref{prop:imcf-plus-isoper-ineq} yields
\begin{equation*}
\sL^{3}_{\overline g}(\cB_{\overline g}(e^{\tau}A) ) \geq \sL^{3}_{\overline g}(\Omega_{\tau}\cup \Omega) + o(1)
\end{equation*}
as $\tau\to\infty$. We may combine these facts to obtain
\begin{align*}
& \sL^{3}_{g}(\Omega_{\tau}\cup\Omega) \\
& \geq \sL^{3}_{g}(\Omega) + \int_{0}^{\tau} e^{\frac{3t}{2}} A^{\frac 32} \left( 4 e^{t} A + 16\pi - e^{-\frac t 2} A^{-\frac 12} (16\pi)^{\frac 32}m \right)^{-\frac 12} dt\\
& =  \sL^{3}_{g}(\Omega) + \sL^{3}_{\overline g} (\cB_{\overline g}(e^{\tau}A) \setminus \cB_{\overline g}(A)) \\
& \qquad + \int_{0}^{\tau} e^{\frac{3t}{2}} A^{\frac 32} \left[ \left( 4 e^{t} A + 16\pi - e^{-\frac t 2} A^{-\frac 12} (16\pi)^{\frac 32}m \right)^{-\frac 12} - \left( 4 e^{t} A + 16\pi \right)^{-\frac 12} \right] dt\\
& =  \sL^{3}_{g}(\Omega) - \sL^{3}_{\overline g}( \cB_{\overline g}(A))+ \sL^{3}_{\overline g} (\cB_{\overline g}(e^{\tau}A))  \\
& \qquad + \int_{0}^{\tau} e^{\frac{3t}{2}} A^{\frac 32} \left[ \left( 4 e^{t} A + 16\pi - e^{-\frac t 2} A^{-\frac 12} (16\pi)^{\frac 32}m \right)^{-\frac 12} - \left( 4 e^{t} A + 16\pi \right)^{-\frac 12} \right] dt\\
& \geq  \sL^{3}_{g}(\Omega) - \sL^{3}_{\overline g}( \cB_{\overline g}(A))+ \sL^{3}_{\overline g} (\Omega_{\tau}\cup\Omega)  +o(1) \\
& \qquad + \int_{0}^{\tau} e^{\frac{3t}{2}} A^{\frac 32} \left[ \left( 4 e^{t} A + 16\pi - e^{-\frac t 2} A^{-\frac 12}(16\pi)^{\frac 32} m \right)^{-\frac 12} - \left( 4 e^{t} A + 16\pi \right)^{-\frac 12} \right] dt\\
& \geq  \sL^{3}_{g}(\Omega) - \sL^{3}_{\overline g}( \cB_{\overline g}(A))+ \sL^{3}_{\overline g} (\Omega_{\tau}\cup\Omega)   + o(1)\\
& \qquad + \int_{0}^{\tau} e^{\frac{3t}{2}} A^{\frac 32} \left[ \left( 4 e^{t} A + 16\pi\left( 1 + \frac 1 3 e^{-\frac t 2}\right)  \right)^{-\frac 12} - \left( 4 e^{t} A + 16\pi \right)^{-\frac 12} \right] dt.
\end{align*}
Thus, taking $\tau\to \infty$ yields 
\begin{align*}
& \sL^{3}(\Omega) \\
&  \leq \sL^{3}_{\overline g}(\cB_{\overline g}(A)) + V(M,g)\\
&  + \int_{0}^{\infty} e^{\frac{3t}{2}} A^{\frac 32} \left[  \left( 4 e^{t} A + 16\pi \right)^{-\frac 12} - \left( 4 e^{t} A + 16\pi \left(1 + \frac 1 3 e^{-\frac t 2} \right) \right)^{-\frac 12}\right] dt.
\end{align*}
Finally, it remains to check that the integral is bounded independently of $A$. Clearly, the only thing to check is that this remains bounded as $A$ becomes large. In this regime, we have that
\begin{align*}
& \int_{0}^{\infty} e^{\frac{3t}{2}} A^{\frac 32} \left[  \left( 4 e^{t} A + 16\pi \right)^{-\frac 12} - \left( 4 e^{t} A + 16\pi \left(1 + \frac 1 3 e^{-\frac t 2} \right) \right)^{-\frac 12}\right] dt \\
& = \frac 12 \int_{0}^{\infty} e^{t} A  \left[  \left( 1 + 4\pi e^{-t} A^{-1} \right)^{-\frac 12} - \left( 1 + 4 \pi e^{-t} A^{-1} \left(1 + \frac 1 3 e^{-\frac t 2} \right) \right)^{-\frac 12}\right] dt\\
& = \frac 12 \int_{\log A}^{\infty} e^{t}   \left[  \left( 1 + 4\pi e^{-t} \right)^{-\frac 12} - \left( 1 + 4 \pi e^{-t}  \left(1 + \frac 1 3 e^{-\frac t 2} A^{\frac 12} \right) \right)^{-\frac 12}\right] dt\\
& = \frac 12 \int_{\log A}^{\infty} e^{t}   \left( 1 + 4\pi e^{-t} \right)^{-\frac 12}   \left[ 1 - \left( 1 +  \frac 4 3 \pi e^{-\frac{3t}{2}}   A^{\frac 12} \left( 1 + 4\pi e^{-t} \right)^{-1}    \right)^{-\frac 12}\right] dt\\
& \leq C \int_{\log A}^{\infty} e^{t}   \left( 1 + 4\pi e^{-t} \right)^{-\frac 32} A ^{\frac 12} e^{-\frac{3t}{2}} dt\\
& \leq C A^{\frac 12} \int_{\log A}^{\infty} e^{-\frac t 2} dt \leq C.
\end{align*}
The second to last inequality follows from the fact that for $t \geq \log A$, we have that $e^{-\frac {3t}{2}}A^{\frac 12} \ll 1$. This establishes the claim in the case that $\Omega$ contains the horizon. 

Now, we suppose that $\Sigma$ has Hawking mass $m:= m_{H}(\Sigma)$ satisfying $  -\frac{1}{3} A^{\frac 12}\leq(16\pi)^{\frac 12}  m \leq A^{\frac 12} $ as before, but does not surround the horizon. We may apply Proposition \ref{prop:how-to-jump} to construct a flow $\Sigma_{\tau}$ starting from $\Sigma$ which jumps over the horizon (notice that the horizon has mean curvature $H_{g}\equiv 2$, so we may neglect the third term in the Hawking mass bounds derived there). By the Hawking mass bounds from Proposition \ref{prop:how-to-jump}, we have that $m_{H}(\Sigma_{\tau}) \geq m_{H}(\Sigma) -\delta$ or $m_{H}(\Sigma_{\tau}) \geq C_{1}m_{H}(\Sigma) -\delta$, depending on whether or not $m_{H}(\Sigma) \geq 0$ or not. In either case, we denote by $m'$, the lower bound for $m_{H}(\Sigma_{\tau})$ along the flow with jumps and note that our assumptions imply that there is a constant $C>0$ depending only on $(M,g)$ so that $m'\geq -(16\pi)^{-\frac 12} C A^{\frac 12}$. Clearly, we may assume that $m' \leq (16\pi)^{-\frac 12}A^{\frac 12}$, after shrinking $m'$ if necessarily. 

%
%
%
%

Suppose that the jump occurs at time $T$. For $\tau > T$, we will denote by $\Omega_{\tau}$ the union of $\Omega_{T,-}$ with the region between $\Sigma_{\tau}$ and $\Sigma_{T,+}$. Finally, we define the jump region $J$, to be the region between $\Sigma_{T,-}\cup \partial M$ and $\Sigma_{T,+}$. Thus, for $\tau > T$, the monotonicity of the Hawking mass through the jump combined with the reasoning used above to derive \eqref{eq:vol-cont-ICMF-region-lb} applied before and after the jump yields
\begin{align*}
\sL^{3}_{g}(\Omega_{\tau}) & \geq  \int_{0}^{T} e^{\frac{3t}{2}} A^{\frac 32} \left( 4 e^{t} A + 16\pi - e^{-\frac t 2} A^{-\frac 12} (16\pi)^{\frac 32}m '\right)^{-\frac 12} dt \\
& \qquad + \int_{T}^{\tau} e^{\frac{3}{2}(t+\beta)} A^{\frac 32} \left( 4 e^{t+\beta} A + 16\pi - e^{-\frac 1 2(t+\beta)} A^{-\frac 12} (16\pi)^{\frac 32}m' \right)^{-\frac 12} dt \\
& =   \int_{0}^{T} e^{\frac{3t}{2}} A^{\frac 32} \left( 4 e^{t} A + 16\pi - e^{-\frac t 2} A^{-\frac 12}(16\pi)^{\frac 32} m' \right)^{-\frac 12} dt \\
& \qquad + \int_{T+\beta}^{\tau+\beta} e^{\frac{3 t}{2}} A^{\frac 32} \left( 4 e^{t} A + 16\pi - e^{-\frac t 2} A^{-\frac 12} (16\pi)^{\frac 32}m '\right)^{-\frac 12} dt\\
& =  \int_{0}^{\tau + \beta} e^{\frac{3t}{2}} A^{\frac 32} \left( 4 e^{t} A + 16\pi - e^{-\frac t 2} A^{-\frac 12} (16\pi)^{\frac 32}m' \right)^{-\frac 12} dt \\
& \qquad -  \int_{T}^{T+\beta} e^{\frac{3 t}{2}} A^{\frac 32} \left( 4 e^{t} A + 16\pi - e^{-\frac t 2} A^{-\frac 12}(16\pi)^{\frac 32} m' \right)^{-\frac 12} dt.
\end{align*}
Thus, the argument used above yields (along with $m' \geq - (16\pi)^{-\frac 12} C A^{\frac 12}$)
\begin{align*}
& \sL^{3}_{g}(\Omega) + \sL^{3}_{g}(J) \\
& \leq \sL^{3}_{\overline g}(\cB_{\overline g}(A)) + V(M,g) \\
& \qquad + \int_{0}^{\infty} e^{\frac{3t}{2}} A^{\frac 32} \left[  \left( 4 e^{t} A + 16\pi \right)^{-\frac 12} - \left( 4 e^{t} A + 16\pi \left(1 +C  e^{-\frac t 2} \right) \right)^{-\frac 12}\right] dt\\
& \qquad + \int_{T}^{T+\beta} e^{\frac{3t}{2}} A^{\frac 32}\left( 4e^{t}A + 16\pi - e^{-\frac t 2} A^{-\frac 12} (16\pi)^{\frac 32}m' \right)^{-\frac 12} dt. 
\end{align*}
The same argument as above shows that the first integral is uniformly bounded independently of $A$, so it remains to consider the second integral. By the bound on $\beta$ in Proposition \ref{prop:how-to-jump} and the assumption that $m'\leq (16\pi)^{-\frac 12} A^{\frac 12}$, defining $A_{\partial M} = \cH^{2}_{g}(\partial M)$ we obtain 
\begin{align*}
& \int_{T}^{T+\beta} e^{\frac{3t}{2}} A^{\frac 32}\left( 4e^{t}A + 16\pi - e^{-\frac t 2} A^{-\frac 12} (16 \pi)^{\frac 32}m' \right)^{-\frac 12} dt \\
& \leq \int_{T}^{T+\log\left( 1 +\frac{A_{\partial M}}{A} e^{-T}\right)} e^{\frac {3t}{2}} A^{\frac 32} \left( 4e^{t} A +16\pi \left( 1- e^{-\frac t 2} \right)\right)^{-\frac 12} dt\\
& \leq \frac A2 \int_{T}^{T+\log\left( 1 +\frac{A_{\partial M}}{A} e^{-T}\right)} e^{t}  dt\\
& = \frac 12 A_{\partial M}.
\end{align*}
This is uniformly bounded independently of $A, m$ and $T$, as claimed. 
\end{proof}

Note that combining Corollary \ref{coro:CY-mH-bd} with Propositions \ref{prop:comp-bdry-bds} and \ref{prop:coarse-vol-bds}, yields
\begin{coro}\label{coro:coarse-bds-general-large-iso}
If $\Omega$ is a large, isoperimetric region with $A = \cH^{2}_{g}(\partial^{*}\Omega)$, then 
\begin{equation*}
\sL^{3}_{g}(\Omega) \leq \sL^{3}_{\overline g}(\cB_{\overline g}(A)) + C,
\end{equation*}
where $C$ depends only on $(M,g)$. 
\end{coro}

In fact, we will require a more qualitative version of this result. 

\begin{prop}\label{prop:sequence-large-iso-one-large-rest-small}
For $k\geq 2$, suppose that $\Omega^{(l)}$ is a sequence of isoperimetric regions with exactly $k$ components 
\begin{equation*}
\Omega^{(l)} = \Omega^{(l)}_{1}\cup\dots\cup \Omega^{(l)}_{k},
\end{equation*} 
and so that $\sL^{3}_{g}(\Omega^{(l)})\to\infty$. Define $A^{(l)}_{j}:= \cH^{2}_{g}(\partial^{*}\Omega^{(l)}_{j})$ and choose the ordering of the components so that $A^{(l)}_{1}\geq A^{(l)}_{2}\geq \dots\geq A^{(l)}_{k}> 0$. Then, the regions other than $\Omega^{(l)}_{1}$ have uniformly bounded area, i.e., $A^{(l)}_{2} = O(1)$ as $l\to\infty$.
\end{prop}
\begin{proof}
Suppose otherwise. As such, after extracting a subsequence we may assume that for some $J \in \{2,\dots, k\}$, $A^{(l)}_{j}\to \infty$ for $j \leq J$ and $A^{(l)}_{j}= O(1)$ for $j> J$. Applying Proposition \ref{prop:coarse-vol-bds} to each component of $\Omega^{(l)}$ yields
\begin{equation*}
\sL^{3}_{g}(\Omega^{(l)}) \leq \sum_{j=1}^{k}\sL^{3}_{\overline g}(\cB_{\overline g}(A_{j}^{(l)})) + O(1),
\end{equation*}
as $l \to \infty$. Note that we have used that the number of components, $k$, is fixed, so the $O(1)$ error terms in Proposition \ref{prop:coarse-vol-bds} remain uniformly bounded after summing over $j$. Comparison against a region of the form $\Gamma^{(l)}:=\{s\leq s_{l}\}$ in $(M,g)$ of area $A^{(l)}$ yields
\begin{equation*}
\sL^{3}_{g}(\Gamma^{(l)}) \leq \sL^{3}_{g}(\Omega^{(l)})  \leq \sum_{j=1}^{k}\sL^{3}_{\overline g}(\cB_{\overline g}(A_{j}^{(l)})) + O(1).
\end{equation*}
As such, Lemmas \ref{lemma:vol-large-balls} and \ref{lemm:vol-large-coord-balls-g} combined with our assumptions concerning the behavior of the $A^{(l)}_{j}$, yield
\begin{align*}
& \frac 12 \left( A^{(l)}_{1}+\dots + A^{(l)}_{J}\right) - \pi \log \left( A^{(l)}_{1}+\dots + A^{(l)}_{k}\right)\\
&  =\sL^{3}_{g}(\Gamma^{(l)}) + O(1)\\
&  \leq \sum_{j=1}^{k}\sL^{3}_{\overline g}(\cB_{\overline g}(A_{j}^{(l)})) + O(1)\\
&   =\frac 12 \left( A^{(l)}_{1}+\dots + A^{(l)}_{J}\right)  - \pi \log \left( A^{(l)}_{1}\cdots A^{(l)}_{J}\right) + O(1).
\end{align*}
Rearranging this yields
\begin{equation*}
\log\left( \frac{A^{(l)}_{1}+\dots + A^{(l)}_{k}}{A^{(l)}_{1}\cdots A^{(l)}_{J}} \right) \geq O(1).
\end{equation*}
This is a contradiction because $J\geq 2$, so the quotient is tending to $0$. 
\end{proof}

\section{Proof of Theorem \ref{theo:exist-iso-AH}}\label{sec:pf-exist-in-AH}
 In this section, we prove Theorem \ref{theo:exist-iso-AH}. Throughout this section, $(M,g)$ will be an asymptotically hyperbolic manifold with $R_{g}\geq -6$. We remark that in the case of compact perturbations of Schwarzschild-AdS, case 3 below simplifies slightly, as Theorem \ref{theo:brendle-IHES} prevents small components of a sequence of isoperimetric regions from sliding off to infinity. 
 
 Suppose that the Theorem \ref{theo:exist-iso-AH} is false, i.e., there exists $V^{(l)}\to\infty$ so that applying the generalized existence result in Proposition \ref{prop:conc-comp}, we obtain a non-empty component ``at infinity'' in hyperbolic space. More precisely, we have $S^{(l)}>0$ and $\Omega^{(l)}$ so that 
\begin{itemize}
\item $\Omega^{(l)}$ is an isoperimetric region in $(M,g)$,
\item $\sL_{g}^{3}(\Omega^{(l)}) + \sL^{3}_{\overline g}(\cB_{\overline g}(S^{(l)})) = V^{(l)}$,
\item $\cH^{2}_{g}(\partial^{*}\Omega^{(l)}) + S^{(l)} = A_{g}(V^{(l)})$, and
\item $\Omega^{(l)}$ and $\cB_{\overline g}(S^{(l)})$ have the same mean curvature. 
\end{itemize}
We define $\Sigma^{(l)}:=\partial^{*}\Omega^{(l)}$ and $A^{(l)} : = \cH^{2}_{g}(\Sigma^{(l)})$. We will consider three cases, based on the behavior of $A^{(l)}$ and $S^{(l)}$ as $l\to\infty$. It is not hard to see that we may find a subsequence such that one such case holds for all $l$. \\

\noindent\emph{Case 1, $S^{(l)} = O(1)$ as $l\to\infty$:} In this case, $\sL^{3}_{g}(\Omega^{(l)})\to \infty$, and by Corollary \ref{coro:coarse-bds-general-large-iso}, $\cH^{2}_{g}(\Sigma^{(l)}) \to \infty$ as well. Thus, by Corollary \ref{coro:coarse-CY-mult-bdry}, $H_{g} \to 2$. This cannot happen, because spheres of bounded size in hyperbolic space have mean curvatures much larger than $2$. \\

\noindent\emph{Case 2, $S^{(l)}\to\infty$ and $A^{(l)} \to\infty$ as $l\to\infty$:} Define $\Gamma^{(l)} : = \{s\leq s_{l}\}$ where $s_{l}$ is chosen so that $\cH^{2}_{g}(\partial\Gamma^{(l)}) = A^{(l)}+S^{(l)}$. Now, using Corollary \ref{coro:coarse-bds-general-large-iso} to bound $\sL^{3}_{g}(\Omega^{(l)})$, we obtain 
\begin{align*}
& \frac 12 (A^{(l)} + S^{(l)}) - \pi \log (A^{(l)} + S^{(l)})  \\
& = \sL^{3}_{\overline g}(\Gamma^{(l)}) + O(1)\\
& = \sL^{3}_{g}(\Gamma^{(l)})  + O(1)\\
& \leq \sL^{3}_{g}(\Omega^{(l)}) + \sL^{3}_{\overline g}(\cB_{\overline g}(S^{(l)})) + O(1)\\
& \leq \sL^{3}_{\overline g}(\cB_{\overline g}(A^{(l)})) + \sL^{3}_{\overline g}(\cB_{\overline g}(S^{(l)})) + O(1)\\
& =  \frac 12 (A^{(l)} + S^{(l)}) - \pi \log (A^{(l)}S^{(l)}) + O(1).
\end{align*}
Rearranging yields the following equation:
\begin{equation*}
\log \left( \frac{A^{(l)}+S^{(l)}}{A^{(l)}S^{(l)}}\right) \geq O(1).
\end{equation*}
Because both areas are diverging, this cannot hold. \\

\noindent\emph{Case 3, $A^{(l)} = O(1)$ as $l\to\infty$:} By Proposition \ref{prop:cptness-iso-regions} (and the assumption that the area $A^{(l)}$ is uniformly bounded), after extracting a subsequence, each component of $\Omega^{(l)}$ is either smoothly converging to the horizon region or sliding off to infinity. Write $\Omega^{(l)} = \Omega^{(l)}_{h}\cup\Omega^{(l)}_{d}$, where $\Omega^{(l)}_{h}$ is converging to the horizon and $\Omega^{(l)}_{d}$ is diverging (we allow for the possibility that one or both of these sets are empty\footnote{Note that if $(M,g)$ is a compact perturbation of Schwarzschild-AdS, then $\Omega^{(l)}_{d}$ must necessarily be empty for $l$ sufficiently large, by Theorem \ref{theo:brendle-IHES}}). In particular, we have that $\sL^{3}(\Omega^{(l)}_{h}) = o(1)$ and $A^{(l)} = \cH^{2}_{g}(\partial^{*}\Omega^{(l)}_{h}) = A_{\partial M} + o(1)$. 

Furthermore, because $\Omega^{(l)}_{d}$ is diverging and $A^{(l)}_{d}:=\cH^{2}_{g}(\partial^{*}\Omega_{d}^{(l)})$ is uniformly bounded by assumption (which implies that $\sL^{3}_{g}(\Omega^{(l)}_{d})$ is also uniformly bounded, because $\Omega^{(l)}$ is isoperimetric), we have that
\begin{equation*}
\cH^{2}_{g}(\partial^{*}\Omega_{d}^{(l)}) = \cH^{2}_{\overline g}(\partial^{*}\Omega_{d}^{(l)}) + o(1) \qquad \text{and} \qquad \sL^{3}_{g}(\Omega^{(l)}_{d}) = \sL^{3}_{\overline g}(\Omega^{(l)}_{d}) + o(1),
\end{equation*}
as $l\to\infty$. Hence, we may apply the isoperimetric inequality in hyperbolic space to conclude
\begin{equation*}
\sL^{3}_{g}(\Omega^{(l)}_{d}) \leq \sL^{3}_{\overline g}(\cB_{\overline g}(A^{(l)}_{d})) + o(1),
\end{equation*}
as $l\to\infty$. 

As in case 2, define $\Gamma^{(l)} : = \{s\leq s_{l}\}$ where $s_{l}$ is chosen so that $\cH^{2}_{g}(\partial\Gamma^{(l)}) = A^{(l)}+S^{(l)}$. It is not hard to check that
\begin{equation*}
\cH^{2}_{\overline g}(\Gamma^{(l)}) = A^{(l)} + S^{(l)} + o(1),
\end{equation*}
so 
\begin{equation*}
\sL^{3}_{\overline g}(\Gamma^{(l)}) = \sL^{3}_{\overline g}(\cB_{\overline g}( A^{(l)} + S^{(l)})) + o(1),
\end{equation*}
as $l\to\infty$. Now, comparing the generalized isoperimetric region consisting of $\Omega^{(l)}$ and $\cB_{\overline g}(S^{(l)})$ with $\Gamma^{(l)}$, we obtain
\begin{equation*}
\sL^{3}_{g}(\Gamma^{(l)}) \leq \sL^{3}_{g}(\Omega^{(l)}) + \sL^{3}_{\overline g}(\cB_{\overline g}(S^{(l)})). 
\end{equation*}
This implies that
\begin{align*}
&  \sL^{3}_{\overline g}(\cB_{\overline g}( A^{(l)}  + S^{(l)})) + V(M,g)\\
& = \sL^{3}_{\overline g}(\cB_{\overline g}( A^{(l)}  + S^{(l)})) + \sL^{3}_{g}(\Gamma^{(l)}) - \sL^{3}_{\overline g}(\Gamma^{(l)}) + o(1)\\
 & = \sL^{3}_{\overline g}(\Gamma^{(l)}) + \sL^{3}_{g}(\Gamma^{(l)}) - \sL^{3}_{\overline g}(\Gamma^{(l)}) \\
& \leq \sL^{3}_{g}(\Omega^{(l)}) + \sL^{3}_{\overline g}(\cB_{\overline g}(S^{(l)}))\\
& = \sL^{3}_{g}(\Omega^{(l)}_{d}) + \sL^{3}_{\overline g}(\cB_{\overline g}(S^{(l)}))+o(1)\\
& = \sL^{3}_{\overline g}(\cB_{\overline g}(A^{(l)}_{d})) + \sL^{3}_{\overline g}(\cB_{\overline g}(S^{(l)}))+o(1)\\
& \leq \sL^{3}_{\overline g}(\cB_{\overline g}(A^{(l)}_{d}+S^{(l)}))+o(1).
\end{align*}
In the last line we used the isoperimetric inequality in hyperbolic space. Using Lemma \ref{lemma:vol-large-balls}, we have that
\begin{align*}
& \frac 12 (A^{(l)}+S^{(l)}) - \pi \log(A^{(l)}+S^{(l)}) +V(M,g)\\
&  \leq \frac 12 (A^{(l)}_{d}+S^{(l)}) - \pi \log(A^{(l)}_{d}+S^{(l)}) +o(1),
\end{align*}
as $l\to\infty$. Because $A^{(l)}= A^{(l)}_{d} +A^{(l)}_{h}=A^{(l)}_{d}+A_{\partial M} + o(1)$, we obtain
\begin{equation*}
V(M,g) + \frac 12 A_{\partial M}  \leq o(1),
\end{equation*}
as $l\to\infty$. However, Proposition \ref{prop:vol-comp-H2} implies that $V(M,g) + \frac 12 A_{\partial M} > 0$ (and this quantity does not depend on $l$), so this is a contradiction.

\section{Behavior of large isoperimetric regions}\label{sec:behav-large-iso}

In this section, we will always assume that $(M,g)$ is a compact perturbation of Schwarzschild-AdS of mass $\mm > 0$ and satisfies $R_{g}\geq -6$.

\begin{lemm}
There do not exist properly embedded, totally umbilical, CMC, $H_{g}\equiv 2$ hypersurfaces $\Sigma$ in $(M,g)$.
\end{lemm}
\begin{proof}
We may adapt an argument from \cite[\S 4]{Brendle:warpedCMC}: Suppose that $\Sigma\hookrightarrow (M,g)$ is a properly embedded, totally umbilical CMC, $H_{g}\equiv 2$ hypersurface. First note that the assumption that $\partial M$ is outermost forces $\Sigma$ to be non-compact. Hence, $\Sigma$ must extend into the exterior region, where $g = \overline g_{\mm}$. We will consider $\hat \Sigma : = \Sigma \setminus B$, where $B$ is a sufficiently large centered coordinate ball so that $g=\overline g_{\mm}$ outside of $B$.

The Codazzi equations combined with the CMC and totally umbilic hypotheses imply that $\nu$ is an eigenvector for $\Ric_{\overline g_{\mm}}(\cdot)$ at each point in $\hat \Sigma$ (we are considering $\Ric_{\overline g_{\mm}}(\cdot)$ as a $(1,1)$-tensor). However, one may check (cf.\ \cite[\S 4]{Brendle:warpedCMC}) that the radial direction is a one dimensional eigenspace for $\Ric_{\overline g_{\mm}}(\cdot)$. From this, we see that at each point $\nu$ must be either radial or orthogonal to $\frac{\partial}{\partial s}$. If there is some point on $\hat \Sigma$ so that $\nu$ is radial, then this would continue to hold at all points on the connected component of $\hat \Sigma$ containing that point, so clearly $\hat \Sigma$ would have to be a centered coordinate sphere. This cannot happen, as $\hat \Sigma$ is unbounded. 

On the other hand, if $\nu$ is orthogonal to $\frac{\partial}{\partial s}$ at each point on $\hat \Sigma$, it is easy to check that each component of $\hat \Sigma$ must lie in a plane $P$ in $\RR^{3}$. However, in this case, $\hat \Sigma$ would necessarily have zero mean curvature, a contradiction.
\end{proof}

For a hypersurface $\Sigma$ in $\RR^{3}$, recall that we have defined the \emph{inner radius} of $\Sigma$ by
\begin{equation*}
\underline s(\Sigma) := \inf \left\{ s(x) : x \in \Sigma \right\},
\end{equation*}
where the coordinate $s$ is the one used in the definition of Schwarzschild-AdS in \eqref{eq:schwAdS}. We may turn the previous lemma into an effective inequality for large isoperimetric regions, somewhat in the spirit of the usual philosophy that ``a Bernstein-type theorem implies a curvature bound,'' cf.\ \cite[p. 27]{White:PCMI}. 
\begin{lemm}\label{lemm:effective-no-unbil-R-6-surf}
For all $S_{0} > 0$, there exists $\lambda = \lambda(S_{0}) > 0$ so that if $\Omega_{1}$ is a connected component of some compact isoperimetric region $\Omega$ and $\partial^{*}\Omega_{1}$ satisfies $\underline s(\partial^{*}\Omega_{1}) \leq S_{0}$ and $\cH^{2}_{g}(\partial^{*}\Omega_{1}) \geq \lambda^{-1}$, then
\begin{equation*}
\int_{\partial^{*}\Omega_{1}} \left( R_{g} + 6 + |\tfsff|^{2} \right)d\cH^{2}_{g} \geq \lambda > 0 .
\end{equation*}
\end{lemm}

\begin{proof}
Suppose that for some $S_{0}$, we could find a sequence of isoperimetric regions $\Omega^{(l)}$ so that some connected component $\Omega_{1}^{(l)}\subset \Omega^{(l)}$ satisfies $\underline s(\Omega_{1}^{(l)})\leq S_{0}$, 
\begin{equation*}
\int_{\partial^{*}\Omega_{1}^{(l)}} \left( R_{g} + 6 + |\tfsff|^{2} \right)d\cH^{2}_{g} \to 0,
\end{equation*}
and $\cH^{2}_{g}(\partial^{*}\Omega_{1}^{(l)}) \to \infty$. By Proposition \ref{prop:cptness-iso-regions}, we may take the limit of $\Omega^{(l)}$ and $\Omega_{1}^{(l)}$ as sets of finite perimeter, to obtain $\Omega_{1}$, a (possibly disconnected) subset of a locally isoperimetric region $\Omega$ in $(M,g)$. 

Because $\underline s(\Omega_{1}^{(l)})\leq S_{0}$ and $\cH^{2}_{g}(\partial^{*}\Omega_{1}^{(l)}) \to \infty$, we claim that it must hold that $\partial^{*}\Omega_{1}$ is non-empty and contains at least one non-compact component. If this were false, then $\Omega_{1}$ would necessarily be either equal to the horizon region $K$ or empty, by Proposition \ref{prop:cptness-iso-regions}. Either possibility would contradict the isoperimetric property of $\Omega^{(l)}$ as follows: By the co-area formula (cf.\ the proof of \cite[Theorem 2.1]{RitoreRosales}) we may find a sequence of radii $r^{(l)}\to\infty$ so that 
\begin{align*}
\sL^{3}_{g}(\Omega^{(l)}_{1} \cap B_{\overline g}(0;r^{(l)})) & \to 0,\\
\cH^{2}_{g}(\Omega^{(l)}_{1} \cap \partial B_{\overline g}(0;r^{(l)})) &  \to 0.
\end{align*}
On the other hand, because $r^{(l)}\to\infty$, the mean curvature of $\partial^{*}\Omega^{(l)}$ is close to $2$, and $\underline s(\partial^{*}\Omega^{(l)}_{1}) \leq S_{0}$, we may apply the monotonicity formula inside of a sequence of small balls to see that 
\begin{equation*}
\cH^{2}_{g}(\partial^{*}\Omega^{(l)}_{1}\cap B_{\overline g}(0;r^{(l)}))  \to \infty. 
\end{equation*}
Putting these facts together, we see that the following region contains the same volume with less area as compared to $\Omega^{(l)}_{1}$
\begin{equation*}
(\Omega^{(l)}_{1} \setminus B_{\overline g}(0;r^{(l)}))  \cup B^{(l)} \cup K
\end{equation*}
(if $\Omega^{(l)}_{1}$ does not contain the horizon, then $K$ should be omitted from this expression). Here, $B^{(l)}$ is a small coordinate ball near infinity which is chosen to replace the lost volume, i.e., $ \sL^{3}_{g} (B^{(l)}) = \sL^{3}_{g}(\Omega^{(l)}_{1} \cap B_{\overline g}(0;r^{(l)}))$. Hence, $\Omega_{1}^{(l)}$ cannot disappear in the limit. 

By Proposition \ref{prop:cptness-iso-regions}, $\partial^{*}\Omega^{(l)}_{1}$ actually tends to $\partial^{*}\Omega_{1}$ locally smoothly and $\partial^{*}\Omega$ is properly embedded. Because the integrand $R_{g} + 6 + |\tfsff|^{2}$ is non-negative, we may conclude from the smooth convergence that
\begin{equation*}
\int_{\partial^{*}\Omega} \left( R_{g} + 6 + |\tfsff|^{2} \right)d\cH^{2}_{g} = 0.
\end{equation*}
Because $R_{g}\geq -6$, we see that $\partial^{*}\Omega$ is a properly embedded, totally umbilical $H_{g} \equiv 2$ surface, 
contradicting the previous lemma. 
\end{proof}

The following proposition is the crucial step in our understanding of large isoperimetric regions. 
\begin{prop}\label{prop:behavior-large-iso}
There exists $A_{0}>0$ and $C_{0}>0$ so that if $\Omega$ is an isoperimetric region in $(M,g)$ with $\cH^{2}(\partial^{*}\Omega) \geq A_{0}$ then either 
\begin{enumerate}
\item The region $\Omega$ is a centered coordinate ball $\cB_{g}(A)$, or
\item we may write $\Omega = \Omega_{1}\cup \Omega_{2}$, where each $\Omega_{1},\Omega_{2}$ is connected and $\Omega_{2}$ is possibly empty. The boundary of the first region, $\partial^{*}\Omega_{1}$, has non-zero genus, and bounded Hawking mass $m_{H}(\partial^{*}\Omega_{1}) \leq 4\mm$. Moreover, the second region satisfies
\begin{equation*}
\cH^{2}_{g}(\partial^{*}\Omega_{2}) \leq C_{0}.
\end{equation*}
\end{enumerate}
\end{prop}

We will split the proof into two cases: In case 1, we consider large connected isoperimetric regions. Then, in case 2, we discuss isoperimetric regions with multiple components.

\begin{proof}[Proof of Proposition \ref{prop:behavior-large-iso} in Case 1] 
We assume that $\Omega^{(l)}$ is a sequence of connected isoperimetric regions with $\cH^{2}_{g}(\partial^{*}\Omega^{(l)})\to \infty$. We remark that by definition, $\Omega^{(l)}$ contains the horizon (if the horizon region is non-empty). Denote $\Sigma^{(l)} : = \partial^{*}\Omega^{(l)}$. We claim that $m_{H}(\Sigma^{(l)})\leq 4\mm$ for $l$ sufficiently large, so we may assume that $m^{(l)}: = m_{H}(\Sigma^{(l)}) > 4\mm$.

Letting $A^{(l)} = \cH^{2}_{g}(\Sigma^{(l)}$), Proposition \ref{prop:imcf-bds-good-case} implies that
\begin{equation*}
\sL^{3}_{g}(\Omega^{(l)}) \leq \sL^{3}_{\overline g_{m^{(l)}}}(\cB_{\overline g_{m}}(A^{(l)})) + V(M,g) - V(\overline M_{m^{(l)}},\overline g_{m^{(l)}}).
\end{equation*}
Because $\Omega^{(l)}$ is isoperimetric, it must contain more volume than $\cB_{g}(A^{(l)})$. Thus, using Lemma \ref{lemm:vol-large-coord-balls-g} we see that (using that $m^{(l)}\leq (A^{(l)})^{\frac 12}$, by the definition of the Hawking mass and the outermost assumption on $\partial M$)
\begin{align*}
&  \frac 12  A^{(l)}  -  \pi \log A^{(l)} + \left( V( M,g) + \pi(1+\log \pi) \right)  - 8\pi^{\frac 32} \mm (A^{(l)})^{-\frac 12} \\
 & = \sL^{3}_{g}(\cB_{g}(A^{(l)}))  + O (A^{(l)})^{-1})\\
 & \leq \sL^{3}_{g}(\Omega^{(l)}) + O (A^{(l)})^{-1}) \\
& \leq  \frac 12 A^{(l)} - \pi \log A^{(l)} + \left( V( M,g) + \pi(1+\log \pi) \right) \\
& \qquad - 8 \pi^{\frac 32}  m^{(l)} (A^{(l)})^{-\frac 12} + E(m^{(l)},A^{(l)}) + O (A^{(l)})^{-1})\\
& \leq  \frac 12 A^{(l)} - \pi \log A^{(l)} + \left( V( M,g) + \pi(1+\log \pi) \right) \\
& \qquad - 8 \pi^{\frac 32}  m^{(l)} (A^{(l)})^{-\frac 12} + C((A^{(l)})^{-1}).
\end{align*}
Because the leading order terms agree and $C$ does not depend on $m^{(l)}$ or $A^{(l)}$, we conclude that $m^{(l)}\leq \mm + O((A^{(l)})^{-\frac 1 2}) \leq 4\mm$, for $l$ sufficiently large. This is a contradiction, so we thus obtain the claimed Hawking mass bounds.

Now, we claim that if $\Sigma^{(l)}$ has genus zero, then for sufficiently large $l$, it must be a centered coordinate sphere. In the genus zero case, Proposition \ref{prop:christodoulou-yau} implies that
\begin{align*}
\int_{\Sigma^{(l)}} \left( R_{g} + 6 + |\tfsff|^{2}\right) d\cH^{2}_{g} & \leq \frac 3 2(16\pi)^{\frac 32} (A^{(l)})^{-\frac 12} m^{(l)}\\
& \leq 6 (16\pi)^{\frac 32} (A^{(l)})^{-\frac 12} \mm\\
& \leq  O((A^{(l)})^{-\frac 12}).
\end{align*}
This contradicts Lemma \ref{lemm:effective-no-unbil-R-6-surf} unless $\underline s(\Sigma^{(l)}) \to \infty$ as $l\to\infty$. If this happens, then Theorem \ref{theo:brendle-IHES} would imply that $\Sigma^{(l)}$ must necessarily be a coordinate sphere. 

To sum up, in the case that $\Omega^{(l)}$ is connected for all $l$, we have shown that for sufficiently large $l$: 
\begin{itemize}
\item If $\Sigma^{(l)}$ has genus zero then it must be a centered coordinate sphere.
\item In general, we have the Hawking mass bound $m_{H}(\Sigma^{(l)}) \leq 2\mm$.
\end{itemize}
This finishes the proof of case (1) of the proposition.
\end{proof}

\begin{proof}[Proof of Proposition \ref{prop:behavior-large-iso} in Case 2] 
Suppose $\Omega^{(l)}$ is a sequence of isoperimetric regions with  $\cH^{2}_{g}(\partial^{*}\Omega^{(l)}) \to \infty$ as $l\to\infty$ and so that $\Omega^{(l)}$ has more than one component. We will show that for $l$ sufficiently large, $\Omega^{(l)}$ consists of two regions: one large region whose boundary has non-zero genus and bounded Hawking mass, and one small region which is converging to the horizon.

By Proposition \ref{prop:comp-bdry-bds} (which says that the number of components of an isoperimetric region is uniformly bounded by some number $n_{0}$), we may extract a subsequence (still labeled by $l$) so that each $\Omega^{(l)}$ has exactly $k$ boundary components, where $1< k\leq n_{0}$. Define $\Sigma_{j}^{(l)} : = \partial^{*}\Omega^{(l)}_{j}$ and $A^{(l)}_{j}:= \cH^{2}_{g}(\Sigma_{j}^{(l)})$. We will always choose the ordering of the components so that $A^{(l)}_{1}\geq A^{(l)}_{2}\geq \dots\geq A^{(l)}_{k}> 0$. We will denote $A^{(l)} := A^{(l)}_{1}+\dots+A^{(l)}_{k}$. By Proposition \ref{prop:sequence-large-iso-one-large-rest-small}, we have that $A^{(l)}_{2}=O(1)$ as $l\to\infty$. 
 
From this, we see that as $l\to\infty$, each of $\Omega_{2}^{(l)},\dots,\Omega_{k}^{(l)}$ must either slide off to infinity or converge to the horizon region as sets of finite perimeter (and thus smoothly). This is because they cannot disappear (by the monotonicity formula, they will always have a definite amount of boundary area, and thus if their volume shrinks away to zero, it would be more optimal to enlarge one of the other components slightly). They also cannot converge to some other Borel set of finite perimeter, because Corollary \ref{coro:coarse-CY-mult-bdry} implies that this region would have a closed hypersurface of constant mean curvature $H_{g}=2$ as its boundary, contradicting the outermost assumption of $\partial M$. If any region slides off to infinity, Theorem \ref{theo:brendle-IHES} implies that it is a slice (and thus there can only be one component of $\Omega$), a contradiction. Thus, for $l$ sufficiently large, it must hold that $k=2$ and $\Omega^{(l)}$ is composed a large region $\Omega^{(l)}_{1}$ and a region $\Omega^{(l)}_{2}$ converging to the horizon. 

As such, $\cH^{2}_{g}(\Sigma^{(l)}_{2}) = A_{\partial M} + o(1)$ and $\sL^{3}_{g}(\Omega^{(l)}_{2}) = o(1)$ as $l\to\infty$. We claim that $ m_{H}(\Sigma_{1}^{(l)}) \leq 4\mm$ for $l$ sufficiently large. If this fails, then we may extract a subsequence with $m_{H}(\Sigma_{1}^{(l)}) > 4\mm$ for all $l$. We claim that this yields a contradiction, via an argument along similar lines to Case (1) above. However, there is an additional complication because $\Omega^{(l)}_{1}$ might not contain the horizon, so we must use the inverse mean curvature flow with jumps. Furthermore, we must be careful to avoid errors in the resulting volume bound which are worse that $o(A^{-\frac 12})$, because we are interested in the $A^{-\frac 12}$ order term in the expansion (which is where the mass terms arise). As such, we give the argument below.

Using Proposition \ref{prop:how-to-jump} we construct $(\Sigma^{(l)}_{1})_{\tau}$, an inverse mean curvature flow with a jump over $\Omega^{(l)}_{2}$, starting at $\Sigma_{1}^{(l)}$. We may arrange that the Hawking mass bound $m_{H}((\Sigma_{1}^{(l)})_{\tau}) \geq 2 \mm$ holds for all $\tau \geq 0$. This is a consequence of the fact that we have the following bound for the final term in the Hawking mass bounds from Proposition \ref{prop:how-to-jump} (note that $\partial\cJ$ in Proposition \ref{prop:how-to-jump} is now $\Sigma_{2}^{(l)}$, which is converging to the horizon):
\begin{align*}
& C_{2}\cH^{2}_{g}(\Sigma_{1}^{(l)})^{\frac 12} \int_{\partial\cJ} (H_{g}^{2}-4)d\cH^{2} \\
& = C_{2} \cH^{2}_{g}(\Sigma_{1}^{(l)})^{\frac 12}\cH^{2}_{g}(\partial \cJ)(H_{g}^{2}-4)\\
& \leq C_{2}(A_{\partial M} + o(1)) \cH^{2}_{g}(\Sigma_{1}^{(l)})^{-\frac 12}\cH^{2}_{g}(\Sigma_{1}^{(l)})(H^{2}_{g}-4)\\
& \leq C_{2} (A_{\partial M} + o(1)) \cH^{2}_{g}(\Sigma_{1}^{(l)})^{-\frac 12} 16\pi \\
& \leq o(1).
\end{align*}
Here, we have used that $C_{2}$ from Proposition \ref{prop:how-to-jump} is uniformly bounded: $H_{g}$ and $\cH^{2}_{g}(\Sigma_{2}^{(l)})$ are uniformly bounded, and that $\Sigma_{1}^{(l)}$ cannot be disjoint from the perturbed region, by Theorem \ref{theo:brendle-IHES}. Furthermore, we have used the assumed positivity of $m_{H}(\Sigma_{1}^{(l)})$.

Now, we repeat the argument used in Proposition \ref{prop:coarse-vol-bds} (in particular, keeping track of the volume change over the jump). Suppose that the flow we have just constructed jumps over $\Sigma_{2}^{(l)}$ at time $T^{(l)}$. Write the surface before the jump as $\Sigma_{T^{(l)},-}^{(l)} = \partial^{*}\Omega_{T^{(l)},-}^{(l)}$ and the surface after the jump as $\Sigma_{T^{(l)},+}^{(l)}$. For $\tau > T^{(l)}$ denote $\Omega_{\tau}^{(l)}$ by the union of $\Omega_{T^{(l)},-}^{(l)}$ and the region between $\Sigma_{\tau}^{(l)}$ and $\Sigma_{T^{(l)},+}^{(l)}$. Furthermore, we define the jump region $J^{(l)}$ to be the region between $\Sigma_{T^{(l)},-}^{(l)} \cup \Sigma_{2}^{(l)}$ and $\Sigma_{T^{(l)},+}^{(l)}$. By the Hawking mass bound $m_{H}((\Sigma_{1}^{(l)})_{\tau}) \geq 2\mm$ and Proposition \ref{prop:vol-swept-IMCF-bd}, we have the following inequality for $\tau > T^{(l)}$,
\begin{align*}
& \sL^{3}_{g}(\Omega_{\tau}^{(l)}) \\
& \geq \int_{0}^{\tau+\beta^{(l)}} e^{\frac {3t}{2}}(A^{(l)}_{1})^{\frac 32}\left( 4 e^{t}A_{1}^{(l)} +16\pi - e^{-\frac t 2}(A_{1}^{(l)})^{-\frac 12}(16\pi)^{\frac 3 2}2\mm\right)^{-\frac 12 }dt\\
& \qquad - \int_{T^{(l)}}^{T^{(l)}+\beta^{(l)}} e^{\frac {3t}{2}}(A^{(l)}_{1})^{\frac 32}\left( 4 e^{t}A_{1}^{(l)} +16\pi -e^{-\frac t 2}(A_{1}^{(l)})^{-\frac 12}(16\pi)^{\frac 3 2}2\mm\right)^{-\frac 12 }dt.
\end{align*}
Recall that $\beta^{(l)} \geq 0$ is chosen so that $\cH^{2}_{g}((\Sigma_{1}^{(l)})_{\tau}) = e^{\tau + \beta^{(l)}}\cH^{2}_{g}(\Sigma_{1}^{(l)})$. Rearranging this and letting $\tau \to \infty$ as in Proposition \ref{prop:coarse-vol-bds}, we obtain
\begin{align*}
& \sL^{3}_{g}(\Omega^{(l)})   + \sL^{3}_{g}(J^{(l)}) \\
& \leq \sL^{3}_{\overline g_{2\mm}}(\cB_{\overline g_{2\mm}}(A_{1}^{(l)})) + V(M,g) - V(\overline M_{2\mm},\overline g_{2\mm})\\
& \qquad + \int_{T^{(l)}}^{T^{(l)}+\beta^{(l)}} e^{\frac {3t}{2}}(A^{(l)}_{1})^{\frac 32}\left( 4 e^{t}A_{1}^{(l)} +16\pi -e^{-\frac t 2}(A_{1}^{(l)})^{-\frac 12}(16\pi)^{\frac 3 2}2\mm\right)^{-\frac 12 }dt.
\end{align*}
For large $l$, we have that $2\mm \leq (16\pi)^{-\frac 12} (A^{(l)}_{1})^{\frac 12}$, and Proposition \ref{prop:how-to-jump} yields the bound
\begin{equation*}
\beta^{(l)}\leq \log \left( 1 + \frac{A_{2}^{(l)}}{A_{1}^{(l)}} e^{-T^{(l)}}\right).
\end{equation*}
Hence, we may bound the integral in the preceding expression as follows
\begin{align*}
& \int_{T^{(l)}}^{T^{(l)}+\beta^{(l)}} e^{\frac {3t}{2}}(A^{(l)}_{1})^{\frac 32}\left( 4 e^{t}A^{(l)} +16\pi -e^{-\frac t 2}(A_{1}^{(l)})^{-\frac 12}(16\pi)^{\frac 3 2}2\mm\right)^{-\frac 12 }dt\\
& \leq \int_{T^{(l)}}^{T^{(l)}+\beta^{(l)}} e^{\frac {3t}{2}}(A^{(l)}_{1})^{\frac 32}\left( 4 e^{t}A^{(l)} +16\pi (1-e^{-\frac t 2})\right)^{-\frac 12 }dt\\
& \leq \frac{A_{1}^{(l)}}{2}\int_{T^{(l)}}^{T^{(l)}+\beta^{(l)}}e^{t} dt\\
& \leq \frac 12 A_{2}^{(l)}.
\end{align*}
Thus, we have shown that
\begin{equation*}
\sL^{3}_{g}(\Omega^{(l)}) \leq \sL^{3}_{\overline g_{2\mm}}(\cB_{\overline g_{2\mm}}(A^{(l)}_{1})) + V(M,g) - V(\overline M_{2\mm}, {\overline g_{2\mm}}) + \frac 12 A^{(l)}_{2}.
\end{equation*}
Comparison with $\cB_{g}(A^{(l)})$ yields
\begin{align*}
& \frac 12 A^{(l)} - \pi \log A^{(l)} + V(M,g) - 8 \pi^{\frac 32}\mm (A^{(l)})^{-\frac 12} + O((A^{(l)})^{-1})\\
&  \leq \frac 12 A^{(l)} - \pi \log A^{(l)}_{1} + V(M,g) - 8\pi^{\frac 32} (2\mm) (A^{(l)}_{1})^{-\frac 12} + O((A^{(l)}_{1})^{-1}).
\end{align*}
Note that $\log \frac{A^{(l)}}{A^{(l)}_{1}} = O((A^{(l)}_{1})^{-1})$ and $(A^{(l})^{-\frac 12} = (A^{(l)}_{1})^{-\frac12} + O((A_{1}^{(l)})^{-\frac 32})$. Thus, comparing the coefficients of the order $-\frac 12$ in this expression yields a contradiction. Thus, we have shown that $m_{H}(\Sigma_{1}^{(l)})\leq 4\mm$ for $l$ sufficiently large.

To conclude that $\genus(\Sigma_{1}^{(l)}) > 0$, we may argue exactly as in case (1): in the genus zero case, Proposition \ref{prop:christodoulou-yau} would combine with these Hawking mass bounds to contradict Lemma \ref{lemm:effective-no-unbil-R-6-surf} (we know that $\underline s(\Sigma_{1}^{(l)})$ is uniformly bounded, as if it becomes large, then $\Sigma_{1}^{(l)}$ must be a coordinate sphere, and there cannot be any other components outside of it, by Theorem \ref{theo:brendle-IHES}, contradicting our assumption that there are two components).
\end{proof}

\section[Proof of the main theorem]{Proof of Theorem \ref{theo:main-theo}}\label{sec:proof-main-theo}
In this section, we give the proof of Theorem \ref{theo:main-theo}, namely we will assume that $(M,g)$ is a compact perturbation of Schwarzschild-AdS with $R_{g}\geq -6$ and show that large isoperimetric regions must agree with $\cB_{g}(A)$. 

By Proposition \ref{prop:behavior-large-iso}, it is sufficient to rule out the possibility of large isoperimetric regions with a component having large volume and nonzero genus (possibly with several other components of uniformly bounded volume). 

It is convenient to work with the following version of the isoperimetric profile
\begin{equation*}
V_{g}(A) : = \sup \left\{ \sL^{3}_{g}(\Omega) \ \ :
\begin{array}{ c }
 \text{$\Omega$ is a finite perimeter Borel set in $\RR^{3}$}\\
  \text{ containing the horizon with }  \cH^{2}_{g}(\partial^{*} \Omega ) = A
 \end{array}
 \right\}.
\end{equation*}
Using Lemma \ref{lemm:iso-prof-strict-increase}, is not hard to show that $V_{g}(A)$ is absolutely continuous and strictly increasing. Furthermore, if $\Omega_{A}$ is an isoperimetric region with $\partial^{*}\Omega$ having area $A$ and mean curvature $H_{A}$, then $V_{g}(A)$ has one sided derivatives at $A$ in both directions and
\begin{equation*}
V_{g}'(A)_{-} \leq H_{A}^{-1} \leq V_{g}'(A)_{+}.
\end{equation*}
This is proven in an identical manner to the same fact for $A_{g}(V)$, cf. \cite[Theorem 3]{Bray:thesis}.

\begin{lemm}\label{lem:large-bad-iso-reg-bds-profile}
For sufficiently large $A$, if $\Omega$ is an isoperimetric region of area $A$ which is not $\cB_{g}(A)$, then we have that
\begin{equation*}
-\frac{d}{dA} \left[ V_{g}'(A)^{-2} \right] \geq 24\pi A^{-2}
\end{equation*}
in the barrier sense at $A$.
\end{lemm}
\begin{proof}
By Proposition \ref{prop:behavior-large-iso}, there exists $c>0$ with the following property: For $A$ sufficiently large, if $\Omega$ is an isoperimetric region of area $A$ which is not $\cB_{g}(A)$, then writing $\Omega$ as the disjoint union of connected components, either $\Omega = \Omega_{1}$ or $\Omega = \Omega_{1}\cup \Omega_{2}$ and $\Sigma_{j} :=\partial^{*}\Omega_{j}$, we have that
\begin{enumerate}
\item $\cH^{2}_{g}(\Sigma_{1}) \geq A - c$,
\item $\genus(\Sigma_{1}) > 0$,
\item and $m_{H}(\Sigma_{1}) \leq 4\mm$.
\end{enumerate}
Considering a variation of $\Omega$, which flows $\Sigma_{1}$ outward at unit speed, we have the inequality
\begin{align*}
& 2 V''_{g}(A) V'(A)^{-3} (A-c)^{2} \\
& \geq 2 \int_{\Sigma_{1}} \left( |\sff|^{2} + \Ric(\nu,\nu) \right) d\cH^{2}_{g}\\
& = \int_{\Sigma_{1}} \left( R_{g} + 6 + |\tfsff|^{2}\right)d\cH^{2}_{g} + 24\pi - 4\pi \chi(\Sigma_{1}) - \frac 32 \cH^{2}_{g}(\Sigma_{1})^{-\frac 12} (16\pi)^{\frac 32} m_{H}(\Sigma_{1})\\
& \geq 24\pi + \int_{\Sigma_{1}} \left( R_{g} + 6 + |\tfsff|^{2}\right)d\cH^{2}_{g} - 6 (A-c)^{-\frac 12} (16\pi)^{\frac 32} \mm
\end{align*}
in the barrier sense at $A$. By Brendle's Alexandrov Theorem (Theorem \ref{theo:brendle-IHES}), $\underline s(\Sigma_{1})$ must be uniformly (independently of $A$) bounded from above. Thus, we may use Lemma \ref{lemm:effective-no-unbil-R-6-surf} to find $\lambda >0$ so that 
\begin{equation*}
 \int_{\Sigma_{1}} \left( R_{g} + 6 + |\tfsff|^{2}\right)d\cH^{2}_{g} \geq \lambda > 0 .
\end{equation*}
Taking $A$ even larger if necessary, we may absorb the error terms (which are all $o(1)$ as $A \to\infty$) into the good term $\lambda$ to establish the claim. 
\end{proof}

\begin{rema}\label{rema:V-cvx}
A similar argument shows that $V_{g}(A)$ is convex for $A$ sufficiently large. We will use this observation below. 
\end{rema}

\begin{prop}\label{prop:some-balls-are-isop}
There exists a sequence of areas $A_{k}\to \infty$ so that $\cB_{g}(A_{k})$ is uniquely isoperimetric. 
\end{prop}
\begin{proof}
Suppose otherwise. By Lemma \ref{lem:large-bad-iso-reg-bds-profile}, for some $A_{0} >0$, if $A > A_{0}$ then
\begin{equation}\label{eq:some-balls-iso-main-ineq}
-\frac{d}{dA} \left[ V_{g}'(A)^{-2} \right] \geq 24\pi A^{-2}
\end{equation}
in the barrier sense. First, let us assume that this holds in the classical sense. Then, we may integrate this from $A$ to $\infty$. Using the fact that the mean curvature of large isoperimetric regions tends to $2$, we see that 
\begin{equation*}
V'_{g}(A)^{-2} - 4 \geq 24 \pi A^{-1} .
\end{equation*}
We may rearrange this to yield
\begin{equation*}
V'_{g}(A) \leq \frac 12 - \frac 32 \pi A^{-1} + O(A^{-2}).
\end{equation*}
Integrating this, we obtain
\begin{equation*}
V_{g}(A) \leq \frac 12 A - \frac 32 \pi \log A +O(1).
\end{equation*}
This contradicts Lemma \ref{lemm:vol-large-coord-balls-g}, because for large enough $A$, the region $\cB_{g}(A)$ contains more volume than this would allow. 

In general, the inequality will only hold in the barrier sense, so we need to justify the previous computation. We will follow\footnote{We remark that an alternative method to justify the argument would use the Alexandrov theorem for convex functions, relating the Alexandrov second derivative with the distributional derivative, see \cite[\S 6]{EvansGariepy}. In some sense, this amounts to replacing the finite difference operators with mollifiers.} the argument used in \cite[Lemma 1]{Bray:thesis}. First, we rearrange \eqref{eq:some-balls-iso-main-ineq} to see that
\begin{equation*}
\frac{d}{dA} \left[ V_{g}'(A)^{-2} - 4 - 24\pi A^{-1} \right] \leq 0
\end{equation*}
which still only holds in the barrier sense for $A > A_{0}$. We claim that this holds in the distributional sense for $A> A_{0}$, i.e., for $\varphi \in C^{\infty}_{c}((A_{0},\infty))$ an arbitrary non-negative test function then
\begin{equation*}
\int_{A_{0}}^{\infty}  \left[ V_{g}'(A)^{-2} - 4 - 24\pi A^{-1} \right] \varphi'(A)dA \geq 0. 
\end{equation*}
We remark that $V_{g}'(A)$ is well defined for a.e.\ $A$, so this expression makes sense. Let us define the finite difference operator $D_{\delta}$ by 
\begin{equation*}
D_{\delta} f(x) : = \frac 1 \delta \left( f(x+\delta) - f(x) \right).
\end{equation*}
Then,
\begin{align*}
& \int_{A_{0}}^{\infty}  \left[ V_{g}'(A)^{-2} - 4 - 24\pi A^{-1} \right] \varphi'(A)dA \\
& = \lim_{\delta\to 0} \int_{A_{0}}^{\infty}  \left[ ((D_{\delta} V_{g})(A))^{-2} - 4 - 24\pi A^{-1} \right] (D_{\delta}\varphi(A))dA \\
& = \lim_{\delta\to 0} \int_{A_{0}}^{\infty} D_{-\delta} \left[ (D_{\delta} V_{g}(A))^{-2} - 4 - 24\pi A^{-1} \right] \varphi(A)dA .
\end{align*}
The final step follows from ``integration by parts'' for the $D_{\delta}$ operator which is actually just a change of variables. Now, for any $ \hat A \in (A_{0},\infty)$, we shown that there exists a comparison function $f_{\hat A}(A)$ satisfying $f_{\hat A}(\hat A + \delta) \leq V_{g}(\hat A + \delta)$ for $|\delta|$ small and so that $f_{\hat A}(\hat A) = V_{g}(\hat A)$. Using this and the fact that $V_{g}(A)$ and $f_{\hat A}(V)$ are increasing, it follows that 
\begin{equation*}
D_{-\delta} ((D_{\delta} V_{g})(\hat A))^{-2} \geq  D_{-\delta}  ((D_{\delta} f_{\hat A}) (A) )^{-2}  |_{A = \hat A}.
\end{equation*}
Thus, applying this inequality in the above integral (changing the variable of integration to $\hat A$) yields
\begin{align*}
 \int_{A_{0}}^{\infty} &  \left[ V_{g}'(A)^{-2} - 4 - 24\pi A^{-1} \right] \varphi'(A)dA \\
& \geq \lim_{\delta\to 0} \int_{A_{0}}^{\infty}   \left[D_{-\delta} ((D_{\delta} f_{\hat A})(A))^{-2}|_{A = \hat A} + 24\pi \hat A^{-2} \right] \varphi(\hat A)d\hat A\\
& =  \int_{A_{0}}^{\infty}   \left[ \frac{d}{d A} \left[ ( f_{\hat A}'(A))^{-2} \right] \Big |_{A = \hat A} + 24\pi \hat A^{-2} \right] \varphi(\hat A)d \hat A\\
& \geq 0.
\end{align*}
In the last line, we used that inequality holds in the barrier sense. Thus, the above inequality holds also in the distributional sense. Now, a simple approximation argument shows that we may plug in
\begin{equation*}
\varphi_{\epsilon}(x) : = \begin{cases}
0 & x \leq A\\
\frac 1 \epsilon (x-A) & A < x < A + \epsilon\\
1 & x > A +\epsilon
\end{cases}
\end{equation*}
as a test function into the distributional inequality (the non-smooth points are easily approximated, while the lack of compact support is not an issue, because the following essential limit holds: $V_{g}'(A)^{-2} - 4 - 24\pi A^{-1} \to 0$ as $A \to \infty$, by the observation that $H_{g}\to 2$ for large isoperimetric regions). From this, we have that
\begin{equation*}
\frac 1 \epsilon \int_{A}^{A+\epsilon} \left[ V_{g}'(\hat A)^{-2} - 4 - 24\pi \hat A^{-1} \right] d\hat A \geq 0.
\end{equation*}
Thus, if $A$ is a point of differentiability and a Lebesgue point of $V_{g}'(A)^{-2}$ (this holds for a.e.\ $A$ because $V_{g}(A)$ is convex for large enough $A$, by a second variation argument argument as in Lemma \ref{lem:large-bad-iso-reg-bds-profile}, and $V'_{g}(A)^{-2}$ is easily seen to be in $L^{1}_{loc}$), we may pass to the limit as $\epsilon \to 0$. Thus we have shown that 
\begin{equation*}
V_{g}'( A)^{-2} - 4 - 24\pi A^{-1} \geq 0
\end{equation*}
for a.e.\ $A > A_{0}$. We may rearrange this as above to obtain an upper bound on $V'_{g}(A)$ for a.e. $A > A_{0}$. By absolute continuity of $V_{g}(A)$, we may now complete the argument as above. 
\end{proof}

Now, we may finish the proof of the main theorem. Define 
\begin{equation*}
\sA : = \{ A > A_{0} : \cB_{g}(A) \text{ is not isoperimetric}\}.
\end{equation*}
Here, $A_{0}$ is chosen large enough so that Proposition \ref{prop:behavior-large-iso} and Corollary \ref{coro:coarse-bds-general-large-iso} apply. First of all, note that $\sA$ is an open subset of $\RR$, because the isoperimetric profile and $\sL^{3}_{g}(\cB_{g}(A))$ are both continuous functions. Furthermore, Proposition \ref{prop:some-balls-are-isop} shows that there exists an unbounded sequence in $\sA^{c}$, i.e., a divergent sequence of areas $A$ so that $\cB_{g}(A)$ is isoperimetric. 

Thus, $\sA$ is the union of a sequence of bounded open intervals. We claim that $\sA$ is empty, as long as we increase $A_{0}$ if necessary. If $\cA$ is not empty, there is some interval $(A_{1},A_{2}) \subset \sA$. We may assume that $A_{1},A_{2}\not \in \sA$ and $A_{1} > A_{0}$. Geometrically, what this means is that:
\begin{enumerate}
\item the regions $\cB_{g}(A_{1})$ and $\cB_{g}(A_{2})$ are isoperimetric, and
\item for $A \in (A_{1},A_{2})$, we have the strict inequality $V_{g}(A) > \sL_{g}^{3}(\cB_{g}(A))$ (with equality at the endpoints).
\end{enumerate}
As a consequence of this, we see that 
\begin{equation*}
\frac{d}{dt} \Big|_{+}V_{g}(A_{1}) = \frac{d}{d A}  \sL_{g}^{3}(\cB_{g}(A))\Big|_{A_{1}}
\end{equation*}
and
\begin{equation*}
\frac{d}{dt} \Big|_{-} V_{g}(A_{2}) = \frac{d}{d A}  \sL_{g}^{3}(\cB_{g}(A))\Big|_{A_{2}}.
\end{equation*}
It is important to obtain a good estimate for the quantity on the right hand side of these equations.
\begin{lemm}\label{lemm:deriv-NT-area}
For $A$ large enough so that the coordinate sphere $\cB_{g}(A)$ lies entirely in the unperturbed region,
\begin{equation*}
\left(\frac{d}{dA} \sL^{3}_{g}(\cB_{g}(A))\right)^{-2} = 4 + 16 \pi A^{-1}  - 64 \pi^{\frac32} \mm A^{-\frac 32} 
\end{equation*}
\end{lemm}

\begin{proof}
Let $\rho_{A}$ denote the lapse function of the foliation $\partial \cB_{g}(A)$. In particular, we have that 
\begin{equation*}
1 = \int_{\partial \cB_{g}(A)} H_{A} \rho_{A} d\cH^{2}_{g} = H_{A} \int_{\partial \cB_{g}(A) } \rho_{A} d\cH^{2}_{g}.
\end{equation*}
Thus,
\begin{align*}
\left(\frac{d}{dA} \sL^{3}_{g}(\cB_{g}(A)) \right)^{-2}& = \left( \int_{\partial \cB_{g}(A) } \rho_{A} d\cH^{2}_{g} \right)^{-2}\\
& = H_{A}^{2} \\
& = 4  +16\pi A^{-1}-  (16\pi)^{\frac 32}m_{H}(\partial \cB_{g}(A))A^{-\frac 32}\\
& = 4  +16\pi A^{-1}-  (16\pi)^{\frac 32}\mm A^{-\frac 32}.\qedhere
\end{align*}
\end{proof}

Now, we ``integrate'' the differential inequality in Lemma \ref{lem:large-bad-iso-reg-bds-profile} from $A_{1}$ to $A_{2}$. To justify this, we may use the argument in Proposition \ref{prop:some-balls-are-isop} to show that the differential inequality holds in the distributional sense in the region $(A_{1},A_{2})$. Suppose that $\epsilon > 0$ is chosen so that $A_{1}+\epsilon,A_{2}-\epsilon$ are points of differentiability of $V_{g}(A)$ and Lebesuge points of $V'_{g}(A)$ (note that for $\epsilon_{0}>0$ small enough, a.e.\ $\epsilon \in (0,\epsilon_{0})$ will have this property). Then by taking a test function similar to before, we may conclude that
\begin{equation*}
- V'_{g}(A_{2}-\epsilon)^{-2} + V'_{g}(A_{1}+\epsilon)^{-2} \geq 24\pi \left( (A_{1}+\epsilon)^{-1} - (A_{2}-\epsilon)^{-1} \right).
\end{equation*}
By convexity\footnote{The fact that the (left and right) derivatives of a convex function are increasing (with no regularity assumptions) is classical fact (due to O.\ Stolz), cf.\ \cite[Theorem 1.3.3]{NiculescuPersson:cvxfunct}.} (cf.\ Remark \ref{rema:V-cvx}), $V'_{g}(A_{1}+\epsilon) \geq \frac{d}{dt}|_{+}V_{g}(A_{1})$ and $V'_{g}(A_{2}-\epsilon) \leq  \frac{d}{dt}|_{-}V_{g}(A_{1})$. Choosing a sequence of $\epsilon$ tending to zero and so that the previous argument applies, we may conclude that
\begin{equation*}
- \left(\frac{d}{dt} \Big|_{-} V_{g}(A_{2}) \right)^{-2} + \left( \frac{d}{dt} \Big|_{+}V_{g}(A_{1}) \right)^{-2} \geq 24\pi \left( A_{1}^{-1} - A_{2}^{-1} \right).
\end{equation*}
Combined with the above formula, this yields
\begin{equation*}
\left(\frac{d}{dA} \sL_{g}^{3}(\cB_{g}(A))\Big|_{A_{1}}\right)^{-2} - \left(\frac{d}{dA} \sL_{g}^{3}(\cB_{g}(A))\Big|_{A_{2}}\right)^{-2}\geq 24\pi \left( A_{1}^{-1} - A_{2}^{-1} \right).
\end{equation*}
We may use Lemma \ref{lemm:deriv-NT-area} to evaluate the left hand side of this expression as
\begin{align*}
& \left(\frac{d}{dA} \sL_{g}^{3}(\cB_{g}(A))\Big|_{A_{1}}\right)^{-2} - \left(\frac{d}{dA} \sL_{g}^{3}(\cB_{g}(A))\Big|_{A_{2}}\right)^{-2} \\
& =  16\pi \left( A_{1}^{-1} - A_{2}^{-1} \right) - 64 \pi^{\frac 32}\mm \left( A_{1}^{-\frac 32} - A_{2}^{-\frac 32} \right) .
\end{align*}
Thus, we see that 
\begin{equation*}
- 64 \pi^{\frac 32} \mm \left( A_{1}^{-\frac 32} - A_{2}^{-\frac 32} \right)  \geq 8\pi \left( A_{1}^{-1} - A_{2}^{-1} \right) .
\end{equation*}
Equivalently, we may rewrite this as
\begin{equation*}
- 64 \pi^{\frac 32}\mm (A_{1}+A_{1}^{\frac 12}A_{2}^{\frac 12}+A_{2})   \geq 8\pi (A_{1}A_{2}^{\frac 12}+A_{2}A_{1}^{\frac 12}  ).
\end{equation*}
This is a contradiction. Thus, we have proven that for large $A$, the regions $\cB_{g}(A)$ are isoperimetric.

Finally we claim that the regions $\cB_{g}(A)$ are uniquely isoperimetric for large enough $A$. The fact that $\cB_{g}(A)$ is isoperimetric implies that 
\begin{equation*}
2 V''_{g}(A)V'_{g}(A) ^{-3} A^{2}= 16 \pi -\frac 32 (16\pi)^{\frac 32}A^{-\frac 12} \mm.
\end{equation*}
It is clear that holds in the classical sense (not just in a barrier sense) because $\partial \cB_{g}(A)$ forms a $C^{\infty}$ foliation of the exterior region. On the other hand, if there was another isoperimetric region, then by the argument in Proposition \ref{prop:some-balls-are-isop} we would also have 
\begin{equation*}
2 V''_{g}(A)V'_{g}(A) ^{-3} A^{2}\geq 24 \pi -\frac 32 (16\pi)^{\frac 32} A^{-\frac 12} \mm + o(1),
\end{equation*}
in the barrier sense at $A$. Clearly, these two equations cannot both hold. This completes the proof of Theorem \ref{theo:main-theo}.

\section[On the necessity of the scalar curvature lower bounds]{On the assumption $R_{g}\geq -6$ in Theorem \ref{theo:main-theo}}\label{sec:main-theo-is-sharp}

In this section, we show that the assumption $R_{g}\geq -6$ in Theorem \ref{theo:main-theo} may not be removed. More precisely, we show that
\begin{theo}\label{theo:main-thm-is-sharp}
For $\mm>0$, there exists a function $\varphi(r)$ so that the metric $g:= dr\otimes dr + \varphi(r)^{2}g_{\SS^{2}}$ defined on $M:=(r_{0},\infty)\times \SS^{2}$ has the following properties: 
\begin{enumerate}
\item $(M,g)$ is a compact perturbation of Schwarzschild-AdS of mass $\mm>0$. In particular $\{r_{0}\}\times \SS^{2}$ is an outermost $H_{g}\equiv 2$, CMC surface. 
\item $(M,g)$ does not have $R_{g}\geq -6$ everywhere.
\item For sufficiently large $A$, the ball $\cB_{g}(A)$ is not isoperimetric in $(M,g)$. 
\end{enumerate}
\end{theo}

\begin{proof}
We fix constants $r_{0}>0$ and $\epsilon>0$ to be specified subsequently. Let $\varphi_{\mm}(r)$ denote the function so that $\overline g_{\mm} = dr\otimes dr + \varphi_{\mm}(r)^{2}g_{\SS^{2}}$. We will define $g:=dr\otimes dr + \varphi(r)^{2}g_{\SS^{2}}$ to be a rotationally symmetric metric on $[r_{0},\infty)\times \SS^{2}$. Note that the mean curvature of $\{r\}\times \SS^{2}$ with respect to $g$ is given by
\begin{equation*}
H_{g}(r) = \frac{2\varphi'(r)}{\varphi(r)},
\end{equation*}
and similarly for $H_{\overline g_{\mm}}(r)$. If we have specified $H_{g}(r)$, then observe that we may integrate the ODE to obtain
\begin{equation*}
\varphi(r) = \varphi(r_{0}) e^{\frac 12 \int_{r_{0}}^{r}H_{g}(\tau) d\tau}.
\end{equation*}
We may find a smooth function $H_{g}(r)$ with the property that $H_{g}(r_{0}) = 2$, $H_{g}(r) > 2$ for $r>r_{0}$, 
\begin{equation*}
\epsilon e^{\frac 12 \int_{r_{0}}^{r} H_{g}(\tau) d\tau } = \varphi_{\mm}(r),
\end{equation*}
for $r > r_{0}+1$,
and
\begin{equation*}
\epsilon e^{\frac 12 \int_{r_{0}}^{r} H_{g}(\tau) d\tau } \leq \varphi_{\mm}(r)
\end{equation*}
for $r \in (r_{0},r_{0}+1)$. Such an $H_{g}(r)$ will start at $2$ when $r=r_{0}$ and then grow to be very large, and then decrease back to agree with $H_{\overline g_{\mm}}(r)$ near $r_{0}+1$. The large bump will allow it $\varphi(r)$ to grow rapidly so that it agrees with $\varphi_{\mm}(r)$ by $r_{0}+1$. 

As such, we set
\begin{equation*}
\varphi(r) := \epsilon e^{\frac 12 \int_{r_{0}}^{r} H_{g}(\tau) d\tau }
\end{equation*}
for all $r \geq r_{0}$, and claim that the metric $g$ satisfies the properties asserted in the theorem. First, note that because the mean curvature of $\{r\}\times \SS^{2}$ is larger than $2$, the maximum principle forbids any compact surfaces with $H_{g}\equiv 2$ in $(M,g)$. Hence, $(M,g)$ is a compact perturbation of Schwarzschild-AdS, as in Definition \ref{def:cpt-pert-SADS}. For $C>0$, by choosing $r_{0} >0$ large and $\epsilon>0$ small, we may ensure that $(M,g)$ does not satisfy the conclusion of Proposition \ref{prop:vol-comp-H2}, namely 
\begin{equation}\label{eq:vol-comp-fail-no-R}
V(M,g) + \frac 12 A_{\partial M} < -C.
\end{equation}
To check this, note that as $r_{0}$ becomes large, the contribution to $V(\overline M_{\mm},\overline g_{\mm})$ outside of radius $r_{0}+1$ becomes negligible (this follows from $V(\overline M_{\mm},\overline g_{\mm})<\infty$). Let us denote this contribution by $V(\overline M_{\mm},\overline g_{\mm})_{out}$. Then, we have that
\begin{equation*}
V(M,g) =V(\overline M_{\mm},\overline g_{\mm})_{out} +  \sL^{3}_{g}((r_{0},r_{0}+1)\times \SS^{2}) - \sL^{3}_{\overline g}(\cB_{\overline g}(4\pi \sinh^{2}(r_{0}+1))).
\end{equation*}
Recall that $\cB_{\overline g}(4\pi \sinh^{2}(r_{0}+1))$ is the ball in hyperbolic space of surface area $4\pi \sinh^{2}(r_{0}+1)$. 

Because we have arranged that $\varphi(r)\leq \varphi_{\mm}(r)$, the easily checked fact that $\varphi_{\mm}(r)^{2} \leq \sinh^{2} r + o(1)$ shows that we may bound
\begin{align*}
\sL_{g}^{3}((r_{0},r_{0}+1)\times \SS^{2}) & \leq \sL_{\overline g_{\mm}}^{3}((r_{0},r_{0}+1)\times \SS^{2})\\
& \leq \int_{r_{0}}^{r_{0}+1} \left( 4\pi \sinh^{2}\tau + o(1) \right) d\tau\\
& \leq \pi \sinh(2 r_{0} + 2) - \pi \sinh(2 r_{0}) +O(1)\\
& \leq \frac \pi 2 e^{2r_{0}+2}-\frac \pi 2 e^{2r_{0}} + O(1),
\end{align*}
as $r_{0}$ becomes large. 
On the other hand, Lemma \ref{lemma:vol-large-balls} implies that
\begin{align*}
\sL^{3}_{\overline g}(\cB_{\overline g}(4\pi \sinh^{2}(r_{0}+1))) & = 2\pi \sinh^{2}(r_{0}+1) - \pi \log (2\pi(r_{0}+1)) + O(1)\\
& = \frac \pi 2 e^{2r_{0}+ 2} - \pi \log (2\pi (r_{0}+1)) + O(1).
\end{align*}
Putting this together, we have that $V(M,g)$ becomes very negative as $r_{0}$ becomes large. Taking $\epsilon$ small and $r_{0}$ large, \eqref{eq:vol-comp-fail-no-R} follows. 

The condition \eqref{eq:vol-comp-fail-no-R} implies the theorem. To see this, choose a sequence of $A_{i}\to \infty$ and consider the centered balls $\cB_{g}(A_{i})$ (recall that $\cH^{2}_{g}(\partial\cB_{g}(A_{i}))=A_{i}$). Then \eqref{eq:vol-comp-fail-no-R} implies that for all large $i$,
\begin{equation*}
\sL_{g}(\cB_{g}(A_{i})) + \frac 12 A_{\partial M} + C < \sL_{ \overline g}(\cB_{\overline g}(A_{i})) \leq \frac 12 A_{i} - \pi \log A_{i} + \pi(1+\log\pi) + o(1). 
\end{equation*}
as $i\to \infty$. Here, we have used the expression derived in Lemma \ref{lemma:vol-large-balls}. 
As such,
\begin{equation*}
\sL_{g}(\cB_{g}(A_{i})) < \frac 12 (A_{i}-A_{\partial M}) - \pi \log (A_{i}- A_{\partial M}) + \pi(1+\log\pi) + o(1)
\end{equation*}
as $i\to\infty$. 

This shows that $\cB_{g}(A_{i})$ is not isoperimetric, as it contains \emph{less} volume than the generalized isoperimetric region consisting of $\Omega$ which is equal to the horizon region (and hence has zero $g$-volume, and $g$-area $A_{\partial M}$) along with a ball in hyperbolic space of surface area $A_{i}-A_{\partial M}$. 
\end{proof} 

\appendix

\section{Volume contained in coordinate balls} \label{sec:vol-coord-balls} For $A$ large enough, we write $\cB_{\overline g_{m}}(A)$ for the centered coordinate ball in $(\overline M_{m},\overline g_{m})$ of surface area $A$. Regarded as a set in $\RR^{3}$ (using the coordinate system as in \eqref{eq:schwAdS}) we will always regard $\cB_{\overline g_{m}}(A)$ as containing the horizon, i.e., a set of the form $\{s\leq s(m,A)\}$ for some $s(m,A)$. 
\begin{lemm}\label{lemma:vol-large-balls}
For $m\geq 0$ and for all $A$ large enough so that $\cB_{\overline g_{m}}(A)$ is defined, we have that
\begin{equation*}
\sL^{3}_{\overline g_{m}}(\cB_{\overline g_{m}}(A)) =  \sL_{\overline g}^{3}(\cB_{\overline g}(A)) + V(\overline M_{m},\overline g_{m}) + O(A^{-\frac 12})
\end{equation*}
as $A \to \infty$. More precisely, we have the expansion
\begin{align*}
\sL^{3}_{\overline g_{m}}(\cB_{\overline g_{m}}(A)) = \frac 12 A & - \pi \log A +  \left( V(\overline M_{m},\overline g_{m}) + \pi(1+\log \pi) \right) \\
& - 8 \pi^{\frac 32} m A^{-\frac 12} - 3\pi^{2} A^{-1} + 16 \pi^{\frac 5 2} m A^{-\frac 32} + O(A^{-2})
\end{align*}
where $V(\overline M_{m},\overline g_{m})$ is the renormalized volume of $(\overline M_{m},\overline g_{m})$. This expression holds as $A \to \infty$ for $m$ fixed. 

We additionally have
\begin{align*}
\sL^{3}_{\overline g_{m}}(\cB_{\overline g_{m}}(A)) = \frac 12 A - \pi \log A + & \left( V(\overline M_{m},\overline g_{m}) + \pi(1+\log \pi) \right) \\
& - 8 \pi^{\frac 32} m A^{-\frac 12} + E(m,A),
\end{align*}
where, if $0 \leq m \leq \alpha A^{\frac 12}$, the error $E(m,A)$ satisfies
\begin{equation*}
|E(m,A)| \leq C A^{-1},
\end{equation*}
for $C=C(\alpha)$ independent of $m$ or $A$. 

 Finally, for all such $A$, we have the inequality
\begin{equation*}
\sL^{3}_{\overline g_{m}}(\cB_{\overline g_{m}}(A)) \leq  \sL_{\overline g}^{3}(\cB_{\overline g}(A)) + V(\overline M_{m},\overline g_{m}) .
\end{equation*}
\end{lemm}
\begin{proof}
Choose $R$ so that the sphere $\{s=R\}$ has area $A$, i.e., $4\pi R^{2} = A$. Then, we have that
\begin{align*}
\sL_{\overline g_{m}}^{3}(\cB_{\overline g_{m}}(A)) & = 4\pi \int_{2m}^{R} \frac{s^{2}}{\sqrt{1+s^{2}-2ms^{-1}}} ds \\
& = 4\pi \int_{0}^{R} \frac{s^{2}}{\sqrt{1+s^{2}}} ds \\
&\qquad  + 4\pi \int_{2m}^{R} \frac{s^{2}}{\sqrt{1+s^{2}-2ms^{-1}}} ds - 4\pi \int_{0}^{R} \frac{s^{2}}{\sqrt{1+s^{2}}}ds \\
& = \sL_{\overline g}^{3}(\cB_{\overline g}(A)) + V(\overline M_{m},\overline g_{m})\\
& \qquad  -  4\pi \int_{R}^{\infty} s^{2} \left( \frac{1}{\sqrt{1+s^{2}-2ms^{-1}}} -\frac{1}{\sqrt{1+s^{2}}} \right) ds.
\end{align*}
The inequality claimed in the end of the lemma follows immediately from this, because $m\geq 0$. To verify the asymptotic expansion, we evaluate
\begin{align*}
\sL^{3}_{\overline g}(\cB_{\overline g}(A)) & = 4\pi \int_{0}^{R} \frac{s^{2}}{\sqrt{1+s^{2}}} ds\\
& = 2\pi R^{2} \sqrt{1+R^{-2}} - 2\pi \sinh^{-1}(R) \\
& = 2\pi R^{2} -2\pi \log R +\pi(1-\log 4) - \frac{3\pi}{4} R^{-2} + O(R^{-4}),
\end{align*}
and 
\begin{align*}
 & 4\pi \int_{R}^{\infty} s^{2} \left( \frac{1}{\sqrt{1+s^{2}-2ms^{-1}}} -  \frac{1}{\sqrt{1+s^{2}}} \right) ds \\
 & = 4\pi \int_{R}^{\infty} \left( m s^{-2} - \frac{3m}{2} s^{-4} + O(s^{-5})\right) ds\\
 & = 4\pi m R^{-1} - 2 \pi m R^{-3} + O(R^{-4}).
\end{align*}
From this, the first series follows by combining these expansions with the relation $A = 4\pi R^{2}$. 

To analyze the possibility that $m$ is growing large with $A$, but satisfies $0\leq m\leq \alpha A^{\frac 12}$, note that
\begin{align*}
 m s^{-2} & + \frac{s^{2}}{\sqrt{1+s^{2}}} - \frac{s^{2}}{\sqrt{1+s^{2}-2ms^{-1}}} \\
  & = ms^{-2} + \frac{s^{2}}{\sqrt{1+s^{2}}}\left( \frac{ \sqrt{1-\frac{2m}{s(1+s^{2})}} - 1}{ \sqrt{1-\frac{2m}{s(1+s^{2})} }} \right)\\
  & = ms^{-2} - \frac{2ms^{-2}}{(1+s^{-2})^{\frac 32}} \left( 1-\frac{2m}{s(1+s^{2})} \right)^{-\frac 12}\left( 1 + \sqrt{1-\frac{2m}{s(1+s^{2})}} \right)^{-1}.
\end{align*}
Because we are going to integrate this expression (in $s$) from $R$ to $\infty$, we are only concerned for $s$ which satisfy $m \leq \alpha (4\pi)^{\frac 12} s$. In this range, we have that
\begin{equation*}
\frac 12  \leq \left( 1 + \sqrt{1-\frac{2m}{s(1+s^{2})}}\right)^{-1}  \leq \left( 1 + \sqrt{1-\frac{2\alpha(4\pi)^{\frac 12}}{1+s^{2}}}\right)^{-1} \leq \frac 12 + C s^{-2},
\end{equation*}
where $C=C(\alpha)$ is independent of $m$, $R$ and $s$. Similarly, taking $C$ larger if necessary (but still not letting it depend on $m$, $R$ or $s$) we have 
\begin{equation*}
1 \leq  \left( 1-\frac{2m}{s(1+s^{2})} \right)^{-\frac 12} \leq 1 + C s^{-2}.
\end{equation*}
Putting this together, we see that for $m \leq \alpha (4\pi)^{\frac 12} s$, there is a constant $C=C(\alpha)$ independent of $m,R,s$ so that
\begin{equation*}
\left| m s^{-2}  + \frac{s^{2}}{\sqrt{1+s^{2}}} - \frac{s^{2}}{\sqrt{1+s^{2}-2ms^{-1}}}\right| \leq C s^{-3}. 
\end{equation*}
Because the left hand side is integrated with respect to $s$ from $R$ to $\infty$ to obtain $E(m,A)$, we obtain the desired bound. 
\end{proof}

Similarly, we may compute the volume of large, centered coordinate balls in a compact perturbation of Schwarzschild-AdS $(M,g)$ as follows
\begin{lemm}\label{lemm:vol-large-coord-balls-g}
Let $(M,g)$ be a compact perturbation of Schwarzschild-AdS of mass $\mm \geq 0$. For $A>0$ sufficiently large, the coordinate sphere $\cB_{g}(A)$ of area $A$ completely contains the perturbed region $\tilde K$, and we have 
\begin{align*}
\sL^{3}_{g}(\cB_{g}(A)) = \frac 12 A - \pi \log A + & \left( V( M,g) + \pi(1+\log \pi) \right) \\
& - 8 \pi^{\frac 32}  \mm A^{-\frac 12} - 3 \pi^{2} A^{-1}+ O(A^{-\frac 32}),
\end{align*}
as $A \to \infty$. 
\end{lemm}

\section{Proof of Proposition \ref{prop:comp-bdry-bds}}\label{app:comp-bdry-bds}

Here we prove Proposition \ref{prop:comp-bdry-bds} which gives an upper bound for the number of components of an isoperimetric region in $(M,g)$ an asymptotically hyperbolic manifold with $R_{g}\geq -6$. We will first prove several preliminary results. 

We note that the reader who is only interested in the statement of Proposition \ref{prop:comp-bdry-bds} for compact perturbations of Schwarzschild-AdS may observe that in this case, only Lemma \ref{lemm:H-bd-vp-stable} is necessary for the proof---the rest of the preliminary results needed in the proof of Proposition \ref{prop:comp-bdry-bds} may be replaced by  a straightforward application of Theorem \ref{theo:brendle-IHES}.

\begin{lemm}[cf. {\cite[Proposition 5.1]{EichmairMetzger:CMC}}]\label{lemm:H-bd-vp-stable}
For $(M,g)$ a Riemannian three manifold with $R_{g}\geq -6$ and $\Sigma$ a closed volume-preserving stable CMC surface, which is not necessarily connected, the mean curvature of $\Sigma$ satisfies
\begin{equation*}
H_{g}^{2}\leq \max\left\{ -2\inf_{\Sigma} \Ric(\nu,\nu), \frac{64\pi}{3}\cH_{g}^{2}(\Sigma)^{-1} +4 \right\} .
\end{equation*}
\end{lemm}
\begin{proof}
If $0< |\sff|^{2} +\Ric(\nu,\nu)$ along $\Sigma$, then $\Sigma$ is connected. If it were not, then taking a volume-preserving variation which is a positive constant on one component and a corresponding negative constant on another component would yield a contradiction. If $\Sigma$ is connected, we may rearrange Proposition \ref{prop:christodoulou-yau} to bound the mean curvature as claimed.
\end{proof}

Recall that for a hypersurface $\Sigma$ in $\RR^{3}$, we have defined the \emph{inner radius} of $\Sigma$ by
\begin{equation*}
\underline s(\Sigma) := \inf \left\{ s(x) : x \in \Sigma \right\},
\end{equation*}
where the coordinate $s$ is the one used in \eqref{eq:schwAdS}. The next lemma follows from a straightforward computation.
\begin{lemm}\label{lemm:sff-g-gbar}
If $(M,g)$ is asymptotically hyperbolic, then there is some constant $s_{0}>0$ depending only on $(M,g)$ with the following property: suppose that $\Sigma$ is a hypersurface in $(M,g)$ with $\underline s(\Sigma) \geq s_{0}$. Then, the second fundamental form of $\Sigma$ when measured with respect to $g$, $\sff_{g}$, and measured with respect to $\overline g$, $h_{\overline g}$ satisfy
\begin{equation*}
|\sff_{g} - \sff_{\overline g}|_{\overline g} \leq O(s^{-3})\left( |\sff_{\overline g}|_{\overline g}+ 1\right).
\end{equation*}
Furthermore, the mean curvatures also satisfy
\begin{equation*}
|H_{g} - H_{\overline g}| \leq  O(s^{-3})\left( |\sff_{\overline g}|_{\overline g}+ 1\right).
\end{equation*}
\end{lemm}

Furthermore, we have the following integral decay estimate.
\begin{lemm}[cf. {\cite[Proposition 4.2]{NevesTian:AHcmc1}}]\label{lemm:int-sminus3-bd}
If $(M,g)$ is asymptotically hyperbolic with $R_{g}\geq -6$ and $\Sigma$ is a closed, connected, volume-preserving stable CMC surface in $(M,g)$, we have that 
\begin{equation*}
\int_{\Sigma} s^{-3}d\cH^{2}_{g} = o(1)
\end{equation*}
as $\underline s(\Sigma) \to \infty$. 
\end{lemm}
\begin{proof}
We define a function $r$ on $(\overline M,\overline g)$ by $s=\sinh r$, where $s$ is the coordinate in \eqref{eq:schwAdS}. Notice that in these coordinates, the hyperbolic metric becomes
\begin{equation*}
\overline g = dr\otimes dr + \sinh^{2}r g_{\SS^{2}}
\end{equation*}
and the asymptotically hyperbolic condition on $g$ means that $g$ (and two covariant derivatives) differs from $\overline g$ by terms of order $O(e^{-3r})$. 
\begin{align*}
\Div_{\Sigma,g}(\partial_{r}) & = (1 + g(\nu,\partial_{r})^{2}) \frac{\cosh r}{\sinh r} + O(e^{-3r})\\
& = (1 + g(\nu,\partial_{r})^{2})(1 + 2 e^{-2r})+ O(e^{-3r}).
\end{align*}
Integrating this yields, via the first variation formula
\begin{align*}
 \int_{\Sigma } & (1 + g(\partial_{r},\nu)^{2})(1 + 2 e^{-2r}) d\cH^{2}_{g} + \int_{\Sigma} O(e^{-3r}) d\cH^{2}_{g}\\
  & =  \int_{\Sigma} \Div_{\Sigma,g}(\partial_{r})d\cH^{2}_{g}\\
 & = \int_{\Sigma} H_{g} g(\partial_{r},\nu) d\cH^{2}_{g} \\
  & = \int_{\Sigma} (H_{g}-2) d\cH^{2}_{g} - \int_{\Sigma} (H_{g}-2)(1-g(\partial_{r},\nu)) d\cH^{2}_{g} + 2 \int_{\Sigma} g(\partial_{r},\nu) d\cH^{2}_{g}.
\end{align*}
We may rearrange this for $\underline s(\Sigma)$ sufficiently large (using the outermost assumption to see that $H_{g} > 2$) to yield
\begin{align*}
\int_{\Sigma} (1-g(\partial_{r},\nu))^{2}d\cH^{2}_{g} + 2\int_{\Sigma}e^{-2r} d\cH^{2}_{g} & \leq \cH^{2}_{g}(\Sigma)(H_{g}-2).
\end{align*}
By Proposition \ref{prop:christodoulou-yau}, we have that
\begin{equation*}
H_{g}^{2}  \leq 4 + \frac{64\pi}{3\cH^{2}_{g}(\Sigma)},
\end{equation*}
and it is easy to see that $\cH^{2}_{g}(\Sigma)\to\infty$ as $\underline s(\Sigma) \to \infty$. From this, we may conclude that for $\underline s(\Sigma)$ sufficiently large, we have the bound
\begin{equation*}
\int_{\Sigma} e^{-2r} d\cH_{g}^{2} \leq \frac{32\pi}{3},
\end{equation*}
from which the claim follows.
\end{proof}

\begin{lemm}[cf. {\cite[Proposition 5.2]{EichmairMetzger:CMC}}]
For $(M,g)$ an asymptotically hyperbolic manifold, if $\Sigma$ is a closed surface in $(M,g)$, then $\int_{\Sigma} (|h|^{2}-2)d\cH^{2}_{g}\geq 8\pi -o(1)$ as $\underline s(\Sigma) \to \infty$.
\end{lemm}
\begin{proof}
The Gau\ss\ equations yield
\begin{equation*}
\int_{ \Sigma}  |\tfsff_{\overline g}|_{\overline g}^{2} d\cH^{2}_{\overline g} - 4\pi \chi(\Sigma)  = \int_{ \Sigma} \left( \frac 12 H_{\overline g}^{2} + R_{\overline g} - 2 \Ric(\nu,\nu) \right) d\cH^{2}_{\overline g} 
\end{equation*}
This implies that the left hand side is conformally invariant. Because hyperbolic space is conformally Euclidean, we may thus apply \cite[(16.32)]{GT} to see that
\begin{equation*}
 \frac 12 \int_{ \Sigma} \left( H_{\overline g}^{2} - 4\right) d\cH^{2}_{\overline g}= \frac 12 \int_{ \Sigma} H^{2}_{\delta} d\cH^{2}_{\delta} \geq 8\pi.
\end{equation*}
As such, we see that 
\begin{equation*}
\int_{ \Sigma} \left( |\sff_{\overline g}|_{\overline g}^{2} - 2\right)d\cH^{2}_{\overline g} \geq 8\pi.
\end{equation*}
We compute
\begin{align*}
 & \int_{\Sigma} |\sff_{g}|_{g}^{2} (1+O(s^{-3}))d\cH^{2}_{g}\\
  & \geq \int_{\Sigma} |\sff_{g}|_{\overline g}^{2} d\cH^{2}_{\overline g}\\
 & \geq \int_{\Sigma} |\sff_{\overline g}|_{\overline g}^{2} d\cH^{2}_{\overline g} + \int_{\Sigma}(|\sff_{g}|_{\overline g}-|\sff_{\overline g}|_{\overline g})(|\sff_{g}|_{\overline g}+|\sff_{\overline g}|_{\overline g})d\cH^{2}_{\overline g}\\
  & \geq \int_{\Sigma} |\sff_{\overline g}|_{\overline g}^{2}d\cH^{2}_{\overline g} - C \int_{\Sigma} s^{-3} (|\sff_{\overline g}|_{\overline g} + 1) (|\sff_{\overline g}|_{\overline g} + C s^{-3})d\cH^{2}_{\overline g}\\
    & \geq  \int_{\Sigma} |\sff_{\overline g}|_{\overline g}^{2} (1-O(s^{-3})) d\cH^{2}_{\overline g} - o(1),
\end{align*}
as $\underline s(\Sigma)\to\infty$. As such
\begin{align*}
& \int_{\Sigma}( |\sff_{g}|_{g}^{2} -2) (1+O(s^{-3}))d\cH^{2}_{g}  + 2\int_{\Sigma}(1+O(s^{-3})) d\cH^{2}_{\overline g}\\
& \geq    \int_{\Sigma} (|\sff_{\overline g}|_{\overline g}^{2}-2) (1-O(s^{-3})) d\cH^{2}_{\overline g} + 2 \int_{\Sigma} (1-O(s^{-3}))d\cH^{2}_{\overline g} - o(1),
\end{align*}
which allows us to finish the proof using the previous lemma. 
\end{proof}

\begin{lemm}[cf. {\cite[Proposition 5.3]{EichmairMetzger:CMC}}]
For $(M,g)$ an asymptotically hyperbolic manifold with $R_{g}\geq -6$, there exists a coordinate ball $B$ so that any closed, volume-preserving stable CMC surface $\Sigma\hookrightarrow (M,g)$ has at most one component $\Sigma'$ with $\Sigma'\cap B =\emptyset$. 
\end{lemm}
\begin{proof}
Assume that $\Sigma'$, $\Sigma''$ two components of a closed, volume-preserving stable CMC surface which are both disjoint from some large coordinate ball $B$ (to be chosen below). We assume that $\cH^{2}_{g}(\Sigma') \leq \cH^{2}_{g}(\Sigma'')$. Then, choose the function $u$ which is $\cH^{2}_{g}(\Sigma'')$ on $\Sigma'$ and $-\cH^{2}_{g}(\Sigma')$ on $\Sigma''$ is volume-preserving. Hence, we have that, using $\Ric(\nu,\nu) + 2 = O(s^{-3})$ and Lemma \ref{lemm:int-sminus3-bd}
\begin{align*}
0 & \geq \int_{\Sigma'}(|\sff|^{2}+\Ric(\nu,\nu)) d\cH^{2}_{g} +  \frac{\cH^{2}_{g}(\Sigma')}{\cH^{2}_{g}(\Sigma'')}\int_{\Sigma''}(|\sff|^{2}+\Ric(\nu,\nu)) d\cH^{2}_{g}\\
& \geq \int_{\Sigma'}(|\sff|^{2}-2)d\cH^{2}_{g} - \int_{\Sigma'\cup\Sigma''}O(s^{-3})d\cH^{2}_{g}\\
& \geq 8\pi - o(1),
\end{align*}
as $\underline s(\Sigma'\cup\Sigma'') \to \infty$. Choosing $B$ large enough, we may ensure that this is a contradiction.
\end{proof}

 Now, we may prove the main result of this appendix, namely that an isoperimetric region in $(M,g)$, an asymptotically hyperbolic manifold with $R_{g}\geq -6$, has a uniformly bounded number of components.
\begin{proof}[Proof of Proposition \ref{prop:comp-bdry-bds}]
First, choose a coordinate ball $B$ large enough so that the previous lemma applies. 

By Lemma \ref{lemm:H-bd-vp-stable}, we may assume that the mean curvature of $\Sigma$ is uniformly bounded. Thus, by the monotonicity formula, the number of components of $\Sigma$ which intersect $B$ is bounded in terms of $\cH^{2}_{g}(\Sigma\cap B)$ (each component of $\Sigma \cap B$ contributes a guaranteed amount to $\cH^{2}_{g}(\Sigma\cap B)$ by combining the monotonicity  formula with the upper bound on the mean curvature). As such, it is sufficient to uniformly bound $\cH^{2}_{g}(\Sigma\cap B)$. Note that $\sL_{g}^{3}(\Omega) \leq \sL^{3}_{g}(\Omega \cup B)$, so by the isoperimetric property of $\Omega$ and Lemma \ref{lemm:iso-prof-strict-increase} 
\begin{align*}
\cH^{2}_{g}(\partial^{*}\Omega \cap B) + \cH^{2}_{g}(\partial^{*}\Omega \backslash B) & = \cH^{2}_{g}(\partial^{*}\Omega) \\ 
& \leq \cH^{2}_{g}(\partial^{*}( \Omega\cup B)) \\
& = \cH^{2}_{g}(\partial B \backslash \Omega)  + \cH^{2}_{g}(\partial^{*}\Omega\backslash B).
\end{align*}
As such, we have the uniform bound $\cH^{2}_{g}(\partial^{*}\Omega \cap B) \leq \cH^{2}_{g}(\partial B)$. From this, the assertion follows.
\end{proof}

\bibliography{bib} 
\bibliographystyle{amsalpha}

\end{document}